\documentclass[reqno,11pt]{amsart}
\usepackage{amsmath,amssymb,mathrsfs,amsthm,amsfonts}
\usepackage[inline]{enumitem}
\usepackage[usenames,dvipsnames]{xcolor}
\usepackage{hyperref}
\usepackage[percent]{overpic}
\usepackage{comment}
\usepackage{stmaryrd}
\usepackage{algorithm}
\usepackage{algorithmic}
\usepackage{pgfplots}
\usepackage{subcaption}
\usepackage{tikz,todonotes}
\usetikzlibrary{calc}
\usetikzlibrary{arrows.meta,backgrounds}
\usepackage{multirow,array,longtable,booktabs}
\usetikzlibrary{arrows}
\usepackage{bm}
\usepackage{graphicx}

\hypersetup{
	colorlinks=true, linkcolor=blue,
	citecolor=ForestGreen
}
\usepackage[paper=letterpaper,margin=1in]{geometry}
\DeclareMathAlphabet{\mathpzc}{OT1}{pzc}{m}{it}
\DeclareMathOperator*{\esssup}{ess\,sup}

\newtheorem{theorem}{Theorem}[section]

\newtheorem{lemma}[theorem]{Lemma}
\newtheorem{proposition}[theorem]{Proposition}
\newtheorem{corollary}[theorem]{Corollary}
\newtheorem{definition}[theorem]{Definition}

\newtheorem{remark}[theorem]{Remark}
\numberwithin{equation}{section}
\allowdisplaybreaks
\usepackage{acronym}

\acrodef{KPZ}{Kardar--Parisi--Zhang}
\acrodef{SHE}{Stochastic Heat Equation}
\acrodef{LDP}{Large Deviation Principle}

\newcommand{\be}{\begin{equation}}
\newcommand{\ee}{\end{equation}}

\makeatletter

\newcommand{\R}{\mathbb{R}} 

\newcommand{\e}{\varepsilon}

\newcommand{\ep}{\varepsilon}

\newcommand{\T}{\mathbb{T}}

\renewcommand{\L}{\mathcal{L}}

\newcommand{\nd}{\noindent}

\renewcommand{\hat}{\widehat}

\renewcommand{\bar}{\overline}

\usepackage{comment}

\newcounter{hyp}
\title[Homogenization of HJ equations in the Wasserstein space]{Quantitative homogenization of convex Hamilton--Jacobi equations in The Wasserstein space}

\author[Z.\ Ding]{Zhiyan Ding}
\address{Z.\ Ding,
	Department of Mathematics, University of Michigan,
	\newline\hphantom{\quad \ \ Z. Ding}
	530 Church St, Ann Arbor, MI 48109, USA}
\email{zyding@umich.edu}

\author[I.\ Ekren]{Ibrahim Ekren}
\thanks{I. Ekren is partially supported by the NSF grant DMS-2406240.}
\address{I.\ Ekren,
	Department of Mathematics, University of Michigan,
	\newline\hphantom{\quad \ \ I. Ekren}
	530 Church St, Ann Arbor, MI 48109, USA}
\email{iekren@umich.edu}

\author[Y. Han]{Yuxi Han}
\address{Y.\ Han,
	Department of Mathematics, Purdue University,
	\newline\hphantom{\quad \ \ Y. Han}
    150 N. University Street, West Lafayette, IN 47907, USA
	}
\email{han891@purdue.edu}

\author[A.\ Zitridis]{Antonios Zitridis}
\address{A.\ Zitridis,
	Department of Mathematics, University of Michigan,
	\newline\hphantom{\quad \ \ A. Zitridis}
	530 Church St, Ann Arbor, MI 48109, USA}
\email{zitridis@umich.edu}

\thanks{}

\begin{document}
\null
\vspace{-4mm}
    \maketitle

\vspace{-10mm}
\begin{abstract}
We study a homogenization problem for first-order Hamilton--Jacobi equations in the Wasserstein space with a convex Hamiltonian. We show that the solution $U^\e$, which is the value function of a mean field control problem, converges uniformly as $\e \to 0$ to the solution of a limiting Hamilton--Jacobi equation whose Hamiltonian is obtained through a suitable cell problem. Furthermore, we establish quantitative rates of convergence. Under general assumptions with multiscale dependence, we prove that the rate of convergence is $O(\sqrt{\e})$.  When the Hamiltonian depends only on the fast variable and the momentum, we establish the sharp convergence rate $O(\e)$. To the best of our knowledge,
this is the first quantitative convergence result extending the optimal rate for first-order Hamilton--Jacobi equations in finite dimensions to the Wasserstein space.
Finally, we show that our analysis extends to dynamic optimal transport problems, where the terminal condition imposes a constraint on the final distribution.
\end{abstract}

\tableofcontents

    \section{Introduction}
    \subsection{Motivation and setup of the problem}
    Let $\e,T>0$ and $d\in \mathbb{N}^*$. This paper is devoted to studying the behavior as $\e\to 0$ of the solution
    of the first-order Hamilton--Jacobi (HJ) equation in the Wasserstein space of Borel probability measures on $\R^d$ with finite second moment (denoted by $\mathcal{P}_2(\R^d)$):
    \be\label{HJBe}\tag{$\mathrm{HJB}_\e$}
    \begin{cases}
        -\partial_tU^\e +\int_{\R^d}H(\frac{x}{\e},x,D_mU^\e(t,m,x),m)m(dx)=0, &(t,m)\in [0,T)\times \mathcal{P}_2(\R^d),\\
        U^\e(T,m)= G(m), & m\in\mathcal{P}_2(\R^d),
    \end{cases}
    \ee
    where $H\colon \T^d\times \R^d\times \R^d\times\mathcal{P}_2(\R^d)\to \R$ is a given Hamiltonian and $G\colon \mathcal{P}_2(\R^d)\to \R$ is the terminal constraint. Here $\T^d=\R^d/\mathbb{Z}^d$ is the $d$-dimensional torus and the derivative
    $D_mU^\e\colon[0,T]\times\mathcal{P}_2(\R^d)\times \R^d\to \R$ of the function $U^\e\colon [0,T]\times\mathcal{P}_2(\R^d)\to \R$ is the Wasserstein derivative \footnote{We refer to \cite[Section 5]{carmona2018probabilistic}.}.
    \vspace{2mm}

    \nd
    Hamilton--Jacobi equations posed on spaces of probability measures have attracted considerable attention in recent years. Such equations arise in several areas including mean field games, mean field control, optimal transport and filtering. For a sample of recent works, we refer to \cite{cardaliaguet2019master} for mean field games and mean field control, \cite{bertucci2024stochastic} for optimal transport, \cite{martini2023kolmogorov,BayraktarEkrenHeZhang2025SC} for filtering.
    In particular, equations of the form \eqref{HJBe} (especially when the Hamiltonian $H$ does not have the $x/\e$ argument) arise as dynamic programming equations associated with mean field control problems. These problems study the collective behavior of a large population of interacting agents/particles who cooperate to minimize a common cost functional that might depend on the distribution of the population; see \cite{carmona2018probabilistic, cardaliaguet2019master} for surveys on the topic. More specifically in our setup, assuming that $p\mapsto H(x/\e,x,p,m)$ is convex, then under mild  assumptions on $H,\;G$, \eqref{HJBe}
    is the dynamic programming equation \cite{djete2022mckean1} satisfied by the value function $U^\e\colon [0,T]\times\mathcal{P}_2(\R^d)\to \R$ of the mean field control problem
    \begin{equation}\label{eqn:U_value_function}
        U^\e(t,m)=\inf_{\alpha\in\mathcal{A}}\left\{\mathbb{E}\left[\int^T_{t}L\left(\frac{X_s}{\varepsilon}, X_s, \alpha_s, \L\left(X_s\right)\right)ds \right]+G(\mathcal{L}(X_T))\right\},
    \end{equation}
    where we have considered a probability space $(\Omega,\mathcal{F},\mathbb{F}=(\mathcal{F}_t)_{t\ge 0},\mathbb{P})$ with no atoms that supports an $\mathbb{F}$-Brownian motion $B$, $\mathcal{L}(X_s)$ is the law of the state process $X$ at time $s\in [t,T]$, and the random variable $X_t$ is assumed to be independent of the Brownian motion $B$ and such that $\L(X_t)=m\in \mathcal{P}_2(\R^d)$. The control set $\mathcal{A}$ is \begin{equation}\label{eqn:cs}
    \mathcal{A}:=\left\{\alpha\colon [t,T]\times \Omega\to \R^d:\;\alpha\text{ is }\mathbb{F}\text{-progressively measurable and } \mathbb{E}\left[ \int_t^T|\alpha_s|^2ds\right]<\infty  \right\},
    \end{equation}
    the dynamics of the state process $X$ are determined from the control $\alpha$ via $dX_s=\alpha_sds$, and $L:\T^d\times \R^d\times \R^d\times\mathcal{P}_2(\R^d)\to \R$ is the Legendre transform of the Hamiltonian $H$:
    \be\label{Lagr}
    L(y,x,v,m) := \sup_{p \in \R^d} \big\{ -p \cdot v - H(y,x,p,m) \big\}.
    \ee
    The dependence of $H$ and $L$ on the fast variable $x/\e$ models the situation where the agents in the mean field control problem \eqref{eqn:U_value_function} evolve in an environment that exhibits rapid spatial variations.
    The parameter $\e>0$ is the scale of these variations. From an application point of view, this problem models a system of interacting agents/particles moving in a heterogeneous landscape, or having fluctuating cost.
    \vspace{2mm}

    \nd
    Our first goal in this paper is to establish the asymptotic behavior of the value function $U^\e$, defined in \eqref{eqn:U_value_function}, as $\e\to0$. Intuitively, when the scale of the heterogeneities becomes very small ($\e\to 0$), the oscillations of the Hamiltonian average out, and one expects $U^\e$ to converge to a function $U\colon [0,T]\times\mathcal{P}_2(\R^d)\to \R$ that can be described as the solution of a limiting HJ equation on the Wasserstein space with an effective Hamiltonian, in analogy with classical homogenization results for HJ equations in finite dimensions (see \cite{LionsPapaVara1987} and subsequent works, for instance, \cite{capuzzo2001rate,MitakeTranYu2019, TranYu2021}).
    In this work, we rigorously prove this convergence under general assumptions on the Hamiltonian $H$ and the terminal cost $G$. Furthermore, we establish a quantitative convergence rate for general multiscale settings and obtain the optimal rate in the non-multiscale setting. To our knowledge, this is the first quantitative result for HJ equations in the Wasserstein space
    and it extends previously known results for HJ equations in finite dimensions; we refer to \cite{tran_hamilton-jacobi_2021} for an overview of results in finite dimensions.
    \vspace{2mm}

    \nd
    As is the case in finite dimensions, HJ equations in $\mathcal{P}_2(\R^d)$ do not necessarily have classical solutions (we present an example in Remark \ref{rem1}(i) for \eqref{HJBe}), therefore a large number of recent works have been devoted to developing an appropriate notion of viscosity sub/super-solution and to establishing a comparison principle. We refer to \cite{conforti2023hamilton,daudin2024comparison,bayraktarqiu2025viscosity,bayraktar2025comparison,BayraktarEkrenHeZhang2025SC,cheung2025viscosity,daudin2025well,zhou2024viscosity, soner2024viscosity,mete2023viscosity, bertucci2024stochastic}, among others. Unlike in finite dimensions \cite{crandall1992user}, there are several reasonable definitions of viscosity solutions for equations posed in the Wasserstein space $\mathcal{P}_2(\R^d)$. Most relevant to the PDE \eqref{HJBe} that we are interested in are the works \cite{gangbo2019differentiability,gangbo2021finite}, which consider viscosity solutions in the ``lifted'' sense of P.L. Lions.
    In this paper,  we adopt this notion of viscosity solutions and we refer to them as  $L$-viscosity solutions, or simply as viscosity solutions if there is no confusion. We defer their precise definition to Section \ref{viscositysolutions} Definition~\ref{visc.}.
    \vspace{2mm}

    \nd
    The second goal of this paper is to show that the above homogenization result for \eqref{HJBe} implies a homogenization result for dynamic optimal transport or mean field planning problems. In this case, the terminal condition $G$ is allowed to take the value $+\infty$ and imposes a constraint on the final distribution of agents. In particular, suppose that $\nu\in \mathcal{P}_2(\R^d)$ and $G(m)=+\infty$ if $m\neq \nu$ and $0$ otherwise, in \eqref{eqn:U_value_function}. Then, $U^\e$ becomes a dynamic optimal transport problem $U_{ot}^\e\colon[0,T]\times\mathcal{P}_2(\R^d)\times\mathcal{P}_2(\R^d)\to \R\cup\{+\infty\}$ with
     \begin{equation}\label{eq:value-function}
        U_{ot}^\varepsilon(t,\mu,\nu)
        \;:=\;
        \inf_{\alpha\in \mathcal{A}_{t,\mu,\nu}}\;
        \mathbb E\Bigg[\int_t^T L\Big(\frac{X_s}{\varepsilon},X_s,\alpha_s,\mathcal{L}(X_s)\Big)\,ds\Bigg],
    \end{equation}
    where the associated state process $X=(X_s)_{s\in [t,T]}$ is such that $X_t$ is independent of the Brownian motion $B$ with $\mathcal{L}(X_t)=\mu$, its dynamics are determined by the control $\alpha$ through $dX_s=\alpha_s\,ds,$ and the infimum is taken over all controls in the set
    \be\label{controlsetot}
    \mathcal{A}_{t,\mu,\nu}:=\left\{\alpha\colon[t,T]\times\Omega\to\R^d:\; \alpha\in\mathcal{A},\;\exists X_t,\;\mathcal{L}(X_t)=\mu,\;\mathcal{L}(X_t+\int_t^T\alpha_sds)=\nu\right\}.
    \ee
    There is extensive literature in the finite dimensional setting on the qualitative behavior ($\Gamma$-limit theory) as $\e\to 0$ of functionals of the form
    \be\label{intfunctionalfd}
    \int_t^Tl(\gamma(s)/\e,\dot{\gamma}(s))ds,
    \ee
    where $l\colon \T^d\times \R^d\to \R$ and $\gamma\colon[t,T]\to \R^d$ is an absolutely continuous curve with derivative $\dot{\gamma}$ (for example, \cite{Weinan1991} and references therein); we refer to the books \cite{braides1998homogenization, jikov2012homogenization} and the references therein for a detailed account of homogenization of integral functionals. In this paper, our second goal is to prove that $U^\e_{ot}$ from \eqref{eq:value-function} has a limit as $\e\to 0$ and derive a rate of convergence. Besides being quantitative, the presence of the measure argument $\mathcal{L}(X_s)$ in \eqref{eq:value-function} yields that our result also provides an extension of previously established qualitative results for finite dimensional problems such as \eqref{intfunctionalfd} to infinite dimensions (for instance, see \cite{capuzzo2001rate, TranYu2021}). A similar qualitative result in infinite dimensions was obtained in the work of Gangbo and Tudorascu \cite{GangboTudorascu2012}, which considers curves $(X_s)_{s\in [t,T]}$ in $L^p,\; p\ge 1$.
    \vspace{2mm}

    \nd
    In the remainder of this introduction, we first review the existing results for HJ equations in finite dimensions and then state precisely our main results. We then turn to a comparison of our results with the existing literature on homogenization of HJ equations in the Wasserstein space, and give an overview of our proof techniques.

    \subsection{Description and statement of main results} Our main results depend on two different sets of assumptions for the Hamiltonian $H=H(y,x,p,m)\colon \T^d\times\R^d\times\R^d\times \mathcal{P}_2(\R^d)\to \R$.
    The first one considers the multiscale homogenization problem, where $H$ is a general function satisfying the conditions in Assumption (A1)
    (section \ref{assumptionssec}). The second assumes that $H(y,x,p,m)$, where $y\in\T^d,\; x\in \R^d,\; p\in \R^d$ and $m\in \mathcal{P}_2(\R^d)$, does not depend on the $x$ and $m$ variables (Assumption (A1')) (no multiscale effect). In the latter case, when there is no confusion, we write $H(y,p)$ instead of $H(y,x,p,m)$.
    \vspace{2mm}

    \nd
\nd
To better explain our results and where they place in the literature, we
first focus on the simpler case where $H=H(y,p),\; (y,p)\in \T^d\times \R^d$, does not depend on $(x,m)\in \R^d\times \mathcal{P}_2(\R^d)$,
and we describe the analogous finite-dimensional homogenization problem. Its dynamic programming equation is the classical first-order HJ equation
\be\label{fdHJBe}
\begin{cases}
    -\partial_tu^\e(t,x)+H\left(\frac{x}{\e},D_xu^\e(t,x)\right)=0,&(t,x)\in [0,T)\times\R^d,\\
    u^\e(T,x)=g(x),& x\in \R^d,
\end{cases}
\ee
where $g\colon \R^d\to \R$ is a continuous terminal cost. The function
$u^\e$ admits the variational representation
\[
u^\e(t,x)\;=\;\inf_{\alpha}\,
\left\{\int_t^T L\!\left(\frac{X_s}{\e},\alpha_s\right)ds+g(X_T)\right\},
\]
with state dynamics $dX_s=\alpha_s\,ds$, $X_t=x$, and $L$ the Legendre
transform of $H$ in the momentum variable. This is the deterministic,
single-particle analogue of the mean-field control problem
\eqref{eqn:U_value_function}: a single agent in $\R^d$ replaces a population
distributed according to $m\in\mathcal{P}_2(\R^d)$, and the terminal cost
$g$ replaces the measure-dependent terminal cost $G$.
\vspace{2mm}

\nd
Under coercivity and continuity assumptions on $H$ and continuity of $g$,
Lions--Papanicolaou--Varadhan \cite{LionsPapaVara1987} showed that
$u^\e\xrightarrow{\e\to 0} u$ locally uniformly, where $u\colon [0,T]\times \R^d\to \R$
is the viscosity solution to the effective problem
\be\label{fdeffHJB}\nonumber
\begin{cases}
    -\partial_tu(t,x)+\overline{H}(D_xu(t,x))=0,&(t,x)\in [0,T)\times\R^d,\\
    u(T,x)=g(x),&x\in \R^d.
\end{cases}
\ee
Here, $\overline{H}\colon \R^d\to \R$ is called the effective Hamiltonian;
$\overline{H}(p)$ is defined as the unique constant such that the cell problem
\be\label{cellintro}
    H(y,p+D_y v(y;p))=\overline{H}(p),\;\; y\in \T^d,
\ee
admits a continuous viscosity solution $v=v(\cdot;p)\colon\T^d\to \R$,
called a corrector. We refer to \cite[Section 3]{tran_hamilton-jacobi_2021}
or Proposition \ref{effHm} for the properties of the effective Hamiltonian
inherited from $H$.
\vspace{1mm}

\nd
From the definition of $\overline{H}$, we can heuristically write the expansion
\be\label{heuristicexp}
u^\e(t,x)\approx u(t,x)+\e v(x/\e;D_xu(t,x)),
\ee
which suggests a rate of convergence for $u^\e\xrightarrow{\e\to 0}u$ of $O(\e)$. Therefore, following the qualitative homogenization result of Lions--Papanicolaou--Varadhan, the question of quantitative homogenization for \eqref{fdHJBe} was addressed. Under the assumption that $g$ is Lipschitz and $H$ is coercive and locally Lipschitz, Capuzzo--Dolceta--Ishii \cite{capuzzo2001rate} obtained a rate of $O(\e^{1/3})$ using the perturbed test function method. Under the additional assumption that $H$ is convex in $p$, Mitake--Tran--Yu \cite{MitakeTranYu2019} proved a rate of $O(\e)$ in dimensions $d=2$ and a rate of $O(\e^{1/2})$ in dimensions $d\ge 3$ using weak KAM theory. Finally, Tran--Yu \cite{TranYu2021} obtained the optimal rate $O(\e)$ in the setting of \cite{MitakeTranYu2019} for any finite dimension $d$, which completes the study in the framework of convex Hamiltonians depending only on the fast variable and the momentum. For further developments and extensions in various directions, we refer to \cite{HanJang2023, HanJingMitakeTran2025, MitakeNi2026, HanTu2025, MitakeSato2023, MitakeNi2025, HuTuZhang2025, MitakeNiTran2025}. For earlier nearly optimal convergence rates in the convex setting, we refer to \cite{Son2020, Cooperman2022} and the references therein.
\vspace{1mm}

\nd
The assumptions on $H$ and $g$ in \cite{MitakeTranYu2019,TranYu2021} help by giving additional structure to the solution $u^\e$ of \eqref{fdHJBe} in the form of the control formulation:
\be\label{controlfd}
u^\e(t,x)=\inf_{\gamma\in AC([t,T];\R^d),\;\gamma(t)=x}\left\{\int_t^TL\left(\frac{\gamma(s)}{\e},\dot{\gamma}(s)\right)ds +g(\gamma(T))\right\},\;\;(t,x)\in [0,T]\times\R^d,
\ee
where $AC([t,T];\R^d)$ is the set of all absolutely continuous curves in $\R^d$, $\gamma(t)=x$ is the initial position and $\dot{\gamma}(s)$ is the speed of $\gamma$. A control formulation can be written for the solution of equation \eqref{HJBe}. More specifically, our assumptions from Section \ref{assumptionssec} ((A1')--(A2) and (A1)--(A2)) guarantee that \eqref{eqn:U_value_function} is the unique viscosity solution of \eqref{HJBe}\footnote{See \cite[Theorem 7.4]{gangbo2021finite} for the uniqueness and \cite{djete2022mckean1} for the dynamic programming principle that leads to the viscosity solution property.}.
\vspace{2mm}

\nd
The above heuristic connection\footnote{A more precise connection between \eqref{HJBe} and HJ equations in finite dimensions will be given in section \ref{viscositysolutions}.} between our problem and homogenization of HJ equations in finite dimensions suggests that an analogous expansion to \eqref{heuristicexp} can be expected for the solution $U^\e$ to \eqref{HJBe}. Our first result is that $U^\e$ converges with a rate of $O(\e)$ to a solution of
\be\label{HJB}\tag{$\mathrm{HJB1}$}
\begin{cases}
    -\partial_tU+\int_{\R^d}\bar{H}(D_mU(t,m,x))m(dx)=0, &(t,m)\in [0,T)\times \mathcal{P}_2(\R^d),\\
    U(T,m)= G(m), & m\in\mathcal{P}_2(\R^d),
\end{cases}
\ee
where $\overline{H}$ is the analog of \eqref{cellintro} in the Wasserstein space. In particular, we have the following theorem, which extends the above works to the infinite dimensional setting $\mathcal{P}_2(\R^d)$.

\begin{theorem}\label{thm:main1}
Assume that $H,G$ satisfy assumptions (A1') and (A2) in Section \ref{assumptionssec}. Then,  $p\mapsto\overline{H}(p)$ is convex         and \eqref{HJB} has a unique viscosity solution in the sense of Definition \ref{visc.}. In addition, for each $\varepsilon>0$, let $U^\varepsilon \colon [0,T]\times\mathcal{P}_2(\R^d) \to \mathbb{R}$ be the viscosity solution to \eqref{HJBe}, and let $U \colon [0,T]\times\mathcal{P}_2(\R^d) \to \mathbb{R}$ be the viscosity solution to \eqref{HJB}. Then, $U^\varepsilon$ converges uniformly to $U$ on $[0,T]\times\mathcal{P}_2(\R^d)$ and there exists a constant $C = C\big(d, C_{1,H}, C_{2,H}, K_0, C_G\big) > 0$
such that
\be\label{rate1}
\bigl| U^\varepsilon(t,m) - U(t,m) \bigr| \le C\,\varepsilon,
\quad \text{for all } t\in[0,T],\; m\in\mathcal{P}_2(\R^d).
\ee
Here, $C_{1,H}, C_{2,H}, K_0, C_G$ are the constants that appear in assumptions (A1') and (A2).
\end{theorem}
\vspace{2mm}

\nd
The above theorem extends the previously known qualitative/quantitative homogenization results for the finite dimensional control problem $u^\e$ in \eqref{controlfd} to infinite dimensions and achieve the optimal convergence rate. In particular, a random initial position ($X_t$ is random) and a random speed $\alpha$ are allowed, as well as a terminal cost $G$ that depends on the law of $X_T$ and not necessarily its value. The proof is developed in Theorem \ref{thm: unif conv} for the rate of convergence, in section \ref{charoflimit} for the characterization of the limit, and in Proposition \ref{prop:opex} for the example verifying the optimality of the rate.

\vspace{2mm}

    \nd
    Our second result extends Theorem \ref{thm:main1} to the more general setting of assumptions (A1)–(A2) from Section \ref{assumptionssec}, where the Hamiltonian may depend on both the macroscopic variable $x\in \R^d$ and the probability measure $m\in \mathcal{P}_2(\R^d)$. In this case, for given $(x,p,m)\in \R^d\times\R^d\times \mathcal{P}_2(\R^d)$ the cell problem takes the form
    \be\label{cellintrogen}
        H(y,x,p+D_y v(y;x,p,m),m)=\overline{H}(x,p,m),\;\;y\in\T^d,
    \ee
    where $\overline{H}\colon \R^d\times\R^d\times\mathcal{P}_2(\R^d)\to \R$ is the effective Hamiltonian\footnote{We still denote the effective Hamiltonian by $\overline{H}$, because if $H$ in \eqref{cellintrogen} is independent of $x,m$, then $\overline{H}(x,p,m)$ is the effective Hamiltonian constructed in \eqref{cellintro}.}. We expect that $U^\e$ in this case will converge to a viscosity solution of
    \be\label{HJB4}\tag{$\mathrm{HJB2}$}
\begin{cases}
    -\partial_tU+\int_{\R^d}\bar{H}(x, D_mU(t,m,x),m)m(dx)=0, &(t,m)\in [0,T)\times \mathcal{P}_2(\R^d),\\
    U(T,m)= G(m), & m\in\mathcal{P}_2(\R^d),
\end{cases}
\ee
    and we further establish an $O(\sqrt{\e})$ convergence rate. The following is proved in Section \ref{sec:generall}.

\begin{theorem}\label{thm:multiscale}
Assume that $H,G$ satisfy assumptions (A1) and (A2). Then, \eqref{HJB4} has a unique viscosity solution. For each $\varepsilon>0$, let $U^\varepsilon \colon [0,T]\times\mathcal{P}_2(\R^d) \to \mathbb{R}$ be the viscosity solution to \eqref{HJBe}, and let $U\colon [0,T]\times\mathcal{P}_2(\R^d) \to \mathbb{R}$ be the viscosity solution to \eqref{HJB4}. Then, $U^\ep$ converges uniformly to $U$ on $[0,T]\times \mathcal{P}_2(\R^d)$ as $\ep\to 0$. Moreover, there exists a constant $C=C\left(d, C_{1, H}, C_{2,H}, C_H, C_G, K_0\right)>0$ such that for any $(t, m)\in [0,T] \times \mathcal{P}_2(\R^d)$,
\[
\left|U^\varepsilon(t,m)-U(t,m)\right| \le
\begin{cases}
C(T-t)\sqrt{\varepsilon}, & T-t \ge \sqrt{\varepsilon},\\[4pt]
C\min\{T-t,\varepsilon\}, & 0\leq T-t<\sqrt{\varepsilon}.
\end{cases}
\]
\end{theorem}
\vspace{2mm}

\nd
The obtained rate of convergence is closely related to the multiscale result in finite-dimensional spaces established in \cite[Theorem 1.1]{HanJang2023}, where the rate is shown to be optimal, at least for $T-t \in [0,\varepsilon]$. Although we are unable to establish optimality in the present setting, our result may be viewed as an extension of the finite-dimensional convergence rate to the infinite-dimensional setting.
\vspace{1mm}

\nd
It is worth noting that under Assumption (A1) (resp. (A1')) on $H$, $\overline{H}$ satisfies certain convexity, coercivity and local Lipschitz assumptions (Proposition \ref{effHm}) which are enough to characterize the unique viscosity solution $U$ of \eqref{HJB4} (resp. \eqref{HJB}) in terms of the optimal control problem
\begin{align}
U(t,m)=\inf_{\alpha\in\mathcal{A}}&\left\{\mathbb E\left[\int_t^T\overline{L}(X_s,\alpha_s,\mathcal{L}(X_s))ds\right]+G(\mathcal{L}(X_T))\right\}\label{formula}\\
&\hspace{2cm}\left(\text{resp. } U(t,m)=\inf_{\alpha\in\mathcal{A}}\left\{\mathbb{E}\left[\int_t^T\overline{L}(\alpha_s)ds\right]+G(\mathcal{L}(X_T))\right\}\right).\nonumber
\end{align}
Here, $\mathcal{A}$ is given in \eqref{eqn:cs} and $\overline{L}=\overline{L}(x,v,m)\colon\R^d\times\R^d\times\mathcal{P}_2(\R^d)\to \R$ (resp. $\overline{L}\colon\R^d\to \R$) is the Legendre transform with respect to the $p$ variable of the effective Hamiltonian $\overline{H}(x,p,m)$ (resp. $\overline{H}(p)$)
\be\label{effLagr}
\bar{L}(x,v,m) := \sup_{p \in \R^d} \big\{ -p \cdot v - \bar{H}(x,p,m) \big\}\;\;\big(\text{resp. } \bar{L}(v) := \sup_{p \in \R^d} \big\{ -p \cdot v - \bar{H}(p) \big\}   \big).
\ee
\vspace{2mm}

\nd
Finally, going beyond the case where $G(m)<\infty$ in \eqref{eqn:U_value_function}, we obtain a quantitative homogenization result for the dynamic optimal transport problem \eqref{eq:value-function}. By using similar arguments as in the proofs of Theorem \ref{thm:main1} or Theorem \ref{thm:multiscale}-- depending on whether $H$ satisfies (A1) or (A1')
-- we have for $t\in [0,T]$ and $\mu,\nu\in\mathcal{P}_2(\R^d)$
$$U^\e_{ot}(t,\mu,\nu)\xrightarrow{\e\to 0}U_{ot}(t,\mu,\nu),$$
where
\be\label{DOTlimit}
U_{ot}(t,\mu,\nu)=\inf_{\alpha\in \mathcal{A}_{t,\mu,\nu}}\left\{\mathbb{E}\left[\int_t^T\overline{L}(X_s,\alpha_s,\mathcal{L}(X_s))ds   \right]\right\}
\ee
is the dynamic optimal transport problem with cost $\overline{L}$.

\begin{theorem}\label{main3}
    Suppose that assumptions (A1)--(A2) are in place. For $\e>0$ consider the function $U^\e_{ot}\colon [0,T]\times\mathcal{P}_2(\R^d)\times\mathcal{P}_2(\R^d)\to \R$ defined in \eqref{eq:value-function} and fix $(t,\mu,\nu)\in [0,T)\times\mathcal{P}_2(\R^d)\times\mathcal{P}_2(\R^d)$. There exists a constant $C$ depending on  $d, C_{1, H}, C_{2,H}, C_H, C_G$ and $K_0$ such that the following inequality holds
    $$|U^\e_{ot}(t,\mu,\nu)-U_{ot}(t,\mu,\nu)|\le C
\left(T-t+\frac{{\bf d}_2^2(\mu,\nu)}{T-t}  \right)\sqrt{\e},$$
    where $U_{ot}$ is defined in \eqref{DOTlimit}. Furthermore, if Assumption (A1) is replaced by the simpler (A1'), then the estimate can be improved to
    $$|U^\e_{ot}(t,\mu,\nu)-U_{ot}(t,\mu,\nu)|\le C \left(1+\frac{{\bf d}_2^2(\mu,\nu)}{(T-t)^2}\right)\e.$$
\end{theorem}

\nd
The above result is proved in Section \ref{DOT}.

    \subsection{Proof outline, methods, and a Corollary}
    To show Theorem \ref{thm:main1}, unlike the classical approach for the finite dimensional problem \eqref{fdHJBe}, where one first proves qualitative convergence via the doubling of variables method and then derives a rate, we first extend the proof techniques developed in~\cite{TranYu2021} to prove $u^\varepsilon$ uniformly converges to a limit function with a quantitative rate $O(\varepsilon)$. Then, we use the doubling variable techniques to prove the limiting function must be a solution to the limiting equation~\eqref{HJB}.


    \vspace{1mm}

    \nd
    The main reason we are unable to work directly with \eqref{HJBe} is the topology of $\mathcal{P}_2(\R^d)$. Even though, under our assumption from section \ref{assumptionssec}, we are able to show that the family $(U^\e)_{\e>0}$ is uniformly equicontinuous (see Corollary \ref{cor:lipschitzestimate}), in order to construct and characterize its limit via the doubling of variables method we require $\mathcal{P}_2(\R^d)$ to be locally compact-- something that is not true. However, this difficulty can be overcome if someone knows apriori that $(U^\e)_{\e>0}$ converges uniformly to its limit; we refer to the main proof in subsection \ref{charoflimit} and the comment before it.
    \vspace{1.5mm}

    \nd
    To show the uniform convergence, as in the works \cite{MitakeTranYu2019,TranYu2021} mentioned above, our argument builds on the optimal control representation of $U^\varepsilon$ (see \eqref{eqn:U_value_function}). In particular, we show that the control set $\mathcal{A}$ in \eqref{eqn:U_value_function} can be extended without changing the value function and this new optimal control problem admits some equivalent representations that reveal its behavior as $\e \to 0$. The main steps are summarized below in the case of Theorem \ref{thm:main1}. In what follows $(t,m)\in [0,T]\times\mathcal{P}_2(\R^d)$ denote the initial time and distribution, respectively, and $H,L$ are defined under the simplified assumption (A1') described above (see also section \ref{assumptionssec}).
    \vspace{3mm}

    \nd
    \textit{Step 1. (Extension of the control set)}\\
    We first extend the control set $\mathcal{A}$ in \eqref{eqn:U_value_function} to
    \[
    \mathcal{A}_1:=\left\{\alpha:[t,T]\times \Omega\rightarrow \mathbb{R}^d: \alpha \text{ is } \mathcal{B}([t,T])\otimes \mathcal{F} \text{-measurable and } \mathbb{E}\left[ \int_t^T|\alpha_s|^2 \,ds \right] <\infty\right\}\]
    and we show that new value function $U^\e_1$ is in fact equal to $U^\e$ (see Proposition \ref{prop:U1eqUeqU2}).
    \vspace{2mm}

    \nd
    \textit{Step 2. (Alternative representation of the value function $U^\e$)}\\
    We derive an alternative formulation of $U^\e(t,m)=U_1^\varepsilon(t,m)$. Rather than working directly with the admissible controls $\alpha\in \mathcal{A}_1$ and the stochastic trajectories $(X_s)_{s\in [t,T]}$, with $dX_s=\alpha_sds$, we consider:
    \begin{enumerate}
    \item general target measures $\nu \in \mathcal{P}_2(\R^d)$;
    \item random variables $X,Y\colon \Omega\to \R^d$ such that $\mathcal{L}(X)=m$ and $\mathcal{L}(Y)=\nu$, and
    \item random paths $\gamma:[t,T]\times\Omega\to\mathbb{R}^d$ joining $X$ to $Y$
    \footnote{In the sense that $\gamma\colon [t,T]\times \Omega\to \R^d$ is a measurable map such that $\gamma(\cdot,\omega)$ is absolutely continuous for any $\omega\in \Omega$, and $\gamma(t,\omega)=X(\omega)$,  $\gamma(T,\omega)=Y(\omega)$, for any $\omega\in \Omega$. We write $\gamma(t)=X$ and $\gamma(T)=Y$ for the last two equalities whenever there is no confusion.},
    \end{enumerate}
    $$U_1^\e(t,m)=\inf_{\nu\in \mathcal{P}_2(\R^d)}\inf_{\substack{X,Y\in L^2(\Omega;\R^d)\\ \mathcal{L}(X)=m,\;\mathcal{L}(Y)=\nu}}\inf_{ \gamma(t)=X,\;\gamma(T)=Y}\left\{\mathbb{E}_\omega\left[\int_t^TL\left(\frac{\gamma(s,\omega)}{\e},\dot{\gamma}(s,\omega)   \right)ds +G(\nu)\right]\right\}.$$
    In this representation, we show that the last infimum can be exchanged with the expectation $\mathbb{E}$ without changing the value (see Lemma \ref{lem:Uep1costh}). After this, for each fixed $\omega\in \Omega$, the running cost inside the expectation becomes the minimal deterministic cost of transporting $x=X(\omega)$ to $y=Y(\omega)$ and is written as
    \be\label{cost}
    h^\e(t,T,x,y):= \inf_{\substack{\gamma\in \text{AC}([t,T];\R^d)\\ \gamma(t)=x,\;\gamma(T)=y}} \int_t^TL\left(\frac{\gamma(s)}{\e},\dot{\gamma}(s)\right)ds,\;\;\text{for }x=X(\omega)\text{ and } y=Y(\omega).
    \ee
    More specifically,
    \be\label{importanteq}
    U^\e_1(t,m)=\inf_{\nu\in \mathcal{P}_2(\R^d)}\inf_{\substack{X,Y\in L^2(\Omega;\R^d)\\ \mathcal{L}(X)=\mu,\;\mathcal{L}(Y)=\nu}}\left\{ \mathbb{E}[h^\e(t,T,X,Y)]+G(\nu)  \right\}.
    \ee
    This running cost $h^\e\colon [0,T]\times [0,T]\times\R^d\times\R^d\to \R$ was used in \cite{TranYu2021} to study the quantitative homogenization problem \eqref{fdHJBe}. A key technical ingredient in justifying the above formula is a uniform $L^2$ velocity bound for nearly optimal paths $\gamma$, ensuring that the constant $C$ in Theorem~\ref{thm:main1} depends only on the structural parameters of the problem.

    \vspace{2mm}

    \nd
    \textit{Step 3. (Definition of the effective value function)}\\
    We introduce the limiting value function $\bar{U}$, in which the running cost is given by the expectation of the averaged cost $\bar{h}$ associated with $h^\e$. Using the subadditivity and superadditivity of $h^\e$, we obtain both the uniform convergence $U^\varepsilon \xrightarrow{\e\to 0} \bar{U}$ on $[0,T]\times\mathcal{P}_2(\R^d)$ and
    the quantitative estimate $O(\varepsilon)$ for the convergence rate. This step is proved in Theorem \ref{thm: unif conv}.
    \vspace{2mm}

    \nd
    \textit{Step 4. (Identification of the limit via the doubling variable method)}\\
    Finally, we use a doubling of variables argument to verify that $\bar{U}$ is an $L$-viscosity solution of the effective Hamilton--Jacobi equation \eqref{HJB}. The uniqueness then yields $\bar{U}=U$, completing the proof of Theorem~\ref{thm:main1}. The proof of the viscosity property of $\overline{U}$ is in subsection \ref{charoflimit} and the proof of Theorem \ref{thm:main1} is in subsection \ref{proofmain1}.
    \vspace{1mm}

    \begin{remark}[A Corollary of our proof methods]
        (i) Even though, a posteriori, \eqref{importanteq} is true for $U^\e$ in place of $U^\e_1$, we are unable to show it directly for $U^\e$ using the outline in step 2 due to the controls in $\mathcal{A}$ being progressively measurable.\\[0.5ex]
        (ii) It is worth noting that the optimal control problem can be equivalently written in the familiar Markovian form
        $$U^\e(t,m)=\inf\left\{\mathbb{E}\left[\int_t^TL\left(\frac{X_s}{\e},X_s,\alpha(s,X_s),\mathcal{L}(X_s)\right)ds+G(\mathcal{L}(X_T))       \right]\right\},$$
        where the infimum is taken over all measurable $\alpha\colon [t,T]\times\R^d\to \R$  such that $dX_s=\alpha(s,X_s)ds,$ for $s\in[t,T]$, $\int_t^T\mathbb{E}[|\alpha(s,X_s)|^2]ds<\infty$ and $X_t\in L^2(\Omega)$ with $\mathcal{L}(X_t)=m$. We refer to \cite[Theorem 8.1]{lacker2022superposition} or \cite[Lemma 6.2]{CecchinDaudinJacksonMartini2025}.\\[0.5ex]
        (iii) In order to show the equality in step 1, we compare the value $U^\e(t,m)$ with the value functions of the optimal control problem \eqref{eqn:U_value_function} when the control set is extended (this is the case of $\mathcal{A}_1$ above) or restricted in the sense that the controls $\alpha$ are not just $\mathbb{F}$-progressively measurable but $\alpha(s,\cdot)$ is $\mathcal{F}_t$-measurable for any $s\in [t,T]$ (see $\mathcal{A}_2$ in section \ref{altcontrset}). In both cases we show that the value function remains the same (see section \ref{equivalence}) the main reason being that the randomness generated by the Brownian motion allows us to produce enough paths $(X_s)_{s\in [t,T]}$ and, in case $m$ has atoms, for the mass to split. It is natural to consider the even further restricted control set
        $$\widetilde{\mathcal{A}}=\left\{\alpha\colon [t,T]\times\Omega\to \R^d: \alpha\in \mathcal{A},\; \alpha\text{ is }\mathcal{B}([t,T])\otimes\sigma(X_t)\text{-measurable}\right\},$$
        where $\sigma(X_t)$ is the $\sigma$-algebra generated by the $\mathcal{F}_t$-measurable random variable $X_t$, and compare it with the value function of \eqref{eqn:U_value_function}. As a consequence of the methods developed in this paper, we obtain the following Corollary.
\end{remark}

\begin{corollary}\label{aeequality}
    Assume that (A1')-(A2) are in place and $L$ is defined as in \eqref{Lagr}. For $(t,m)\in [0,T]\times\mathcal{P}_2(\R^d)$ we define
            \be\nonumber
\widetilde{U}^\e(t,m)=\inf_{\alpha\in\widetilde{\mathcal{A}}}\left\{ \mathbb{E}\left[ \int_t^TL\left( \frac{X_s}{\e},\alpha_s\right)ds+G(\mathcal{L}(X_T))   \right]\right\},
            \ee
where $dX_s=\alpha_sds,\;s\in[t,T],\; \mathcal{L}(X_t)=m$.
Then the following statements hold.\\
(1) $\widetilde{U}^\e(t,m)=U^\e(t,m)$, for any $m\in\mathcal{P}_2(\R^d)$ which is non-atomic. As a result, $\widetilde{U}^\e(t,m)$ converges to $U(t,m)$, where $U$ is the unique solution of \eqref{HJB}, with rate $O(\e)$ for any $m\in\mathcal{P}_2(\R^d)$ non-atomic.
\\
(2) Let $g:\R^d\to\R$ be bounded and Lipschitz, and let $u^\e$ be the solution to \eqref{fdHJBe}.
Suppose that
    \[
    G(m)=\int_{\R^d} g(x)\,m(dx),
    \qquad m\in\mathcal{P}_2(\R^d).
    \]
 Then, for every $x\in\R^d$, $U^\e(t,\delta_x)=\widetilde{U}^\e(t,\delta_x)=u^\e(t,x)$.
 Consequently, Theorem~\ref{thm:main1} implies that
 \[
 \left|u^\e(t,x)-u(t,x)\right| \leq C \e, \quad \text{for all } t\in [0,T], \,x \in \R^d,
 \]
 recovering \cite[Theorem 1.1]{TranYu2021}.
\end{corollary}


        \nd
        \begin{remark}
        Note that in the setting of part $(2)$, with proper regularity assumptions, one can prove that $U^\e(t,m)=\int_{\R^d} u^\e(t,x)m(dx)$ for any $(t, m) \in [0,T] \times \mathcal{P}_2(\R^d)$. Formally, let $W(t,m)=\int_{\R^d} u^\e(t,x)m(dx)$. We notice $D_mW(t,m)=\nabla_x \frac{\delta W}{\delta m}=D_x u^\e(t,x)$. Because $u^\e$ solves \eqref{fdHJBe}, we have
        \[
        \partial_tu^\e(t,x)+H(x/\e,x,D_xu^\e(t,x))=0,\quad \forall (t,x) \in [0,T) \times \R^d.
        \]
        Integrating $x$ with $m$ gives
        \[
        \partial_t \int u^\e (t,x)m(dx)+\int H(x/\e,x,D_x u^\e(t,x))m(dx)=0\,.
        \]
        The rigorous argument follows by generalizing the proof of Corollary \ref{aeequality} $(2)$, where only the special case $m=\delta_{x_0}$ is considered in this paper. More precisely, using the optimal control representation, one shows that $U^\e(t,m)=\widetilde U^\e(t,m)=\int_{\mathbb{R}^d} u^\e(t,x) m(dx)$,
        and the desired identity follows. Since this is not the main focus of our work, we omit the detailed proof for simplicity.

        However, for general terminal costs $G$, as in part (1), the function $\widetilde{U}^\e(t,\cdot)$ might not be continuous everywhere, hence the equality $\widetilde{U}^\e=U^\e$ is not expected to hold everywhere.
        A case of strict inequality $U^\e(t,m)<\widetilde{U}^\e(t,m)$, with $G(m)={\bf d}_1(m, \text{Leb}|_{[0,1]^d})$, was presented in Cecchin-Daudin-Jackson-Martini \cite[Corollary 6.10]{CecchinDaudinJacksonMartini2025}. We give the proof of Corollary \ref{aeequality} and an example of a discontinuous $\widetilde{U}^\e$ in Appendix \ref{Corapp}.
    \end{remark}



\nd
The proof of Theorem \ref{thm:multiscale}, which assumes the more general Assumption (A1), follows a similar strategy to that of multiscale problems in $\R^d$ (see \cite{HanJang2023}), with an additional difficulty arising from the need to handle the measurability of admissible paths and the dependence of the cost $L$ on the law $\mathcal{L}(X_s)$. The main idea is to break the time period $[t,T]$ into $N\in \mathbb{N}$ evenly spaced subintervals $[t_i,t_{i+1}],\;i\in \{0,\ldots,N-1\}$. For each subinterval, we fix the argument of the macroscopic variable $X_s$ and the probability measure $\mathcal{L}(X_s)$ at the value of initial time $t_i$, while retaining the dependence on $X_s/\e$. We may then apply the result of Theorem \ref{thm:main1} for each subinterval and we also estimate the approximation error for fixing the macroscopic and measure arguments. The final step is to optimize with respect to $N$ in order to minimize the error.
\vspace{2mm}

\nd
To prove Theorem \ref{main3} we first show that the functions $U^\e_{ot}(t,\mu,\nu),\; U_{ot}(t,\mu,\nu)$ from \eqref{eq:value-function} and \eqref{DOTlimit}, maintain their values if the control set $\mathcal{A}_{t,\mu,\nu}$ is extended to controls in $\mathcal{A}_1$ so that the state process $(X_s)_{s\in [t,T]}$, with $dX_s=\alpha_sds,\;\mathcal{L}(X_t)=\mu$, satisfies $\mathcal{L}(X_T)=\nu$. This requires us to explore certain continuity properties of dynamic optimal transport problems $U^\e_{ot}$ under different sets of controls. We then apply the techniques developed in Theorems \ref{thm:main1} and \ref{thm:multiscale} to obtain the rate of convergence.



\subsection{Related literature and outlook}
To the best of our knowledge, the quantitative homogenization results in Theorems \ref{thm:main1} and \ref{thm:multiscale} are new. Homogenization of first order HJ equations in infinite dimensions was studied from a qualitative point of view in Gangbo-Tudorascu \cite[Section 3.3]{GangboTudorascu2012}, and in the recent \cite{park2026periodic} from a quantitative point of view, where, in the latter work, the main first-order HJ equation is posed in $L^2$ and additional rearrangement and periodicity assumptions are made to compensate for the lack of compactness\footnote{We note that similar compactness properties can be obtained in our setting if $H(y,\cdot,p,\cdot),\; G(\cdot)$ are apriori defined on $\T^d\times\mathcal{P}_2(\T^d)$ (resp. $\mathcal{P}_2(\R^d))$ instead of $\R^d\times \mathcal{P}_2(\R^d)$ (resp. $\mathcal{P}_2(\T^d))$ and hence \eqref{HJBe} is posed in $[0,T]\times\mathcal{P}_2(\T^d)$ instead of $[0,T]\times\mathcal{P}_2(\R^d)$.}.
An analogous qualitative version of Theorem \ref{main3} was obtained also in \cite{GangboTudorascu2012}, where the authors proved the $\Gamma$-convergence of the integral functional, and extended classical results (e.g. Weinan \cite[Section 7]{Weinan1991})  to infinite dimensions. To obtain a quantitative version of the result in \cite{GangboTudorascu2012} in the present paper, we use the quantitative convergence of transport cost functions proved in~\cite{TranYu2021}. This result controls the convergence of the cost of transporting $x_1$ to $x_2$ with oscillatory running cost $L(x/\e,\cdot)$ to the effective cost associated with the homogenized running cost $\overline{L}$. We combine this estimate with a suitable partition of the optimal transport path to handle the multiscale dependence and nonlinear dependence on the law (general mean-field interactions).


\vspace{1.5mm}

\nd
Explicit examples of effective Hamiltonians $\overline{H}$ and effective Lagrangians $\overline{L}$ are sparse in the literature. We refer to \cite[Section 4]{tran_hamilton-jacobi_2021} for some of them. A computation verifying Theorem \ref{main3} has appeared in \cite[Section 5.2]{Gao_Yip_2025}, where the homogenized optimal transport distance $U_{ot}(t,\mu,\nu)$ is compared with the metric corresponding to a homogenized gradient flow in the Wasserstein space.
\vspace{0.5mm}

\vspace{-3mm}

\nd
A natural follow up question is whether the results presented in this paper can be extended to the case where the dynamics of $X_s$ in \eqref{eqn:U_value_function} depend on the Brownian motion $B$ (case with idiosyncratic noise), or on a common noise $B^0$ (and the distributions in the cost functional are subject to the common noise), or both; we refer to \cite{djete2022mckean1,djete2022mckean} for the precise setup. In the presence of noise, questions of homogenization have only been studied in the periodic case where measures are defined on the compact space $\mathcal{P}_2(\T^d)$ and the PDE techniques from the theory of finite dimensional HJ equations are more easily accessible. For related singular perturbation problems involving second-order Hamilton--Jacobi equations in the Wasserstein space and coupled systems of controlled conditional slow-fast McKean--Vlasov SDEs, we refer to \cite{zitridis2025singular,Zitridis2023}. To our knowledge, our proof methods, especially the formulations developed in section \ref{section4} (and  allowed us to circumvent the difficulties created by the lack of local compactness of $\mathcal{P}_2(\R^d)$) cannot be directly applied to the case where there is idiosyncratic or common noise. We intend to investigate this extension in a future work.

    \vspace{4mm}

    \nd
    The rest of the paper is organized as follows. In section 2, we state our assumptions on $H,\; G$ and some preliminary results about the effective Hamiltonian $\overline{H}$ and its dual $\overline{L}$ defined in \eqref{effLagr}. In section 3, we discuss the definition of $L$-viscosity solutions for \eqref{HJBe}, \eqref{HJB} and \eqref{HJB4}, as well as a relation between these equations and their finite dimensional counterparts. In section 4, we present and prove the equivalent formulations of $U^\e$ from \eqref{eqn:U_value_function}. In section 5, we study the behavior of $U^\e$ under the stronger Assumption (A1') (Theorem \ref{thm:main1}). In section 6, we prove Theorem \ref{thm:multiscale}. Section 7 is devoted to the proof of Theorem \ref{main3}. Finally, we use the appendix to show some technical propositions and Corollary \ref{aeequality}.

\section{Notation, Assumptions and Preliminary Results}\label{assumptionssec}

\subsection{Notation} Throughout the paper, for $2\ge p\ge 1$, we use the notation $\mathcal{P}_p=\mathcal{P}_p(\R^d)$ for the space of probability measures over $\R^d$ with finite second moment endowed with the ${\bf d}_p$ Wasserstein metric, which is given by
     $${\bf d}_p(\mu,\nu):=\inf_{\pi\in \Pi(\mu,\nu)}\left( \iint_{\R^d\times\R^d} |x-y|^p\pi(dx,dy)\right)^{1/p},\;\;\mu,\nu\in \mathcal{P}_p,$$
     where $\Pi(\mu,\nu)$ is the set of probability measures over $\R^d\times\R^d$ with marginals $\mu$ and $\nu$. For $x\in \R^d$, $\delta_x\in\mathcal{P}_2$ is the Dirac delta measure at $x$. For $N\in\mathbb{N}$ and $\bm x=(x_1,\ldots,x_N)\in (\R^d)^N$, we denote by $m^N_{\bm x}=\frac{1}{N}\sum_{i=1}^N\delta_{x_i}$ the empirical measure. $|\cdot|$ is the usual euclidean norm in $\R^d$ and $\|\cdot\|_p$ the $p$-norm; that is $\|\cdot\|_2=|\cdot|$. For $A\subseteq\R$ we use the symbol $\mathcal{B}(A)$ for the Borel sigma algebra over $A$. For a path $\gamma:[t,T] \times \Omega \to \mathbb{R}^d$, we denote by $\gamma_s := \gamma(s,\cdot)$ the associated random variable at time $s$. Given a Hilbert space $\mathcal{H}$, we use the symbol $\langle\cdot,\cdot\rangle$ for the inner product and $\|\cdot\|_2$ for the norm.\\
     For the rest of the paper, $\mathbb{E}$ or $\mathbb{E}_\omega$ will be the expectation defined for the probability space $(\Omega,\mathcal{F},\mathbb{F},\mathbb{P})$. For $\mu\in \mathcal{P}_2$ and $f\colon \R^d\to \R^d$, we also use the notation $\mathbb{E}_{\mu}[f]=\int f(x)dx$. For a random variable $X\colon \Omega\to\R^d$ with $X\in L^2(\Omega;\R^d)$ we write $\mathcal{L}(X)=\mu\in\mathcal{P}_2$ for its law or simply $X\sim \mu$. If the random variables $X,Y$ have the same law, we write $X\stackrel{d}{=}Y$. We also write $\sigma(X,Y)$ for the sigma algebra generated by the random variables $X,Y$.\\
     For functions $u(t,x)=u\colon[0,T]\times\R^d\to \R$ we denote by $\dot{u}$ or $\partial_tu$ its time derivative and by $D_xu$ or simply $Du$ (whenever there is no confusion) its spacial derivative.

\subsection{Assumptions}We impose the following assumptions on the Hamiltonian $H$ and the terminal cost $G$. Let $r\in (0,1)$ be given.
\vspace{3mm}

\nd
(A1) Let $H:\T^d\times\R^d\times\R^d \times \mathcal{P}_2 \rightarrow \R$. There exist constants  $C_{1,H},C_{2,H},K_0,C_H>0 $ such that :
\begin{align}
&\hspace{5cm}H(y,x,p,m) \text{ is convex in }p, \\
    &\hspace{3cm}C_{1, H}|p|^{2} -K_0\le H(y,x,p,m) \leq C_{2, H}|p|^{2}+K_0, \label{quadgrowthHm}\\
    &|H(y,x,p,m)-H(y',x',p',m')|\le C_H(1+|p|+|p'|)(|x-x'|+|y-y'|+|p-p'|+{\bf d}_{2-r}(m,m')),\label{ineqHm}
\end{align}
for any $y, y'\in \T^d, x, x',p,p'\in \R^d$, and $m,m'\in \mathcal{P}_2$.
\vspace{1mm}

\nd
(A2) The terminal cost $G:\mathcal{P}_2\rightarrow \R$ is bounded and ${\bf d}_{2-r}$-Lipschitz; that is, there exists a constant $C_G$ such that
\be\label{terminalassumption}
|G(m)|\le C_G\quad\text{and}\quad |G(m_1)-G(m_2)|\le C_G {\bf d}_{2-r}(m_1,m_2), \text{ for every }m,m_1,m_2\in\mathcal{P}_2.
\ee

\vspace{3mm}

To establish the convergence rate of $U^\ep$ to $U$ under these assumptions, we first consider a simplified setting in which the Hamiltonian does not depend (or is constant) on the macroscopic variable $x$ and the probability measure $m$. In particular, in Section~\ref{sec:rocresult} we replace (A1) with the reduced assumption (A1') to obtain quantitative convergence results in this case, and then return to the full assumption (A1) in Section~\ref{sec:generall}.

\vspace{3mm}

\nd
(A1') The Hamiltonian $H:\T^d\times\R^d \rightarrow \R$ satisfies the following assumptions:
\begin{align}
H(y,p) &\text{ is convex in }p, \\
    C_{1, H}|p|^{2} -K_0&\le H(y,p) \leq C_{2, H}|p|^{2}+K_0,\;\text{ for some constants } C_{1, H}, C_{2, H}, K_0>
    0, \label{quadgrowth}\\
    |H(y,p)-H(y',p')|&\le C_H(1+|p|+|p'|)(|y-y'|+|p-p'|),\label{ineq}
\end{align}
for some constants $C_H>0$ and for any $y,y'\in \R^d\; ,p,p'\in \R^d$.
\vspace{2mm}

    The assumptions imposed on $H,G$ guarantee the existence and the uniqueness of an $L$-viscosity solution of \eqref{HJBe}. We refer to section \ref{viscositysolutions} for a more detailed discussion about the definition, the solution and its properties.


\subsection{Preliminary Results}
\subsubsection{Properties of Lagrangians }
Under (A1) and (A2), we collect several properties of $L$, the Legendre transform of $H$ (defined in \eqref{Lagr}), the effective Hamiltonian $\bar{H}$, and its Legendre transform $\bar{L}$. The proofs of Propositions~\ref{prop:Lmprop} and \ref{effHm} are given in Appendix \ref{proofofprelim}, while the proof of Proposition~\ref{prop:Lbarmprop} is similar and therefore omitted. The same properties still hold if (A1) is replaced by (A1'), that is, when $H(y,x,p,m)$ is independent of the variables $x$ and $m$.

\begin{proposition}\label{prop:Lmprop}
 Let $L:\T^d\times \R^d \times \R^d \times \mathcal{P}_2\rightarrow \mathbb{R}$ be the Legendre tranform of $H$, defined by
\[
L(y,x,v,m) := \sup_{p \in \R^d} \big\{ -p \cdot v - H(y,x,p,m) \big\}.
\]
Then, the following statements hold.
\begin{enumerate}
    \item[$\rm{(1)}$] \(L\) is finite everywhere on \(\T^d\times \R^d \times \R^d \times \mathcal{P}_2 \).
    \item[$\rm{(2)}$] For any \(y \in \mathbb{T}^d, x \in \R^d\), and \(m \in \mathcal{P}_2\), the map \(v \mapsto L(y,x,v, m)\) is convex.
    \item[$\rm{(3)}$] For any \(y \in \mathbb{T}^d, x\in \R^d ,v \in \mathbb{R}^d\), and \(m \in \mathcal{P}_2\),
    \begin{equation}\label{eqn:Lmquad}
    \frac{|v|^2}{4C_{2,H}} -K_0 \leq L(y,x, v, m) \leq \frac{|v|^2}{4C_{1,H}} +K_0.       \end{equation}

    \item[$\rm{(4)}$] There exists a constant $C_L=C_L(C_{1,H},C_H,K_0)>0$ such that for all
    $y,y'\in \T^d$, $x,x', v,v'\in \R^d$, and $m,m'\in \mathcal{P}_2$,
    \begin{equation}\label{eqn:LmLip}
    |L(y,x,v,m)-L(y',x',v',m')|\le C_L(1+|v|+|v'|)(|x-x'|+|y-y'|+|v-v'|+{\bf d}_2(m,m')).
    \end{equation}
\end{enumerate}
\end{proposition}

\begin{proposition}\label{effHm}
Let $\delta>0$. Fix $x, p\in \mathbb{R}^d$ and $m \in \mathcal{P}_2$. Let $v^\delta = v^\delta(\cdot,x,p,m)\colon \mathbb{T}^d \to \mathbb{R}$
denote the unique viscosity solution of the discounted cell problem
\begin{equation}\label{discountcellm}
    \delta v^\delta(y) + H\big(y,x, Dv^\delta(y) + p, m\big) = 0,
    \qquad y \in \mathbb{T}^d.
\end{equation}
Then there exists a constant $C>0$, depending only on $p$ and $H$, such that
\[
\delta\| v^\delta\|_{L^\infty(\mathbb{T}^d)} + \|Dv^\delta\|_{L^\infty(\mathbb{T}^d)} \le C.
\]
Moreover, as $\delta \to 0^+$, the family $\delta v^\delta(\cdot,x,p,m)$ converges uniformly on $\mathbb{T}^d$ to a constant $-\overline{H}(x,p,m)$. Finally, allowing $x, p \in \mathbb{R}^d$ and $m\in \mathcal{P}_2$ to vary, the function $\overline{H}\colon \mathbb{R}^d\times\mathbb{R}^d \times \mathcal{P}_2 \to \mathbb{R}$ satisfies the following growth and regularity estimates: there exists a constant $C_{\overline{H}}=C_{\overline{H}}(C_{1,H},C_{2,H}, C_H, K_0) > 0$ such that
\begin{align}
\bar{H}(x,p,m) &\text{ is convex in }p,\\
    C_{1,H} |p|^2 -K_0
    &\le \overline{H}(x,p,m)
    \le C_{2,H} |p|^2 + K_0, \label{eqn:Hbarmgrowth}\\[1mm]
    \big|\overline{H}(x, p,m) - \overline{H}(x', p',m')\big|
    &\le C_{\overline{H}}\big(1 + |p| + |p'|\big)\left(|x-x'|+|p - p'|+{\bf d}_2(m,m')\right),\label{eqn:Hbarmreg}
\end{align}
for all $x,x',p,p' \in \mathbb{R}^d, m, m' \in \mathcal{P}_2$, where $C_{1,H}, C_{2, H}, C_H, K_0>0$ are the constants appearing in
\eqref{quadgrowthHm} and \eqref{ineqHm}.
\end{proposition}

\begin{proposition}\label{prop:Lbarmprop}
 Let $\bar{L}: \mathbb{R}^d \times \mathbb{R}^d \times \mathcal{P}_2\rightarrow \mathbb{R}$ be the Legendre transform of $\bar{H}$, defined in \eqref{effLagr} by
\[
\bar{L}(x,v,m) = \sup_{p \in \R^d} \big\{ -p \cdot v - \bar{H}(x,p,m) \big\}.
\]
Then, the following statements hold.
\begin{enumerate}
    \item[$\rm{(1)}$] \(\bar{L}\) is finite everywhere on \(\mathbb{R}^d \times\mathbb{R}^d \times \mathcal{P}_2 \).
    \item[$\rm{(2)}$] For any \(x \in \R^d, m \in \mathcal{P}_2\), the map \(v \mapsto \bar{L}(x, v, m)\) is convex.
    \item[$\rm{(3)}$] For any $x, v \in \mathbb{R}^d$, and \(m \in \mathcal{P}_2\),
    \begin{equation}\label{eqn:Lbarmquad}
    \frac{|v|^2}{4C_{2,H}} -K_0 \leq \bar{L}(x, v, m) \leq \frac{|v|^2}{4C_{1,H}} +K_0.       \end{equation}

    \item[$\rm{(4)}$] There exists a constant $C_{\bar{L}}=C_{\bar{L}}(C_{1,H},C_{2,H}, C_H,K_0)>0$ such that for all
    $x, x'\in \R^d$, $v,v'\in \R^d$, and $m,m'\in \mathcal{P}_2$,
    \begin{equation}\label{eqn:LbarmLip}
    \left|\bar{L}(x,v,m)-\bar{L}(x', v',m')\right|\le C_{\bar{L}}(1+|v|+|v'|)(|x-x'|+|v-v'|+{\bf d}_2(m,m')).
    \end{equation}
\end{enumerate}
\end{proposition}
\vspace{1mm}

\subsubsection{Properties of the cost function $h^\e$}
To study the convergence rates later, we introduce the cost function $h^\e$ associated with the Lagrangian $L$. Under Assumption (A1'), $h^\e$ was given in \eqref{cost} and represents the minimal action required to travel from $x$ to $y$ over the time interval $[t,T]$. We give the following more general definition for the time interval $[t_1,t_2]$.

\begin{definition}\label{def:costf} (1) Assume (A1). For $\ep>0$, define $h^\e:\R^d \times \mathcal{P}_2 \times [0,T] \times [0,T] \times \R^d \times \R^d \to \mathbb{R}$ by
\begin{equation}\label{def:mep}
h^\e(c, \mu; t_1, t_2, x,y)=\inf_{\substack{\gamma\in {\rm AC}\left([t_1,t_2];\mathbb{R}^d\right)\\
\gamma(t_1)=x,\gamma(t_2)=y }}\int^{t_2}_{t_1}L\left(\frac{\gamma(t)}{\varepsilon}, c, \dot{\gamma}(t),\mu\right)dt.
\end{equation}
For $\ep=1$, we write
\begin{equation}\label{def:mep1}
h(c,\mu;t_1, t_2, x,y)=\inf_{\substack{\gamma\in \mathrm{AC}\left([t_1,t_2];\mathbb{R}^d\right)\\
\gamma(t_1)=x,\gamma(t_2)=y}}\int^{t_2}_{t_1}L\left(\gamma(t),c, \dot{\gamma}(t),\mu\right)dt.
\end{equation}

\noindent
(2) Assume (A1'). For $\ep>0$, define $h^\e: [0,T] \times [0,T] \times \R^d \times \R^d \to \mathbb{R}$ by
\begin{equation}\label{def:meps}
h^\e(t_1, t_2, x,y)=\inf_{\substack{\gamma\in {\rm AC}\left([t_1,t_2];\mathbb{R}^d\right)\\
\gamma(t_1)=x,\gamma(t_2)=y }}\int^{t_2}_{t_1}L\left(\frac{\gamma(t)}{\varepsilon}, \dot{\gamma}(t)\right)dt.
\end{equation}
For $\ep=1$, we write
\begin{equation}\label{def:meps1}
h(t_1, t_2, x,y)=\inf_{\substack{\gamma\in \mathrm{AC}\left([t_1,t_2];\mathbb{R}^d\right)\\
\gamma(t_1)=x,\gamma(t_2)=y}}\int^{t_2}_{t_1}L\left(\gamma(t),\dot{\gamma}(t)\right)dt.
\end{equation}

\end{definition}

\nd
In both cases, for $\ep>0$, we denote by $h^\e$ the minimal action
associated with the Lagrangian $L$. When $\ep=1$, the cost function $h$
coincides with the cost function $m$ introduced in \cite{TranYu2021}.

\begin{proposition}\label{prop:mmeas}
(1) Assume (A1) and (A2). Fix $\varepsilon>0$, $t_1<t_2$ with $t_1,t_2\in[0,T]$, and $\mu\in\mathcal{P}_2$.
Then the map $(c,x,y)\mapsto h^\e(c, \mu; t_1, t_2, x,y)$ is continuous on $\mathbb{R}^d\times\mathbb{R}^d\times\mathbb{R}^d$.
In particular, it is Borel measurable.

\noindent
(2) Assume (A1') and (A2). Fix $\varepsilon>0$, $t_1<t_2$ with $t_1,t_2\in[0,T]$. Then the map $(x,y)\mapsto h^\e(t_1,t_2,x,y)$ is continuous on $\mathbb{R}^d\times\mathbb{R}^d$. In particular, it is Borel measurable.
\end{proposition}
\begin{proof} We only prove part (1), as the proof of part (2) is analogous.

Let $x, x',y,y',c,c' \in \R^d$. Then there exists a minimizer $\gamma^\ast: [t_1,t_2] \to \mathbb{R}^d$ for $h^\e(c,\mu;t_1, t_2, x,y)$ (see \cite[Lemma D.2]{tran_hamilton-jacobi_2021}, \cite{cannarsa2004semiconcave}), that is,

\[
\gamma^\ast \in \operatorname*{argmin}_{\substack{
\gamma \in \mathrm{AC}\bigl([t_1,t_2];\mathbb{R}^d\bigr) \\
\gamma(t_1)=x,\ \gamma(t_2)=y
}}\int^{t_2}_{t_1}L\left(\frac{\gamma(t)}{\ep}, c, \dot{\gamma}(t), \mu \right)dt.
\]
Consider the path $\eta:[t_1, t_2] \to \mathbb{R}^d$ defined by $\displaystyle \eta(t)=x+\frac{t-t_1}{t_2-t_1}\left(y-x\right)$. By the definition of $h^\e$ and \eqref{eqn:Lmquad},
\begin{equation}\label{eqn:hupbd}
h^\e\left(c, \mu;t_1, t_2, x, y\right) \leq \int_{t_1}^{t_2} L \left(\frac{\eta(t)}{\ep}, c, \dot{\eta}(t), \mu\right) dt\leq \frac{\left|x-y\right|^2}{4C_{1,H}(t_2-t_1)}+K_0\left(t_2-t_1\right).
\end{equation}
On the other hand, again by \eqref{eqn:Lmquad},
\[
h^\e\left(c, \mu;t_1,t_2, x, y\right)=\int^{t_2}_{t_1}L\left(\frac{\gamma^\ast (t)}{\ep}, c, \dot{\gamma}^\ast(t),\mu \right)dt \geq   \frac{1}{4C_{2,H}}\int_{t_1}^{t_2}\left|\dot{\gamma}^\ast(t)\right|^2 \,dt -K_0(t_2-t_1).
\]
Hence,
\begin{equation}\label{eqn:velbdgst}
\int_{t_1}^{t_2}\left|\dot{\gamma}^\ast(t)\right|^2\,dt \leq \frac{C_{2,H}|x-y|^2}{C_{1,H}(t_2-t_1)}+8C_{2,H}K_0(t_2-t_1).
\end{equation}

Let $\delta=\frac{t_2-t_1}{4}$. Define $\tilde{\gamma}:[t_1, t_2]\to \R^d$ by
\[
   \tilde{\gamma}(t) = \left\{\begin{aligned}
   & \gamma^\ast(t)+\frac{\delta+t_1-t}{\delta}\left(x'-x\right), \quad \,\,\,  \text{ if } t\in \left[t_1, t_1+\delta\right], \\
   &\gamma^\ast(t), \qquad \qquad \qquad \qquad \qquad  \quad \text{ if } t \in \left[t_1+\delta,t_2-\delta\right],\\
   &\gamma^\ast(t)+\frac{t-\left(t_2-\delta\right)}{\delta}\left(y'-y\right), \quad  \text{ if } t\in \left[t_2-\delta, t_2\right].
   \end{aligned}
   \right.
\]
By construction, $\tilde{\gamma}$ is admissible for $h^\e (c',\mu;t_1,t_2, x',y')$. Therefore,
\[
\begin{aligned}
&h^\e (c',\mu;t_1,t_2, x',y')-h^\e (c,\mu;t_1,t_2, x,y) \\
&\qquad \qquad\qquad\qquad\le
\int_{t_1}^{t_2} L\left(\frac{\tilde\gamma(t)}{\ep},c',\dot{\tilde\gamma}(t), \mu\right)\,dt
-\int_{t_1}^{t_2} L\left(
\frac{\gamma^\ast(t)}{\ep},c,
\dot{\gamma}^\ast(t),\mu
\right)\,dt \\
&\qquad\qquad\qquad\qquad\le \mathrm{I}+\mathrm{II}+\mathrm{III},
\end{aligned}
\]
where
\[
\mathrm{I}
:=\int_{t_1}^{t_1+\delta}
\Bigg|
L\left(
\frac{
\gamma^\ast(t)
+\frac{\delta+t_1-t}{\delta}(x'-x)
}{\ep},c',\dot{\gamma}^\ast(t)-\frac{x'-x}{\delta},\mu
\right)
- L\left(
\frac{\gamma^\ast(t)}{\ep},c,
\dot{\gamma}^\ast(t),\mu
\right)
\Bigg|\,dt ,
\]

\[
\mathrm{II}
:=\int_{t_2-\delta}^{t_2}
\Bigg|
L\left(
\frac{
\gamma^\ast(t)
+\frac{t-(t_2-\delta)}{\delta}(y'-y)
}{\ep},c',
\dot{\gamma}^\ast(t)+\frac{y'-y}{\delta},\mu
\right)
- L\left(
\frac{\gamma^\ast(t)}{\ep},c,
\dot{\gamma}^\ast(t), \mu
\right)
\Bigg|\,dt,
\]
and
\[
\mathrm{III}:=\int_{t_1+\delta}^{t_2-\delta}
\Bigg|
L\left(
\frac{
\gamma^\ast(t)}{\ep},c',\dot{\gamma}^\ast(t),\mu
\right)- L\left(
\frac{\gamma^\ast(t)}{\ep},c,
\dot{\gamma}^\ast(t),\mu
\right)
\Bigg|\,dt.
\]
By \eqref{eqn:LmLip} and  \eqref{eqn:velbdgst},
\[
\begin{aligned}
\mathrm{I} &\leq C_L \left(\left(\frac{1}{\ep}+\frac{1}{\delta}\right) \left|x'-x\right| +|c'-c|\right) \int_{t_1}^{t_1+\delta} \left(1+\frac{|x'-x|}{\delta}+2\left|\dot{\gamma}^\ast(t)\right|\right)dt\\
&\leq  C_L \left(\left(\frac{1}{\ep}+\frac{1}{\delta}\right) \left|x'-x\right| +|c'-c|\right)\left(\delta +|x'-x|+2\delta^\frac{1}{2} \left(\int_{t_1}^{t_1+\delta}\left|\dot{\gamma}^\ast(t)\right|^2 \,dt\right)^\frac{1}{2} \right)\\
&\leq C\left(\left|x'-x\right|+|c'-c|\right)\left(1+\left|x'-x\right|+|x-y|\right),
\end{aligned}
\]
for some constant $C=C(\ep,C_{1,H}, C_{2, H}, C_H, K_0, T-t_1)>0$.
Similarly,
\[
\mathrm{II}\leq C\left(\left|y'-y\right|+|c'-c|\right)\left(1+\left|y'-y\right|+|x-y|\right).
\]
Again by \eqref{eqn:LmLip} and  \eqref{eqn:velbdgst},
\[
\begin{aligned}
\mathrm{III} &\leq C_L |c'-c| \int_{t_1+\delta}^{t_2-\delta}\left(1+2\left|\dot{\gamma}^\ast(t)\right|\right) \,dt \leq C_L |c'-c| \int_{t_1}^{t_2}\left(1+2\left|\dot{\gamma}^\ast(t)\right|\right) \,dt \\
& \leq C_L |c'-c| \left((t_2-t_1)+2\left(\int_{t_1}^{t_2}\left|\dot{\gamma}^\ast(t)\right|^2\,dt\right)^\frac{1}{2}(t_2-t_1)^\frac{1}{2}\right)\\
& \leq C|c'-c|\left(1+|x-y|\right),
\end{aligned}
\]
for some constant $C=C(\ep,C_{1,H}, C_{2, H}, C_H, K_0, t_2-t_1)>0$.
Combining these estimates, we obtain
\[
\begin{aligned}
&h^\e (c',\mu;t_1,t_2, x',y')-h^\e (c,\mu;t_1,t_2, x,y)\\
&\qquad \qquad\qquad\leq C \left(\left|x'-x\right|+\left|y'-y\right|+|c'-c|\right)\left(1+\left|x'-x\right|+\left|y'-y\right|+|x-y|\right).
\end{aligned}
\]
By symmetry, exchanging $(x, y)$ and $(x', y')$ gives
\[
\begin{aligned}
&h^\e\left(c,\mu;t_1, t_2, x, y\right)-h^\e\left(c',\mu;t_1, t_2, x', y'\right)\\
&\qquad \qquad\qquad\leq C \left(\left|x'-x\right|+\left|y'-y\right|+|c'-c|\right)\left(1+\left|x'-x\right|+\left|y'-y\right|+|x'-y'|\right).
\end{aligned}
\]
Therefore,
\begin{equation}\label{eqn:mlip}
\begin{aligned}
&\left|h^\e(c',\mu;t_1,t_2, x',y')-h^\e(c,\mu;t_1,t_2, x,y)\right|\\
& \qquad \qquad\leq C \left(\left|x'-x\right|+\left|y'-y\right|+|c'-c|\right)\left(1+\left|x'-x\right|+\left|y'-y\right|+|x-y|+|x'-y'|\right),
\end{aligned}
\end{equation}
where $C=C(\ep,C_{1,H},C_{2,H}, C_H,K_0,t_2-t_1)>0$. In particular, the map $(c,x,y)\mapsto h^\e(c,\mu;t_1,t_2, x,y)$ is continuous, and hence measurable.
\end{proof}

For the homogenization problem in $\R^d$ considered in \cite{TranYu2021}, the quantitative convergence argument relies on establishing both the subadditivity and the superadditivity of the cost function $h$. Subadditivity, together with Fekete’s lemma, implies the existence of the averaged metric $\overline{h}$. Combining subadditivity and superadditivity with $\overline{h}$ then yields an $O(\ep)$ rate of convergence. These results are summarized in the following proposition, which we will use to establish the uniform convergence of $U^\ep$ to its limit, as well as the corresponding rate of convergence.

\begin{proposition}[Restatement of \cite{TranYu2021}, Lemma 3.1, 3.2]\label{prop:subsuperm}
Let $x, y \in \R^d$ and $0\leq t_1 < t_2 \leq T$.

\nd
(1) (Subadditivity) For any \(k,l> 0\), there exists a constant \(C=C(d, C_{1, H}, C_{2, H}, K_0)>0\) independent of \(t_1, t_2, x , y\) such that
    \[
    \begin{aligned}
    &h\bigl((k+l)t_1,(k+l)t_2,(k+l)x,(k+l)y\bigr) \\
    & \qquad \qquad \qquad \qquad \le h\left(kt_1,kt_2,kx,ky\right)+ h\left(lt_1, lt_2, lx,ly\right) + C\left(1+\frac{|x-y|^2}{(t_2-t_1)^2}\right).
    \end{aligned}
    \]
    For the case \(k=l=1\), we have
    \[
  h\bigl(2t_1,2t_2, 2x,2y\bigr) \le 2h\bigl(t_1,t_2,x,y\bigr)+ C\left(1+\frac{|x-y|^2}{(t_2-t_1)^2}\right).
    \]

\nd
 (2) (Superadditivity) There exists a constant \(C=C(d, C_{1, H}, C_{2, H}, K_0)>0\) independent of \(t_1,t_2,x,y\) such that
        \[
  2h\bigl(t_1,t_2,x,y\bigr) \leq h\bigl(2t_1,2t_2, 2x,2y\bigr) + C\left(1+\frac{|x-y|^2}{(t_2-t_1)^2}\right).
    \]

\nd
(3) The limit
\[
\bar{h}\left(t_1,t_2,x,y\right):=\lim_{\ep\to 0^+} \ep h\left(\frac{t_1}{\ep},\frac{t_2}{\ep}, \frac{x}{\ep},\frac{y}{\ep}\right) =\lim_{\ep\to 0^+}  h^\e (t_1, t_2, x,y )
\]
exists, and
\[
\left|\bar{h}\left(t_1,t_2,x,y\right)-h^\e (t_1, t_2, x,y )\right|\leq C \left(1+\frac{|x-y|^2}{(t_2-t_1)^2}\right)\ep,
\]
for some constant \(C=C(d, C_{1, H}, C_{2, H}, K_0)>0\) independent of \(t_1,t_2,x , y\).
\end{proposition}

For the reader's convenience, we provide proofs of $\rm{(1)}$ and $\rm{(3)}$ in Appendix~\ref{proofofprelim}. The proof of $\rm{(2)}$ can be found in \cite{TranYu2021}.

\section{Solutions to 1st order HJ equations in the Wasserstein Space}\label{viscositysolutions}

\nd
In this section, we provide the definition of $L$-viscosity solutions of \eqref{HJBe} and we prove some regularity properties that they possess. As mentioned in the introduction, apart from being the dynamic programming equations to stochastic optimal control problems, Hamilton--Jacobi equations in the Wasserstein space of the form \eqref{HJBe} arise from finite dimensional problems of the form
\be\label{HJB2}\tag{$\mathrm{HJB}_N$}
    \begin{cases}
    -\partial_tV^N_\e+\frac{1}{N}\sum_{i=1}^NH(x^i/\e,x^i,ND_{x^i}V^N_\e,m^N_{\bm x})=0, &(t,\bm x)\in (0,+\infty)\times (\R^d)^N,\\
    V^N_\e(T,\bm x)= G(m_{\bm x}^N), & \bm x\in (\R^d)^N,
\end{cases}
\ee
where $N\in\mathbb{N}$, as $N\to +\infty$. The latter equation  is understood in the viscosity sense (see \cite{crandall1992user}). Given a solution $V^N_{\e}$ of \eqref{HJB2}, we may establish uniform equicontinuity for the sequence $(V^N_\e)_{N\in\mathbb{N}}$.
\vspace{1mm}

\begin{proposition}\label{Nregularity}
Assume that (A1) and (A2) hold. Fix $\e>0$ and consider $V^{N}_\e$ the unique viscosity solution of \eqref{HJB2} with  $H$ substituted by $H^\e(x,y,p,m)=H\left(\frac{x}{\e},y,p,m\right)$. There exists a constant $C_0>0$ independent of $\e$ and a constant $C>0$, which might depend on $\e$,  such that for each $N\in \mathbb{N}$, the unique viscosity solution of \eqref{HJB2} satisfies
\begin{align}
&|V^N_\e(t,\bm x)- V^N_\e(s,\bm y)|\le C_0\left( {\bf d}_{2}(m_{\bm x}^N,m_{\bm y}^N)+|t-s|\right),\;\;\text{and}\label{epsilonunifinorm}\\[1ex]
&|V^N_\e(t,\bm x)- V^N_\e(s,\bm y)|\le C\left( {\bf d}_{2-r}(m_{\bm x}^N,m_{\bm y}^N)+|t-s|\right)\label{notuniform},
\end{align}
for all $t,s\in[0,T]$ and $\bm x,\bm y\in (\R^d)^N$.
\end{proposition}

\begin{proof}
    We give the proof in Appendix \ref{tech}.
\end{proof}

\nd
We note that analogous regularity estimates when $\e=1$, that is when $H^\e=H$, can be found in \cite[Theorem 3.3]{gangbo2021finite} or \cite[Proposition 2.5]{daudin2025error}.
\vspace{3mm}

\nd
To find the relation between a solution of \eqref{HJB2} and a solution of \eqref{HJBe}, one exploits the uniform equicontinuity of $(V^{N})_{N\in\mathbb{N}}$ provided by Proposition \ref{Nregularity}, to show that the family of functions $(U^N_\e)_{N\in\mathbb{N}}$ with $U^N_\e\colon [0,T]\times \mathcal{P}_{2-r}\to \R$ and
\be\label{extension}
U^N_\e(t,m):= \inf_{\bm x\in (\R^d)^N}\left\{ V^N_\e(t,\bm x)+C{\bf d}_{2-r}(m^N_{\bm x},m)   \right\},
\ee
where $C$ is the constant in \eqref{notuniform}, is uniformly equicontinuous \footnote{In particular, $U^N_\e$ is ${\bf d}_{2-r}$--Lipschitz with constant $C$.} and equibounded as a sequence in $N$, since $(V_\e^N)_{N\in\mathbb{N}}$ is equibounded. Therefore, we can apply the Arzela--Ascoli theorem on the relatively compact subsets of $\mathcal{P}_{2-r}(\R^d)$:
\be\label{subsets}
M_R^2=\left\{m\in \mathcal{P}_{2-r}(\R^d):\; \int_{\R^d}|x|^2 m(dx)\le R \right\},\; R>0,
\ee
to derive that (up to a subsequence) $U^N_\e$ converges uniformly in $[0,T]\times M_R^2$ to a function $U^\e\colon [0,T]\times \mathcal{P}_2\to \R$, which is also ${\bf d}_{2-r}$-Lipchitz with constant $C$.
For $\e=1$, it was shown
in \cite[Theorem 1.2]{gangbo2021finite} under assumptions (A1), (A2), that for any bounded $B\subset \mathcal{P}_2$ we have
$$\lim_{N\rightarrow +\infty}\sup_{(t,m_{\bm x}^N)\in [0,T]\times B}|U^\e(t,m_{\bm x}^N)-V^N_\e(t,\bm x)|=0,$$
and, furthermore, the limit $U^\e$ satisfies \eqref{HJBe} in the ``lifted'' viscosity sense of P.L Lions ($L$-viscosity solutions).
\vspace{1mm}

\nd
Before introducing the definition of $L$-viscosity solutions, we collect the results of the above discussion together with a uniform in $\e$ ${\bf d}_2$-Lipschitz estimate for $U^\e$. We remark that we do not get a uniform in $\e$ ${\bf d}_{2-r}$-Lipschitz estimate as the constant $C$ in \eqref{notuniform} from Proposition \ref{Nregularity} depends on $\e$.

\begin{corollary}\label{cor:lipschitzestimate}
    Assume that (A1) and (A2) are in place and $\e>0$ is fixed. Let $U^N_\e\colon [0,T]\times\mathcal{P}_{2-r}\to \R$, $N\in \mathbb{N}$, be as in \eqref{extension}. Then, the sequence $(U^N_\e)_{N\in \mathbb{N}}$ converges (up to a subsequence) uniformly in the sets $[0,T]\times M_R^2$, where $M_R^2$ is as in \eqref{subsets}, to a function $U^\e\colon[0,T]\times\mathcal{P}_2\to \R$. The function $U^\e$ is an $L$-viscosity solution of \eqref{HJBe} and, furthermore,
    \be\label{uniforminepsilon}
|U^\e(t,m_1)-U^\e(s,m_2)|\le C_0\left( |t-s|+{\bf d}_2(m_1,m_2) \right),
    \ee
    for all $t,s\in [0,T]$ and $m_1,m_2\in\mathcal{P}_2$, where $C_0$ is the constant from \eqref{epsilonunifinorm}.
\end{corollary}

\begin{proof}
    We give the proof in Appendix \ref{tech}.
\end{proof}

\begin{remark}
    The above estimate is the infinite dimensional analogue of the equi-Lipschitz property of the family $(u^\e)_{\e>0}$, where $u^\e$ solves \eqref{fdHJBe}, obtained in \cite{LionsPapaVara1987}. In addition, we note that the results of Proposition \ref{Nregularity} and Corollary \ref{cor:lipschitzestimate} hold even without the convexity assumption on the map $p\mapsto H(y,x,p,m)$.
\end{remark}



\subsection{$L$-viscosity solutions} Let $(\Omega,\mathcal{F},\mathbb{P})$ be a non-atomic probability space. We consider $\mathcal{H}=L^2(\Omega;\R^d)$, which is a separable Hilbert space endowed with the inner product $\langle X,Y \rangle=\mathbb{E}[XY]$, $X,Y\in\mathcal{H}$. Its norm is denoted by $\|X\|_2=\left(\mathbb{E}[X^2]\right)^{1/2}$.

\begin{definition}
    For a given $\Phi:\mathcal{P}_2\rightarrow \R^d$, we denote by $\hat{\Phi}:\mathcal{H}\rightarrow \R$ its ``lift'' defined by
    $$\hat{\Phi}(X)=\Phi(\mathcal{L}(X)).$$
\end{definition}

\begin{remark}
Recall that for a function $\hat{\Phi}:\mathcal{H}\rightarrow \R$ its gradient $D \hat{\Phi}(X)\in \mathcal{H}$ satisfies
$$\hat{\Phi}(X+Y)=\hat{\Phi}(X)+\langle D\hat{\Phi}(X), Y\rangle+o(\|Y\|_2),\;\; X,Y\in\mathcal{H}.$$
If $\Phi$ is smooth enough (see \cite{carmona2018probabilistic}), the gradient $D\hat{\Phi}$ is related to the Wasserstein derivative $D_m\Phi$ in the following way
$$D_m\Phi(\mathcal{L}(X),X)=D\hat{\Phi}(X),\;\; X\in\mathcal{H}.$$
\end{remark}
\vspace{1mm}

\nd
With this in mind, we notice that, at least formally, if $U$ satisfies \eqref{HJBe}, then its ``lift'' $\hat{U}:[0,T]\times\mathcal{H}\rightarrow \R$ satisfies the lifted equation
\be\label{HJB3}\tag{$\mathrm{HJB}_\mathcal{H}$}
\begin{cases}
    -\partial_t \hat{U} +\hat{H}^\e(X,D\hat{U}(t,X))=0, &(t,X)\in [0,T)\times \mathcal{H},\\
    \hat{U}(T,X)=\hat{G}(X), &X\in\mathcal{H},
\end{cases}
\ee
where
$$\hat{H}^\e(X,P)=\mathbb{E}[H(X/\e,X,P,\mathcal{L}(X))],\quad \hat{G}(X)=G(\mathcal{L}(X)),\;\; \text{for }X,P\in \mathcal{H}.$$

\begin{definition}\label{L-visc}
    An upper semicontinuous (resp. lower semicontinuous) function $V:[0,T]\times \mathcal{H}\rightarrow \R$ is a viscosity subsolution (resp. supersolution) of \eqref{HJB3} if
    \begin{itemize}
        \item $V(T,X)\le (\text{resp. }\ge)\; \hat{G}(X)$;
        \item For any $\Phi\in C^{1,2}([0,T]\times\mathcal{H})$ and each $(t_0,X_0)\in [0,T)\times \mathcal{H}$ such that $V(t_0,X_0)-\Phi(t_0,X_0)$ has a local maximum (resp. local minimum), we have
        $$-\partial_t \Phi(t_0,X_0)+\hat{H}^\e(X_0,D\Phi(t_0,X_0))\le (\text{resp. }\ge)\; 0.$$
    \end{itemize}
    $V$ is a viscosity solution of \eqref{HJB3} if it is a viscosity subsolution and a viscosity supersolution.
\end{definition}


\begin{definition}\label{visc.}
    A function $U^\e\colon [0,T]\times\mathcal{P}_2\rightarrow \R$ is a viscosity solution/subsolution/supersolution of \eqref{HJBe} if its lift $\hat{U}^\e$ is a viscosity solution/subsolution/supersolution of \eqref{HJB3}.
\end{definition}
\vspace{1mm}

\nd
From now on, when refering to viscosity solutions of \eqref{HJB} or \eqref{HJBe}, we mean in the sense of Definition \ref{visc.}.
It was proved in \cite[Theorem 7.4]{gangbo2021finite} that there is comparison between bounded viscosity sub and super solutions of \eqref{HJB3} given our assumptions (A1) and (A2). Furthermore, using the dynamic programming principle from \cite{djete2022mckean1}, it can be shown that the value function of the mean field control problem \eqref{eqn:U_value_function} is a viscosity solution of \eqref{HJBe}. We record this in the following Theorem.

\begin{theorem}\label{comprep}
    Assume that $H,G$ satisfy assumptions (A1)-(A2) or (A1')-(A2). Then, \eqref{HJBe} has a unique $L$-viscosity solution $U^\e$ which is given by \eqref{eqn:U_value_function} for $L$ as in \eqref{Lagr}. Furthermore, if $\overline{H}$ is the corresponding effective Hamiltonian constructed in Proposition \ref{effHm}, the equations \eqref{HJB} and \eqref{HJB4} have unique $L$-viscosity solutions.
\end{theorem}

\nd
We finish the section with the following remark.

\begin{remark}\label{rem1}
    (i) Assume that $\e=1$. If $U:=U^{\e=1}$ is a classical solution of \eqref{HJBe}, that is the Wasserstein $D_mU$ and time $\partial_tU$ derivatives exist and \eqref{HJBe} holds everywhere, then we can verify that $U(t,m_{\bm x}^N)=:u^N(t,\bm x)$ is a classical solution of \eqref{HJB2}. Assuming uniqueness of solutions of \eqref{HJB2}, it follows that $u^N(t,\bm x)=V^N(t,\bm x)$.
    \\[0.5ex]
    (ii) The PDE \eqref{HJBe} does not always admit a classical solution. For example, if $H(y,x,p,m)=\frac{|p|^2}{2}$ and $G(m)=g(\int x m(dx))$ for some $g\in C^\infty_c(\R^d)$, then the L-viscosity solution of \eqref{HJBe} is $U(t,m)= u(t, \int x m(dx))$, where $u$ is the unique viscosity solution of
    \be\nonumber
    \begin{cases}
        -\partial_t u +\frac{1}{2}|Du|^2=0, &(t,x)\in [0,T)\times \R^d,\\
        u(T,x)=g(x), &x\in \R^d,
    \end{cases}
    \ee
    which does not have classical solutions in general. The proof of this fact is similar to \cite[Proposition 2.10]{daudin2024optimal} which uses the optimal control formulation of the solution. \\[0.5ex]
(iii) By considering periodic data, the construction presented in this section can be transferred to $\T^d$ (instead of $\R^d$).
\end{remark}

\section{Different types of admissible controls}
\label{section4}
\subsection{Alternative control sets}\label{altcontrset}

In this section, we introduce several classes of admissible controls that will be used to establish the rate of convergence.

Recall that $(\Omega, \mathcal{F}, \mathbb{F}=(\mathcal{F}_t)_{t\geq 0},\mathbb{P})$ is a sufficiently rich, atomless filtered probability space, and let $B$ be an $\mathbb{F}$-Brownian motion. Let $\xi:\Omega \to \mathbb{R}^d$ be a Gaussian random variable, independent of $B$, such that $\sigma(\xi) \subset \mathcal{F}_0$. Let $\ep>0$ and $T>0$. For fixed $(t,m) \in [0,T] \times \mathcal{P}_2$, consider the control problem
\begin{equation} \label{eqn:control}
\inf_{\alpha\in \mathcal{A}} \left\{\mathbb{E}\left[\int_t^T L \left(\frac{X_s}{\ep},X_s,\alpha_s, \mathcal{L}(X_s)\right)\,ds\right]+G(\mathcal{L}(X_T))\right\},
\end{equation}
where $\mathcal{A}=\left\{ \alpha : [t,T]\times \Omega \to \mathbb{R}^d :
\alpha \text{ is } \mathbb{F}\text{-progressively measurable and } \mathbb{E}\left[ \int_t^T|\alpha_s|^2 \, ds \right] <\infty\right\}$. For each $\alpha \in \mathcal{A}$ and each $\mathcal{F}_t$-measurable random variable
$X_t : \Omega \to \mathbb{R}^d$ with $\mathcal{L}(X_t)=m$, independent of $B$, consider the controlled process $X$ that satisfies $X_s(\omega)=X_t(\omega)+\int_t^s \alpha(r,\omega)dr$ a.s.. The associated value function is $U^\e(t,m)$.




We now introduce two alternative admissible control sets.

\nd
(1) Define \[
\mathcal{A}_1:=\left\{\alpha:[t,T]\times \Omega\rightarrow \mathbb{R}^d: \alpha \text{ is } \mathcal{B}([t,T])\otimes \mathcal{F} \text{-measurable and } \mathbb{E}\left[ \int_t^T|\alpha_s|^2 \,ds \right] <\infty\right\}.\]
For each $\alpha \in \mathcal{A}_1$ and each $\mathcal{F}_t$-measurable random variable
$X_t : \Omega \to \mathbb{R}^d$ with $\mathcal{L}(X_t)=m$, independent of $B$,
consider the controlled process $X$ as above. The corresponding value function in \eqref{eqn:control} associated with $\mathcal{A}_1$
is denoted by $U_1(t,m)$.

With this choice of controls, for any $\nu\in\mathcal{P}_2(\R^d)$, we can generate any coupling $\pi\in \Pi(m,\nu)$ as $\pi=\mathcal{L}((X_t,X_T))$.

\nd
(2) Define
\[
\mathcal{A}_2
:=
\left\{
\alpha : [t,T]\times \Omega \to \mathbb{R}^d \;\left|\;
\begin{aligned}
&\alpha_s(\omega)
= \tilde{\alpha}\left(s,\xi(\omega)\right),\mathbb{E}\left[\int_t^T \lvert \alpha_s\rvert^2 \,ds \right] < \infty, \\
&\tilde{\alpha} : [t,T]\times \mathbb{R}^d \to \mathbb{R}^d
\text{ Borel measurable}, \\
\end{aligned}
\right.
\right\}.
\]
For each $\alpha \in \mathcal{A}_2$ and each $\mathcal{F}_t$-measurable random variable
$X_t : \Omega \to \mathbb{R}^d$ with $\mathcal{L}(X_t)=m$, independent of $B$,
consider the controlled process $X$ as above. Denote the corresponding value function by $U^\e_2(t,m)$.

In this setting, admissible controls are measurable with respect to the independent Gaussian variable $\xi$. This additional randomness allows the control to split mass. Our aim is to prove that $U^\varepsilon = U_1^\varepsilon$. However, a direct comparison is difficult because progressively measurable controls are technically delicate to handle. To overcome this difficulty, we introduce the intermediate value function $U_2^\varepsilon$, which provides a more tractable formulation and simplifies the proof. This is a randomization technique.

Since $\mathcal{A}_1 \subset \mathcal{A} \subset\mathcal{A}_2$, it follows that
\[
U^\ep_1(t,m) \leq U^\ep(t,m) \leq U^\ep_2(t,m).
\]
We will prove that under assumptions (A1) and (A2), $U_1^\ep=U^\ep=U_2^\ep$. Before proving this, we first provide equivalent formulations of $U_1^\varepsilon$, which will be useful in the subsequent analysis.

\subsection{Equivalent formulations of $U^\ep_1$}\label{formulations}
To derive the equivalent formulations of $U^\e_1$, we first define two admissible sets. Let $(t,m) \in [0,T] \times \mathcal{P}_2$.
\begin{enumerate}
\item For any $\nu\in \mathcal{P}_2$, we define
\[
\mathcal{E}_{m,\nu}
:=
\left\{
(X,Y)\in L^2(\Omega,\mathcal{F};\mathbb{R}^d)
\times
L^2(\Omega,\mathcal{F};\mathbb{R}^d): \mathcal{L}(X)=m,\mathcal{L}(Y)=\nu
\right\}.
\]

\item For $(X, Y) \in \mathcal{E}_{m, \nu}$, we define the set of admissible paths as
\begin{equation}\label{eqn:gammaxy}
\Gamma(X, Y) = \left\{\,\gamma:[t,T]\times \Omega \to \mathbb{R}^d\;\middle|\;
\begin{array}{l}
\gamma \text{ is } \mathcal{B}([t,T]) \otimes \mathcal{F}\text{-measurable},\\[4pt]
\gamma(\cdot,\omega)\in \text{AC}([t,T];\mathbb{R}^d) \,\,\, \forall \, \omega \in \Omega,\\[4pt]
 \gamma(t,\cdot)=X,\quad \gamma(T,\cdot)=Y
\end{array}
\right\}.
\end{equation}
\end{enumerate}
Then, we define a new value function $\hat{U}_1^\ep$.

\begin{definition}
Define $\tilde{U}_1^\varepsilon : [0,T] \times \mathcal{P}_2 \to \mathbb{R}$ by
\begin{equation}\label{eqn:Uep1h}
\tilde{U}_1^\varepsilon(t,m)
:= \inf_{\nu \in \mathcal{P}_2}
\left\{ \inf_{(X, Y) \in \mathcal{E}_{m,\nu}}
\left\{
\inf_{\gamma \in\Gamma(X, Y)}
\mathbb{E}\left[
\int_{t}^T
L\left(
\frac{\gamma_s}{\e},
\gamma_s,
\dot{\gamma}_s,
\mathcal{L}(\gamma_s)
\right)\, ds
\right]
\right\}+G(\nu)
\right\},
\end{equation}
where $\gamma_s:=\gamma(s,\cdot)$.
\end{definition}

\nd
Although the definitions of $U^\varepsilon_1$ and $\tilde{U}^\varepsilon_1$ seem very different, we will prove the equivalence of these two definitions, meaning $U^\e_1=\tilde{U}^\e_1$. The key observation is that the value functions of the optimal control problem in both~\eqref{eqn:control} and~\eqref{eqn:Uep1h} {are functions} of the optimal path. Thus, as long as both formulations allow similar sets of admissible paths, the optimal value is the same.
\vspace{1mm}

\nd
Before proving the equivalence between $U^\e_1$ and $\tilde{U}^\e_1$, we establish a velocity bound for admissible paths $\gamma$. This estimate shows that, for nearly optimal choices of $\nu$ and $(X,Y)$, the expected squared distance between $X$ and $Y$ is controlled by $(T-t)^2$. In particular, if $\gamma^\delta$ is a $\delta$-optimal path with $0<\delta<T-t$, then the expected $L^2$-norm of its velocity is bounded.
\begin{proposition}\label{prop:Uepm2new}
(1) For any $t \in [0, T]$ and $m\in \mathcal{P}_2$,
\begin{equation}\label{eqn:Uep1hnew}
\tilde{U}_1^\e(t,m)
= \inf_{\nu \in \mathcal{P}_2} \left\{ \inf_{(X, Y) \in \mathcal{E}^\ast_{m,\nu}} \left\{ \inf_{\gamma \in \Gamma(X, Y)} \mathbb{E}\left[\int_{t}^T L\left( \frac{\gamma_s}{\e}, \gamma_s, \dot{\gamma}_s, \mathcal{L}(\gamma_s) \right)\, ds \right] \right\}+G(\nu) \right\},
\end{equation}
where
\begin{equation}\label{eqn:estar}
\mathcal{E}^\ast_{m,\nu}=\left\{(X, Y) \in \mathcal{E}_{m,\nu}: \mathbb{E}\left|X-Y\right|^2 \leq M_0 (T-t)^2\right\},
\end{equation}
for some constant $M_0=M_0(C_{2,H}, C_G, K_0)>0$ that is independent of $m, \nu, X,Y, t$, and $T$.

\noindent
(2) Let $t \in [0, T)$ and $\delta \in \left(0, T-t\right)$. Suppose $(\nu_\delta,X^\delta, Y^\delta,\gamma^\delta)$ is a $\delta$-minimizer for $\tilde{U}^\e_1(t,m)$, that is
\begin{equation}\label{eqn:U1opdel}
\mathbb{E}\left[
\int_{t}^T
L\left(
\frac{\gamma^\delta_s}{\e},
\gamma^\delta_s,
\dot{\gamma}^\delta_s,
\mathcal{L}(\gamma^\delta_s)
\right)\, ds
\right] +G(\nu_\delta) \leq \tilde{U}_1^\ep(t, m)+\delta.
\end{equation}
Then there exists a constant $M_1=M_1(C_{2, H}, C_G, K_0)>0$ such that
\[
\mathbb{E}\left[\int_{t}^T|\dot{\gamma}^\delta_s|^2 \,ds\right] \leq M_1(T-t).
\]
\end{proposition}

\begin{remark}
In the second part of Proposition~\ref{prop:Uepm2new}, $\delta<T-t$ is a technical assumption to ensure $M_1$ is a constant independent of $\delta$. This condition can always be satisfied when we consider a path that is nearly optimal.
\end{remark}

\begin{proof}
If $t=T$, then by the definition of $\Gamma(X,  Y)$ in \eqref{eqn:gammaxy}, we have $\gamma(t,\omega)=\gamma(T,\omega)$ for all $\omega\in \Omega$. Consequently, the admissible pair $(X, Y)$ must satisfy $X=Y$, and the only admissible measure $\nu \in \mathcal{P}_2$ is $\nu=m$. Hence, $\tilde{U}_1^\ep\left(t, m\right)=G(m)$. The same conclusion holds for the value function defined on the right-hand side of \eqref{eqn:Uep1hnew}. Hence, \eqref{eqn:Uep1hnew} is valid when $t=T$.

Suppose $t<T$. By \eqref{eqn:Uep1h}, choosing $\nu=m, (X, X) \in \mathcal{E}_{m, m}$, $ \gamma (s,\cdot) =X$ for all $s \in [t,T]
$, we obtain
\begin{equation}\label{eqn:U1upperbd}
   \tilde{U}_1^\ep\left(t, m\right) \leq \int_{\R^d}\int^T_{t}L\left(\frac{x}{\varepsilon},x, 0,m\right) \, ds\, m(dx)+G(m)\leq K_0(T-t)+ G(m),
\end{equation}
where the second inequality follows from \eqref{eqn:Lmquad}.
Let $\delta \in \left(0, T- t\right)$. Then there exist $\nu_\delta \in \mathcal{P}_2, (X^\delta,Y^\delta) \in \mathcal{E}_{m,\nu_\delta}, \gamma^\delta\in \Gamma(X^\delta,Y^\delta)$ such that $\gamma^\delta \left(t, \cdot\right)=X^\delta$, $\gamma^\delta \left(T, \cdot\right)=Y^\delta$, and
\eqref{eqn:U1opdel} holds. By \eqref{eqn:Lmquad},
\begin{equation} \label{eqn:U1oplowerbddel}
\mathbb{E}\left[
\int_{t}^T
L\left(
\frac{\gamma^\delta_s}{\e},
\gamma^\delta_s,
\dot{\gamma}^\delta_s,
\mathcal{L}(\gamma^\delta_s)
\right)\, ds
\right]  \geq \frac{1}{4C_{2, H}}\mathbb{E}\left[\int_{t}^T|\dot{\gamma}^\delta_s|^2 \,ds\right]- K_0(T-t).
\end{equation}
Combining \eqref{eqn:U1upperbd}, \eqref{eqn:U1opdel}, \eqref{eqn:U1oplowerbddel}, we obtain
\[
\frac{1}{4C_{2, H}}\mathbb{E}\left[\int_{t}^T|\dot{\gamma}^\delta_s|^2 \,ds\right] \leq (2K_0+1)(T-t)+ G(m)-G(\nu_\delta).
\]
Using \eqref{terminalassumption}, this implies,
\begin{equation}\label{eqn:velbdimp}
\begin{aligned}
\mathbb{E}\left[\int_{t}^T|\dot{\gamma}^\delta_s|^2 \,ds\right]
& \leq 4C_{2, H}(2K_0+1)(T-t)+ 4C_{2, H}C_G{\bf d}_2(m,\nu_{\delta})\\
& \leq 4C_{2, H}(2K_0+1)(T-t)+ 4C_{2, H}C_G \left(\mathbb{E}\left|X^\delta-Y^\delta\right|^2\right)^\frac{1}{2}.
\end{aligned}
\end{equation}
By H{\"o}lder’s inequality, for each $\omega \in \Omega$,

\[
    \left|X^\delta(\omega)-Y^\delta(\omega)\right|\leq \int_{t}^T \left|\dot{\gamma}^\delta(t,\omega)\right|ds \leq \left(\int_{t}^T\left|\dot{\gamma}^\delta (t,\omega)\right|^2
ds\right)^\frac{1}{2}(T-t)^\frac{1}{2}.
\]
Combining this estimate with \eqref{eqn:velbdimp} yields
\[
    \frac{\mathbb{E}\left|X^\delta-Y^\delta\right|^2} {(T-t)}\leq 4C_{2, H}(2K_0+1)(T-t)+ 4C_{2, H}C_G \left(\mathbb{E}\left|X^\delta-Y^\delta \right|^2\right)^\frac{1}{2},
\]
which implies
\begin{equation}\label{eqn:M_0def}
    \mathbb{E}\left|X^\delta-Y^\delta \right|^2 \leq M_0 (T-t)^2,
\end{equation}
for some constant $M_0=M_0(C_{2,H}, C_G, K_0)>0$. Since $\delta>0$ is arbitrary, part (1) follows.

Finally, \eqref{eqn:velbdimp} and \eqref{eqn:M_0def} imply that
\[
\mathbb{E}\left[\int_{t}^T|\dot{\gamma}^\delta_s|^2 \,ds\right] \leq M_1 (T-t),   \]
    for some constant $M_1=M_1(C_{2, H}, C_G, K_0, M_0)=M_1(C_{2, H}, C_G, K_0)>0$, which completes the proof of part (2).
\end{proof}
We now prove the equivalence of $U_1^\ep$ and $\tilde{U}_1^\ep$.

\begin{proposition}\label{prop:UeqU2}
     Assume (A1) and (A2), and let $\e>0$. Then $U_1^\ep=\tilde{U}_1^\ep$.
\end{proposition}

\begin{proof}
Step 1: We show $\tilde{U}_1^\ep\leq U_1^\ep$.

Let $\delta>0$. Then there exist $\alpha \in \mathcal{A}_1$ and an $\mathcal{F}_t$-measurable random variable $X_t$ independent of $B$, with $\mathcal{L}(X_t)=m$, such that for the process $X_s$ that satisfies $X_s=X_t+\int_t^s\alpha_r \,dr$ a.s.,
\[
\mathbb{E}\left[\int^T_{t}L\left(\frac{X_s}{\varepsilon},X_s,\alpha_s, \mathcal{L}(X_s)\right)\, ds \right]+G(\mathcal{L}(X_T)) \leq  U_1^\e(t, m)+\delta.
\]
Define $\gamma:[t, T] \times \Omega \to \R^d$ by $\gamma(s,\omega) =X_t(\omega)+\int_t^s \alpha(r,\omega)\,ds$, which is $\mathcal{B}([t,T]) \otimes\mathcal{F}$-measurable. For each $\omega\in \Omega$, $\gamma(\cdot, \omega) \in \mathrm{AC}([t,T];\R^d)$ and $\dot{\gamma}(s, \omega)=\alpha(s,\omega)$ for $s\in[t,T]$. Set $X:=X_{t}$, $Y:=X_T$, and $\nu: = \mathcal{L}(X_T)$. Note that $X, Y \in L^2(\Omega,\mathcal{F};\mathbb{R}^d)$ since $m\in \mathcal{P}_2$ and $\alpha\in \mathcal{A}_1$.Therefore, $(\nu, X,Y, \gamma)$ belong to the admissible sets of $\tilde{U}_1^\ep\left(t, m\right)$ in \eqref{eqn:Uep1h}, and
\[
\begin{aligned}
\tilde{U}_1^\ep(t,m)\leq &\; \mathbb{E}\left[
\int_{t}^T
L\left(
\frac{\gamma_s}{\e},
\gamma_s,
\dot{\gamma}_s,
\mathcal{L}(\gamma_s)
\right)\, ds
\right] +G(\nu) \\
= &\; \mathbb{E}\left[\int^T_{t}L\left(\frac{X_s}{\varepsilon},X_s,\alpha_s, \mathcal{L}(X_s)\right)\, ds \right]+G(\mathcal{L}(X_T)) \leq   U_1^\ep(t, m)+\delta.
\end{aligned}
\]
Letting $\delta \to 0$ yields $\tilde{U}_1^\ep\leq U_1^\ep$.

\noindent
Step 2: We show $U^\ep_1\leq \tilde{U}_1^\ep$.

If $t=T$, then $U^\ep_1(t,m)=G(m)=\tilde{U}_1^\ep\left(t, m\right)$.

Suppose $t<T$. Let $\delta \in \left(0, T- t\right)$. Then there exist $\nu \in \mathcal{P}_2, (X, Y)\in \mathcal{E}^\ast_{m,\nu}, \gamma\in \Gamma(X, Y)$ such that $\gamma \left(t, \cdot\right)=X$, $\gamma \left(T, \cdot\right)=Y$, and
\[
\mathbb{E}\left[
\int_{t}^T
L\left(
\frac{\gamma_s}{\e},
\gamma_s,
\dot{\gamma}_s,
\mathcal{L}(\gamma_s)
\right)\, ds
\right]  +G(\nu) \leq \tilde{U}_1^\ep(t, m)+\delta.
\]
For each $\omega \in \Omega$, since $\gamma(\cdot,\omega)\in \text{AC}([t,T];\mathbb{R}^d)$, $\dot{\gamma}(s,\omega)$ exists for almost every $s \in [t, T]$. Define $\alpha:[t, T] \times \Omega \to \R^d$ by
\[
\alpha(s,\omega)=\left\{
\begin{aligned}
&\dot{\gamma}(s,\omega),\qquad \text{if the derivative exists},\\
&0,\qquad \qquad \quad \, \, \text{otherwise}.
\end{aligned}
\right.
\]
Then $\alpha$ is $\mathcal{B}([t, T]) \otimes \mathcal{F}$-measurable (see, for example, \cite[Theorem 1]{Moshe_1977}). By Proposition~\ref{prop:Uepm2new},

\[
\mathbb{E}\left[\int_{t}^T\left|\alpha_s\right|^2 \,ds\right]=\mathbb{E}\left[\int_{t}^T\left|\dot{\gamma}_s\right|^2\,
ds \right]\leq C,
\]
for some constant $C=C(C_{2,H}, C_G, K_0, T-t)>0$. Hence, $\alpha \in \mathcal{A}_1$. Now define $X_s(\omega):=\gamma (s, \omega)$ for $s \in [t,T]$. Then $X_s$ satisfies $X_s(\omega)=X(\omega) + \int_t^s\alpha(s,\omega)\,ds$ and $\mathcal{L}(X_T)= \mathcal{L}(\gamma_T)=\mathcal{L}(Y)=\nu$. Therefore,
\[
\begin{aligned}
U^\e_1(t, m) &\leq  \mathbb{E}\left[\int^T_{t}L\left(\frac{X_s}{\varepsilon},X_s,\alpha_s, \mathcal{L}(X_s)\right)\, ds \right]+G(\mathcal{L}(X_T))\\
&=\mathbb{E}\left[
\int_{t}^T
L\left(
\frac{\gamma_s}{\e},
\gamma_s,
\dot{\gamma}_s,
\mathcal{L}(\gamma_s)
\right)\, ds
\right] +G(\nu) \leq \tilde{U}_1^\ep(t, m)+\delta.
\end{aligned}
\]
Letting $\delta \to 0$ yields $U_1^\ep\leq \tilde{U}_1^\ep$.
\end{proof}

\nd
In the definition of $\tilde{U}^\e_1$, the infimum over $\gamma$ can be rewritten as the cost function $h$. This opens the door for applying the techniques developed in finite dimensional case~\cite{TranYu2021}. Specifically, with assumptions (A1'), (A2) instead of (A1), (A2), we further reformulate $U_1^\e$ in terms of the cost function $h^\ep$ in the following lemma.

\begin{lemma}\label{lem:Uep1costh}
     Assume (A1') and (A2), and let $\e>0$. Then
\begin{equation}\label{eqn:Uep1costh}
U_1^\ep(t,m)= \inf_{\nu \in \mathcal{P}_2} \left\{
\inf_{(X, Y) \in \mathcal{E}^\ast_{m,\nu}}
\mathbb{E}[h^\e(t,T,X,Y)]+ G(\nu)\right\},
\end{equation}
where $\mathcal{E}_{m,\nu}^\ast$ is defined in \eqref{eqn:estar}, and $h^\varepsilon$ is given in Definition~\ref{def:costf}.
\end{lemma}
Note that the right-hand side of \eqref{eqn:Uep1costh} is well defined since $h^\e(t,T,\cdot,\cdot)$ is Borel measurable by Proposition~\ref{prop:mmeas}.
\begin{proof}
Denote the right-hand side of \eqref{eqn:Uep1costh} by
\[
V^\ep(t,m):=\inf_{\nu \in \mathcal{P}_2} \left\{
\inf_{(X, Y) \in \mathcal{E}^\ast_{m,\nu}}
\mathbb{E}[h^\e(t,T,X,Y)]+ G(\nu)\right\}.
\]
By Proposition \ref{prop:UeqU2}, with assumptions (A1') and (A2),
\begin{equation}\label{eqn:Uep1sim}
U_1^\e(t,m)
= \inf_{\nu \in \mathcal{P}_2} \left\{ \inf_{(X, Y) \in \mathcal{E}^\ast_{m,\nu}} \left\{ \inf_{\gamma \in \Gamma(X, Y)} \mathbb{E}\left[\int_{t}^T L\left( \frac{\gamma_s}{\e}, \dot{\gamma}_s \right)\, ds \right] \right\}+G(\nu) \right\}.
\end{equation}
By the definition of $\Gamma (X,Y)$ and $h^\ep(t,T,X,Y)$, for any $\gamma \in \Gamma(X,Y)$ and each fixed $\omega \in \Omega$,
\[
\int_{t}^T L\left( \frac{\gamma(s,\omega)}{\e}, \dot{\gamma}_s(s,\omega) \right)\, ds  \geq h^\e(t,T,X(\omega),Y(\omega)).
\]
Hence, $U^\e_1(t,m)\geq V^\e(t,m)$. It remains to prove the reverse inequality.

If $t=T$, then any admissible pair ${(X, Y) \in \mathcal{E}^\ast_{m,\nu}}$ satisfies $X=Y$, which implies the only admissible measure $\nu\in \mathcal{P}_2$ is $\nu=m$. Therefore, $U^\e_1(T,m)=G(m)=V^\e(T,m)$.

Suppose $t<T$. Let $\delta \in \left(0, T- t\right)$. Then there exist $\nu \in \mathcal{P}_2, (X,Y)\in \mathcal{E}^\ast_{m,\nu}$ such that
\[
\mathbb{E}\left[h^\e\left(t, T, X, Y\right)\right]+G(\nu) \leq V^\e(t, m)+\frac{\delta}{2}.
\]
Since $(X,Y)\in \mathcal{E}^\ast_{m, \nu}$, $\mathbb{E}\left|X-Y\right|^2 \leq M_0 (T-t)^2$ for some constant $M_0=M_0(C_{2,H}, C_G, K_0)>0$. Lemma \ref{lem:measapp} implies that there exists a Borel measurable path $\gamma:[t,T]\times\mathbb{R}^d\times\mathbb{R}^d\to\mathbb{R}^d$ such that
\[
 \gamma(\cdot, x,y) \in \mathrm{AC}([t, T];\R^d), \quad  \gamma(t,x,y)=x, \quad \gamma(T,x,y)=y, \qquad \text{for all }x, y \in \R^d,
\]
and
\[
\mathbb{E}\left[\int_{t}^{T}
L \left(\frac{\gamma(s,X,Y)}{\varepsilon}, \dot{\gamma}(s,X,Y)\right)\,ds\right] \leq
\mathbb{E} \left[h^\e\left(t, T, X, Y\right)\right]+ \frac{\delta}{2}.
\]
Moreover, $(t,\omega) \mapsto \gamma(t,X(\omega),Y(\omega))$ is $[t,T]\otimes \mathcal{F}$-measurable and admissible for $U^\e_1(t, m)$. Consequently,
\[
\begin{aligned}
U^\e_1(t,m) &\leq \mathbb{E}\left[\int_{t}^{T}
L \left(\frac{\gamma(s,X,Y)}{\varepsilon}, \dot{\gamma}(s,X,Y)\right)\,ds\right]+G\left(\nu\right)\\
& \leq \mathbb{E} \left[h^\e\left(t, T, X, Y\right)\right] + G\left(\nu\right) +\frac{\delta}{2} \leq V^\e (t, m)+\delta.
\end{aligned}
\]
Since $\delta>0$ is arbitrary, the conclusion follows.
\end{proof}

\subsection{Equivalence of $U^\e, U^\e_1$, and $U^\e_2$}\label{equivalence}
We begin with a localization estimate that will be used repeatedly in the proof of $U^\e=U^\e_1=U^\e_2$. The main point is that, along any admissible curve $\gamma$ whose velocity is bounded in $L^2$ in expectation, one may freeze the slow spatial variable and the measure argument on each subinterval of a uniform time partition at a total cost of order $O(N^{-1})$. More precisely, if $[t,T]$ is divided into $N$ subintervals and, on each $[t_k,t_{k+1}]$, the variables $x$ and $m$ in the Lagrangian are replaced by $\gamma_{t_k}$ and $\mathcal L(\gamma_{t_k})$, respectively, then the total freezing error is bounded by $C(T-t)^2/N$.

The second part of the lemma identifies the infimum of the corresponding frozen cost. Since the frozen problem decouples across the subintervals and across sample points $\omega$, its value is given by the sum of the costs $h^\e$ between the endpoints $\gamma_{t_k}$ and $\gamma_{t_{k+1}}$. Moreover, one can choose measurable nearly minimizing curves on each subinterval such that the resulting concatenated admissible path still has bounded $L^2$ velocity in expectation. This allows the freezing estimate to be applied not only to the original curve $\gamma$, but also to these piecewise nearly optimal paths.

\begin{lemma} \label{lem:fixest}
    Let $\e >0, t \in[0, T) ,m, \nu \in \mathcal{P}_2, (X,Y) \in \mathcal{E}^\ast_{m,\nu}$, and $\gamma \in \Gamma(X,Y)$ with
    \[
    \mathbb{E}\left[\int_{t}^T|\dot{\gamma}_s|^2 \,ds\right] \leq M(T-t),
    \]
    for some constant $M>0$. Let $N \in \mathbb{N}$, $\Delta t := (T-t)/N$, and $t_k := t + k\Delta t$ for $k=0,\dots,N$. Then the following hold.

    \nd
    (1) There exists a constant $C=C(C_{1,H}, C_H, K_0)>0$ such that
    \[
\sum_{k=0}^{N-1}\mathbb{E}\left[\int_{t_k}^{t_{k+1}} \left|L\left( \frac{\gamma_s}{\e}, \gamma_s, \dot{\gamma}_s, \mathcal{L}(\gamma_s) \right)-
L\left(\frac{\gamma_s}{\ep}, \gamma_{t_k}, \dot{\gamma}_s, \mathcal{L}(\gamma_{t_k})\right)\right|\,ds\right] \leq \frac{C(1+M)(T-t)^2}{N}.
    \]

\nd
(2) Define $\Gamma(X,Y;\gamma)
:=
\{\eta\in \Gamma(X,Y) : \eta_{t_k}=\gamma_{t_k}\,\, \text{a.s. for all } k=0,\dots,N\}$.  For $\eta \in \Gamma(X,Y;\gamma)$, define
\[
J_\e(\eta)
:=
\sum_{k=0}^{N-1}
\mathbb E\left[
\int_{t_k}^{t_{k+1}}
L\left(\frac{\eta_s}{\varepsilon},\gamma_{t_k},\dot{\eta}_s,\mathcal{L}(\gamma_{t_k})\right)\,ds\right].
\]
Then
\[
\inf_{\eta\in\Gamma(X,Y;\gamma)} J_\varepsilon(\eta)
=
\sum_{k=0}^{N-1}
\mathbb E \left[
h^\e\left(\gamma_{t_k},\mathcal{L}(\gamma_{t_k});t_k,t_{k+1},
\gamma_{t_k},\gamma_{t_{k+1}}\right)\right].
\]
Moreover, for every $\delta\in(0,T-t)$, there exists a $\delta$-minimizer
$\eta^\delta\in\Gamma(X,Y;\gamma)$ of $J_\varepsilon$ such that
\[
\eta^\delta(s, \omega)
=
\eta^\delta_k\left(s,\gamma_{t_k}(\omega),\gamma_{t_{k+1}}(\omega)\right),
\qquad s\in[t_k,t_{k+1}],
\]
where each $\eta^\delta_k:[t_k,t_{k+1}]\times\mathbb R^d\times\mathbb R^d\to\mathbb R^d$ is Borel measurable and satisfies
\[
\eta^\delta_k(\cdot,x,y)\in \mathrm{AC}([t_k,t_{k+1}];\mathbb R^d),
\quad
\eta^\delta_k(t_k,x,y)=x,\quad
\eta^\delta_k(t_{k+1},x,y)=y,\quad \text{ for all }x, y\in \mathbb{R}^d.
\]
Furthermore, ther exists a constant $\tilde{M}=\tilde{M}(C_{1,H},C_{2,H},K_0,M)>0$ such that
\[
\mathbb E\left[
\int_t^T
\left|\dot{\eta}^\delta_s\right|^2 \,ds\right]
\le \tilde{M} (T-t),
\]
and
\begin{equation}\label{eq:freeze-est-eta}
\sum_{k=0}^{N-1}\mathbb{E}\left[\int_{t_k}^{t_{k+1}} \left|L\left( \frac{\eta^\delta_s}{\e}, \eta^\delta_s, \dot{\eta}^\delta_s, \mathcal{L}(\eta^\delta_s) \right)-
L\left(\frac{\eta^\delta_s}{\ep}, \gamma_{t_k}, \dot{\eta}^\delta_s, \mathcal{L}(\gamma_{t_k})\right)\right|\,ds\right] \leq \frac{C(1+\tilde{M})(T-t)^2}{N}.
\end{equation}
\end{lemma}

\begin{proof}
We first prove (1). For each $k \in \left\{0, 1, \cdots, N-1\right\}$, denote
\[
E_k:=\mathbb{E}\left[\int_{t_k}^{t_{k+1}} \left|L\left( \frac{\gamma_s}{\e}, \gamma_s, \dot{\gamma}_s, \mathcal{L}(\gamma_s) \right)-
L\left(\frac{\gamma_s}{\ep}, \gamma_{t_k}, \dot{\gamma}_s, \mathcal{L}(\gamma_{t_k})\right)\right|\,ds\right].
\]
For $s\in[t_k, t_{k+1}]$, by H{\"o}lder's inequality,
\[
 {\bf d}_2\left(\mathcal{L}(\gamma_s), \mathcal{L}(\gamma_
    {t_k})\right) \leq \left(\mathbb{E}|\gamma_s-\gamma_{t_k}|^2 \right)^\frac{1}{2}\leq (\Delta t)^\frac{1}{2} A_k^\frac{1}{2},
\]
where $A_k := \mathbb{E}\left[\int_{t_k}^{t_{k+1}} |\dot{\gamma}_s|^2\,ds\right]$.
Since
\[
\mathbb{E}\left[\int_{t_k}^{t_{k+1}} |\dot{\gamma}_s|\,ds\right]
\le (\Delta t)^\frac{1}{2} A_k^\frac{1}{2},
\qquad
\mathbb{E}\left(\int_{t_k}^{t_{k+1}}|\dot\gamma_s|\,ds\right)^2
\le (\Delta t) A_k,
\]
it follows from \eqref{eqn:LmLip} that
\[
\begin{aligned}
E_k
&\le C_L \mathbb{E}\left[\int_{t_k}^{t_{k+1}}
\left(1+2|\dot\gamma_s|\right)
\left(
|\gamma_s-\gamma_{t_k}|
+ \mathbf d_2\bigl(\mathcal{L}(\gamma_s),\mathcal{L}(\gamma_{t_k})\bigr)
\right)\,ds \right]\\
&\le C_L \mathbb{E}\left[
\left(
t_{k+1}-t_k + 2\int_{t_k}^{t_{k+1}} |\dot\gamma_s|\,ds
\right)
\left(
\int_{t_k}^{t_{k+1}} |\dot\gamma_s|\,ds
+ (\Delta t)^{\frac{1}{2}} A_k^{\frac{1}{2}}
\right)
\right] \\
& \leq C_L (\Delta t )\left(\mathbb{E}\left[\int_{t_k}^{t_{k+1}} |\dot\gamma_s|\,ds\right]\right)
+2C_L \mathbb{E} \left[\int_{t_k}^{t_{k+1}} |\dot\gamma_s|\,ds\right]^2 \\
&\qquad \qquad \qquad \qquad \qquad \qquad + C_L (\Delta t)^{\frac{3}{2}} A_k^{\frac{1}{2}}
+2C_L (\Delta t)^{\frac{1}{2}}
\left(\mathbb{E}\left[\int_{t_k}^{t_{k+1}} |\dot\gamma_s|\,ds\right]\right) A_k^\frac{1}{2}\\
&\leq 4C_L (\Delta t) A_k
+ 2C_L (\Delta t)^{\frac{3}{2}} A_k^\frac{1}{2}.
\end{aligned}
\]
Using \eqref{eqn:gamdelvelbd} and the Cauchy–Schwarz inequality,
\[
\sum_{k=0}^{N-1}A_k\leq M(T-t) ,\qquad\sum_{k=0}^{N-1}A_k^\frac{1}{2} \leq \sqrt{N} \left(\sum_{k=0}^{N-1}A_k\right)^\frac{1}{2}\leq \sqrt{N} \sqrt{M(T-t)}.
\]
Hence,
\[
\sum_{k=0}^{N-1}E_k \leq  \frac{4C_LM(T-t)^2}{N} +   \frac{2C_L \sqrt{M}(T-t)^2}{N}.
\]

We now prove {\rm (2)}. By the definition of $h^\e$, for any $\eta\in \Gamma(X,Y;\gamma)$,
\[
J_\e(\eta)
\ge
\sum_{k=0}^{N-1}
 \mathbb{E}\left[
h^\e\left(\gamma_{t_k}, \mathcal{L}(\gamma_{t_k});t_k,t_{k+1},
\gamma_{t_k},\gamma_{t_{k+1}}\right)\right].
\]
Therefore,
\[
\inf_{\eta\in\Gamma(X,Y;\gamma)}J_\varepsilon(\eta)
\ge
\sum_{k=0}^{N-1}
\mathbb E\left[
h^\varepsilon\left(
\gamma_{t_k},
\mathcal{L}(\gamma_{t_k});
t_k,t_{k+1},
\gamma_{t_k},\gamma_{t_{k+1}}
\right)
\right].
\]

To obtain the reverse inequality, let $\delta>0$.
By Lemma \ref{lem:measapp}, for each $k$ and large $n$, there exists a Borel measurable path $\gamma_n^k:[t_k,t_{k+1}] \times \mathbb{R}^d \times \mathbb{R}^d \to \mathbb{R}^d$ such that
\[
\quad \gamma_n^k(\cdot, x,y) \in \mathrm{AC}([t_k, t_{k+1};\R^d]), \quad \gamma_n^k(t_k,x,y)=x, \quad \gamma_n^k(t_{k+1},x,y)=y,\quad \text{ for all }x, y\in \mathbb{R}^d,
\]
and
\[
\int_{t_k}^{t_{k+1}}
L \left(\frac{\gamma_n^k(s,x,y)}{\e}, x, \dot{\gamma}_n^k(s,x,y),\mathcal{L}(\gamma_{t_k})\right)\,ds
<
h^\e\left(x,\mathcal{L}(\gamma_{t_k});t_k, t_{k+1}, x, y\right)+\frac{C_N}{\sqrt{n}}\left(1+|x-y|^2\right),
\]
for some constant $C_N>0$. Since $ \sum_{k=0}^{N-1}\mathbb{E}\left|\gamma_{t_k}-\gamma_{t_{k_1}}\right|^2 \leq \Delta t \mathbb{E} \left[\int_t^T\left|\dot{\gamma}_s\right|^2 \, ds\right]\leq \frac{M(T-t)^2}{N}$,
\begin{equation}\label{eqn:gamm}
\begin{aligned}
&\sum_{k=0}^{N-1}\mathbb{E}\left[\int_{t_k}^{t_{k+1}}
L \left(\frac{\gamma_n^k(s,\gamma_{t_k},\gamma_{t_{k+1}})}{\varepsilon}, \gamma_{t_k}, \dot{\gamma}_n^k(s,\gamma_{t_k},\gamma_{t_{k+1}}),\mathcal{L}(\gamma_{t_k})\right)\,ds\right]\\
& \qquad \qquad \qquad\qquad\qquad\qquad\leq
\sum_{k=0}^{N-1}\mathbb{E} \left[h^\varepsilon\left(\gamma_{t_k},\mathcal{L}(\gamma_{t_k});t_k, t_{k+1}, \gamma_{t_k}, \gamma_{t_{k+1}}\right)\right]+\frac{C_N}{\sqrt{n}},
\end{aligned}
\end{equation}
for some constant $C_N>0$. Choose $n$ sufficiently large so that $\frac{C_N}{ \sqrt{n}} <\delta$ and define
\[
\eta^\delta(s, \omega):=
\gamma^k_n\left(s,\gamma_{t_k}(\omega),\gamma_{t_{k+1}}(\omega)\right),
\qquad s\in[t_k,t_{k+1}], \,\, k=0,1, \cdots, N-1.
\]
Then $\eta^\delta\in\Gamma(X,Y;\gamma)$ and
\[
\inf_{\eta\in\Gamma(X,Y;\gamma)}
J_\e(\eta) \leq J_\ep(\eta^\delta) \leq \sum_{k=0}^{N-1}\mathbb{E} \left[h^\varepsilon\left(\gamma_{t_k},\mathcal{L}(\gamma_{t_k});t_k, t_{k+1}, \gamma_{t_k}, \gamma_{t_{k+1}}\right)\right]+\delta.
\]
This proves the reverse inequality and also gives the measurable representation of a $\delta$-minimizer.

Since $\gamma \in \Gamma(X,Y; \gamma)$ is an admissible path,
\begin{equation}\label{eqn:gamplugin}
\begin{aligned}
    J_\ep(\eta^\delta) &\leq \inf_{\eta\in \Gamma(X,Y,\gamma)}
J_\ep(\eta)+\delta \leq \sum_{k=0}^{N-1}
\mathbb E \left[
\int_{t_k}^{t_{k+1}} L \left(
\frac{\gamma_s}{\varepsilon},\gamma_{t_k},
\dot{\gamma}_s,
\mathcal{L}(\gamma_{t_k})
\right)\,ds\right] +(T-t)\\
& \leq \frac{1}{4C_{1,H}} \mathbb{E}
\left[\int_{t}^{T} \left|\dot{\gamma}_s\right|^2\,ds\right]+(K_0+1)(T-t) \leq \left(\frac{M}{4C_{1,H}}+K_0+1\right)(T-t),
\end{aligned}
\end{equation}
where the third inequality follows from \eqref{eqn:Lmquad}. On the other hand, again by \eqref{eqn:Lmquad},
\[
J_\ep (\eta^\delta) \geq \sum_{k=0}^{N-1} \left(\frac{1}{4C_{2,H}}
\mathbb{E} \left[
\int_{t_k}^{t_{k+1}}  \left|\dot{\eta}^\delta_s \right|^2 \,ds \right] - K_0  \Delta t\right) \geq \frac{1}{4C_{2,H}} \mathbb{E}
\left[\int_t^{T}  \left|\dot{\eta}^\delta_s\right|^2 \,ds\right]-(T-t)K_0.
\]
Combing with \eqref{eqn:gamplugin}, we have
\[
\mathbb{E}
\left[\int_{t}^{T} \left|\dot{\eta}^\delta_s\right|^2 \,ds \right] \leq \tilde{M}(T-t),
\]
where $\tilde{M}=4C_{2,H}
\left(
\frac{M}{4C_{1,H}}+2K_0+1
\right)$.

Finally, applying the freezing estimate from part {\rm (1)} to the path $\eta^\delta$, with $M$ replaced by $\tilde{M}$, gives \eqref{eq:freeze-est-eta}. The proof is complete.
\end{proof}

We are now ready to prove that $U^\e=U^\e_1$. Although it is not clear whether, for an arbitrary $\nu\in\mathcal P_2$, every coupling $\pi\in\Pi(m,\nu)$ can be realized in the form $\mathcal L((X_t,X_T))$, the admissible controls in $\mathcal A$ generate ``enough'' couplings for the two value functions to coincide in the same sense that the Benamou--Brenier formulation of optimal transport covers all the ``interesting'' couplings.

\begin{proposition}\label{prop:U1eqUeqU2}
Assume (A1) and (A2). Then $U^\e=U^\e_1=U^\e_2$.
\end{proposition}

\begin{proof}
    By Propositions~\ref{prop:Uepm2new} and \ref{prop:UeqU2},
    \[
    U_1^\e(t,m)
= \inf_{\nu \in \mathcal{P}_2} \left\{ \inf_{(X, Y) \in \mathcal{E}^\ast_{m,\nu}} \left\{ \inf_{\gamma \in \Gamma(X, Y)} \mathbb{E}\left[\int_{t}^T L\left( \frac{\gamma_s}{\e}, \gamma_s, \dot{\gamma}_s, \mathcal{L}(\gamma_s) \right)\, ds \right] \right\}+G(\nu) \right\}.
    \]
Let $\delta >0$, and let $(\nu,X,Y,\gamma)$ be $\delta$-optimal. Then by Proposition~\ref{prop:Uepm2new},
\begin{equation}\label{eqn:gamdelvelbd}
\mathbb{E}\left[\int_{t}^T|\dot{\gamma}_s|^2 \,ds\right] \leq M_1(T-t)
\end{equation}
for some constant $M_1=M_1(C_{2,H}, C_G, K_0)>0$. Let $N\in\mathbb{N}$, $\Delta t=(T-t)/N$, $t_k=t+k\Delta t$. By Lemma~\ref{lem:fixest},
\begin{equation}\label{eqn:Uep1mul}
\begin{aligned}
U^\varepsilon_1(t,m) +\delta & \geq \mathbb{E}\left[\int_{t}^T L\left( \frac{\gamma_s}{\e}, \gamma_s, \dot{\gamma}_s, \mathcal{L}(\gamma_s) \right)\, ds \right]
+G(\nu)\\
& \geq \sum_{k=0}^{N-1} \mathbb{E} \left[ \int_{t_k}^{t_{k+1}}
L\left(\frac{\gamma_s}{\ep}, \gamma_{t_k}, \dot{\gamma}_s, \mathcal{L}(\gamma_{t_k})\right)\,ds\right]-\frac{C}{N}  +G(\nu)\\
& \geq \sum_{k=0}^{N-1} \mathbb{E} \left[h^\e\left(\gamma_{t_k}, \mathcal{L}(\gamma_{t_k});t_k,t_{k+1},\gamma_{t_k},\gamma_{t_{k+1}}\right)\right]-\frac{C}{N} +G(\nu),
\end{aligned}
\end{equation}
for some constant $C>0$ independent of $N$.
By Lemma~\ref{lem:fixest} again, for any $\delta>0$, there exists
$\eta^{\delta}\in\Gamma(X,Y;\gamma)$ such that
\begin{equation}\label{eqn:eta-piecewise-near-opt}
\sum_{k=0}^{N-1} \mathbb E\left[ \int_{t_k}^{t_{k+1}} L\left( \frac{\eta^\delta_s}{\varepsilon}, \gamma_{t_k}, \dot{\eta}^\delta_s, \mathcal{L}(\gamma_{t_k})\right)\,ds \right] \le \sum_{k=0}^{N-1} \mathbb E\left[ h^\varepsilon\left( \gamma_{t_k}, \mathcal{L}(\gamma_{t_k}); t_k,t_{k+1}, \gamma_{t_k}, \gamma_{t_{k+1}} \right) \right] + \delta,
\end{equation}
and
\begin{equation}\label{eqn:etavelbd}
\mathbb E\left[\int_t^T|\dot\eta_s^\delta|^2\,ds\right]\le \tilde{M}(T-t),
\end{equation}
where $\tilde M= \tilde{M}(C_{1,H},C_{2,H}, K_0, C_G)>0$ is a constant. Moreover,
\[
\eta^\delta(s,\omega)
=
\eta^\delta_k\left(
s,\gamma_{t_k}(\omega),\gamma_{t_{k+1}}(\omega)
\right),
\qquad s\in[t_k,t_{k+1}],
\]
for some Borel measurable maps $\eta^\delta_k:[t_k,t_{k+1}]\times\mathbb R^d\times\mathbb R^d\to\mathbb R^d$ with $\eta^\delta_k(\cdot,x,y)\in \mathrm{AC}([t_k,t_{k+1}];\mathbb R^d)$ and $\eta^\delta_k(t_k,x,y)=x,\; \eta^\delta_k(t_{k+1},x,y)=y$.


Since $\xi$ is an independent Gaussian random variable, there exist Borel measurable functions $f_k:\mathbb{R}^d \to \mathbb{R}^d$ for $k=0,1,\cdots, N$ such that
\begin{equation}\label{eqn:coupling_gamma}
\left(\gamma_t, \gamma_{t_1}, \ldots, \gamma_{T}\right)\stackrel{d}{=}\left(f_0(\xi), f_1(\xi), \ldots, f_N(\xi)\right).
\end{equation}
 Define $\tilde{\alpha}^\delta:[t,T] \times \mathbb{R}^d \to \mathbb{R}^d$ by
\[
\tilde{\alpha}^\delta(s,z):=\dot{\eta}^\delta_k(s,f_k(z),f_{k+1}(z)), \quad \text{ for } s \in [t_k, t_{k+1}],\quad k=0,1, \cdots, N-1,
\]
and set $\alpha^\delta(s,\omega):=\tilde{\alpha}^\delta(s,\xi(\omega))$. Then \eqref{eqn:etavelbd} implies $\alpha^\delta\in\mathcal{A}_2$. The corresponding controlled process satisfies
\[
\begin{aligned}
X_s^\delta\stackrel{d}{=}\eta^\delta(s,\cdot), \qquad \text{ for } s \in[t, T].
\end{aligned}
\]
Hence, using \eqref{eqn:Uep1mul}–\eqref{eqn:etavelbd} and Lemma~\ref{lem:fixest}, we obtain
\[
\begin{aligned}
U^\e_2(t,m) &\leq \mathbb{E}\left[\int_t^TL\left(\frac{X_s^\delta}{\ep}, X_s^\delta,\alpha_s^\delta,\mathcal{L}(X_s^\delta)\right)\, ds\right]+G(\nu)\\
&=\mathbb E\left[ \int_t^T L\left( \frac{\eta^\delta_s}{\varepsilon}, \eta^\delta_s, \dot\eta^\delta_s, \mathcal L(\eta^\delta_s) \right)\,ds \right]+G(\nu)
\\&\le \sum_{k=0}^{N-1} \mathbb E\left[ \int_{t_k}^{t_{k+1}} L\left(\frac{\eta^\delta_s}{\varepsilon}, \gamma_{t_k}, \dot\eta^\delta_s, \mathcal{L}(\gamma_{t_k}) \right)\,ds \right] +\frac{C}{N}+G(\nu)
\\
&\leq \sum_{k=0}^{N-1}\mathbb{E} \left[h^\e \left(\gamma_{t_k},\mathcal{L}(\gamma_{t_k});t_k, t_{k+1}, \gamma_{t_k}, \gamma_{t_{k+1}}\right)\right]+\frac{C}{N}+\delta+G(\nu)\\
&\leq U_1^\ep(t,m)+2\delta+\frac{C}{N}.\\
\end{aligned}\]
Letting $\delta\to 0$, then $N\to\infty$, we conclude $U_2^\ep(t,m) \leq U_1^\ep(t,m)$.
\end{proof}

\begin{remark} We note that, the use of Lemma~\ref{lem:fixest} in the above proof is to overcome the technical difficulty that appears only in the multiscale case. Specifically, without cutting the intervals, the coupling constructed in~\eqref{eqn:coupling_gamma} needs to satisfy an infinite number of constaints, i.e. $\gamma_t=f_t(\xi)$ for all $t$. The existence of such coupling is not clear to the best of our knowledge. When $H$ does not have $x,m$ dependence, such constraint no longer appears and the proof can be much simplified. In this simplified case, we do not need to freeze the slow variable and separate the time interval into $N$ subintervals. The same proof as above can be directly conducted in $[t,T]$ to prove $U^\e=U^\e_1=U^\e_2$.
\end{remark}

\section{Uniform convergence of $U^\ep$ and rate of convergence under (A1') and (A2)} \label{sec:rocresult}

\nd
By Corollary \ref{cor:lipschitzestimate}, we know that $U^\e$ satisfies
$$|U^\e(t,m_1)-U^\e(s,m_2)|\leq C_0 \left(|t-s|+{\bf d}_2(m_1,m_2)\right),\;\;\text{for all }t,s\in[0,T],\; m_1,m_2\in\mathcal{P}_2,$$
where the constant $C_0$ is independent of $\e$. Thus, the family $(U^\e)_{\e>0}$ is uniformly equicontinuous and equibounded, hence by Arzela--Ascoli it converges as $\e\to 0$ (up to subsequences) uniformly in the compact subsets of $[0,T]\times \mathcal{P}_2$. Moreover, since $[0,T]\times \mathcal{P}_2$ is separable, we may obtain that $U^\e$ converges pointwise to a function $U\colon[0,T]\times\mathcal{P}_2\to \R$ satisfying the same Lipschitz estimate as $U^\e$.
\vspace{1mm}

\nd
Indeed, let $\mathcal{D}$ be a countable dense subset of $[0,T]\times\mathcal{P}_2$. Then, a diagonal argument yields the pointwise convergence of $U^\e$ to a function $U\colon\mathcal{D}\to \R$, which is also Lipschitz with constant $C_0$ over $\mathcal{D}$. Since $\mathcal{D}$ is dense, we may extend $U$ to a $C_0$-Lipschitz function (still denoted by $U$) $U\colon [0,T]\times\mathcal{P}_2\to \R$. We argue that $U^\e(t,m)\xrightarrow{\e\to 0}U(t,m)$, for any $(t,m)\in[0,T]\times\mathcal{P}_2$. For any $\eta>0$, we may choose $(t_n,m_n)\in\mathcal{D}$ such that $|t-t_n|+{\bf d}_2(m,m_n)\le \eta/2C_0$. Then we write
\begin{align*}
    \left|U^\e(t,m)-U(t,m)\right|&\le |U^\e(t,m)-U^\e(t_n,m_n)|+|U^\e(t_n,m_n)-U(t_n,m_n)|+|U(t_n,m_n)-U(t,m)|\\
    &\le 2C_0(|t-t_n|+{\bf d}_2(m,m_n)) +|U^\e(t_n,m_n)-U(t_n,m_n)|\\
    &\le \eta + |U^\e(t_n,m_n)-U(t_n,m_n)|.
\end{align*}
Sending $\e\to 0$, this implies $\limsup_{\e\to 0}\left|U^\e(t,m)-U(t,m)\right|\le \eta$, and our claim follows by the arbitrariness of $\eta>0$.

\vspace{1mm}

 \nd
In this section we strengthen the above convergence. Under assumptions (A1') and (A2), we prove that $U^\varepsilon$ converges uniformly to a limit $\bar U\colon [0,T]\times \mathcal{P}_2\to \R$ and we establish the optimal convergence rate $O(\varepsilon)$. The argument relies on the optimal control representation of $U^\varepsilon$. As shown in the previous section, $U^\varepsilon$ admits an equivalent formulation in terms of the cost function $h^\ep$. The subadditivity and superadditivity properties of $h$, summarized in Proposition~\ref{prop:subsuperm}, then yield the optimal convergence rate.

\subsection{Uniform convergence and rate of convergence}

Using the averaged metric $\bar{h}$, we define a new value function $\bar{U}$, which will be the uniform limit of $U^\ep$.

\begin{definition} Define $\bar{U} :[0, T] \times \mathcal{P}_2 \to \R $ by
\[
\bar{U}(t,m):=\inf_{\nu\in\mathcal{P}_2}\left\{\inf_{(X, Y) \in \mathcal{E}^\ast_{m,\nu}} \mathbb{E} \left[\bar{h}\left(t, T, X, Y\right)\right]+G(\nu)\right\},
\]
where $\mathcal{E}_{m,\nu}^\ast$ is defined in \eqref{eqn:estar}, and $h^\varepsilon$ is given in Proposition~\ref{prop:subsuperm}.
\end{definition}

We now present the uniform convergence of $U^\varepsilon$ to $\bar U$, together with the rate of convergence, in the following theorem. Moreover, the rate $O(\varepsilon)$ is optimal; see Proposition~\ref{prop:opex}.

\begin{theorem}\label{thm: unif conv}
    Assume {\rm (A1')} and {\rm (A2)}. $U^\ep$ converges to $\bar{U}$ uniformly on $[0, T] \times \mathcal{P}_2$ as $\ep \to 0$ , and there exists a constant $C=C\left(d, C_{1, H}, C_{2,H},K_0, C_G\right)>0$ such that for any $(t, m)\in [0,T] \times \mathcal{P}_2$,
    \[
    \left|U^\ep\left(t,m\right)-\bar{U}\left(t, m\right)\right|\leq C \ep.
    \]
\end{theorem}
\begin{proof}
Let $\ep>0$, and $(t, m)\in [0,T] \times \mathcal{P}_2$. If $t=T$, by the definition of $m$ in \eqref{def:meps1},
\[
   h^\e \left(t,t,x,y\right)=h^\e\left(0,0,x,y\right)=\left\{\begin{aligned}
    &0, \qquad \quad \text{ if } x=y,\\
    &+\infty, \quad \, \text{ if } x \neq y.\\
    \end{aligned}\right.
\]
Hence,
\[
    \bar{h}\left(t,t,x,y\right)=\bar{h}\left(0,0,x,y\right)=\left\{\begin{aligned}
    &0, \qquad \quad \text{ if } x=y,\\
    &+\infty, \quad \, \text{ if } x \neq y,\\
    \end{aligned}\right.
\]
and
$\bar{U}(t, m)=G(m)=U^\ep(t, m)$.

If $t<T$, for any $\nu \in \mathcal{P}_2$ and $(X,Y) \in \mathcal{E}^\ast_{m, \nu}$, by Proposition \ref{prop:subsuperm},
    \[
\mathbb{E} \left[\left|\bar{h}\left(t, T,X,Y\right)- h^\e \left(t,T,X,Y\right)\right|\right] \leq C \left(1+ \frac{\mathbb{E}|X-Y|^2}{(T-t)^2}\right)\e \leq C\left(1+M_0\right)\ep,
    \]
for some constant \(C=C(d, C_{1, H}, C_{2, H}, K_0)>0\). Since $M_0=M_0\left(C_{2,H},C_G, K_0\right)>0$,
\[
\begin{aligned}
U^\ep(t, m)&=\inf_{\nu \in \mathcal{P}_2} \left\{
\inf_{(X, Y) \in \mathcal{E}^\ast_{m,\nu}}
\mathbb{E}[h^\e(t,T,X,Y)]+ G(\nu)\right\}\\
&\leq \inf_{\nu\in\mathcal{P}_2}\left\{\inf_{(X, Y) \in \mathcal{E}^\ast_{m,\nu}}
\mathbb{E}[\bar{h}(t,T,X,Y)]+G(\nu)\right\} +C(1+M_0)\ep\\
&=\bar{U}(t,m)+\tilde{C}\ep,
\end{aligned}
\]
where $\tilde{C}=\tilde{C}(d,C_{1, H}, C_{2,H}, K_0, M_0)=\tilde{C}(d,C_{1, H}, C_{2,H}, K_0, C_G)>0$.
By a similar argument, we have
\[
\bar{U}\left(t, m\right)\leq U^\ep\left(t, m\right)+\tilde{C}\ep.
\]
\end{proof}

\begin{remark}
    We do not identify $\bar{h}$ in terms of $\bar{L}$ at this stage; this will be addressed later (see Proposition~\ref{prop:rocmu}). Instead, in the next subsection, we employ the doubling-variable method to show that $\bar{U}$ indeed solves the limiting equation. The reason for this approach is that identifying $\bar{h}$ in terms of $\bar{L}$ is highly nontrivial (see, for instance, \cite{LionsPapaVara1987,Weinan1991}), and in more complicated settings such an identification may fail. By using the doubling-variable method, we make the structure of the argument transparent, allowing it to be extended to more general situations in future work.
\end{remark}

\subsection{The limit $\overline{U}$ and the limiting equation}\label{charoflimit}

To characterize the limit $\overline{U}$, we prove that it is the unique $L$-viscosity solution of \eqref{HJB}. For the proof, we will need the following proposition.

\begin{proposition}\label{supersolution}
Assume that $H$ satisfies Assumption (A1'). Suppose that $\delta>0$ are fixed and for $p\in \R^d$ let $v^\delta(\cdot):=v^{\delta}(\cdot, p)$ be the unique viscosity solution of \eqref{discountcellm}.
For a fixed $P\in\mathcal{H}$, consider the function $V^\delta\colon \mathcal{H}\to \R$ with $V^\delta(X)=\mathbb{E}\left[ v^\delta(X,P)\right]$ and assume that for some fixed $\phi\in C^1(\mathcal{H})$, the function $V^\delta -\phi$ admits a minimum at some $X_0\in \mathcal{H}$. Then,
\be \label{strict supersol}
\delta V^\delta(X_0)+\mathbb{E}\left[ H( X_0,P+D\phi(X_0))   \right]\ge 0
\ee
\end{proposition}

\begin{proof} The proof is divided in two parts.
\vspace{1mm}

\nd
\textit{Step 1.} (Alternative control formulation for $V^\delta$).\\
    It is known that, for any $y \in \T^d$, $v^\delta$ satisfies
    $$v^\delta(y,p)=\inf_{\alpha\in \mathcal{A}_{det}}\left\{ \int_0^{+\infty} e^{-\delta t}\left(L(X_t,\alpha_t )+\alpha_t\cdot p\right)dt\right\}\text{, where }\;\;X_t=y+\int_0^t\alpha_sds$$
    and $\mathcal{A}_{det}=\{\alpha:[0,+\infty)\to \R^d: \alpha\text{ is }\mathcal{B}([0,+\infty))\text{-measurable and }\int_0^{+\infty}e^{-\delta t}|\alpha_t|^2dt<\infty\}.$

    \nd
    In addition, $v^\delta(y,p)$ is continuous and satisfies $|v^\delta(y,p)|\le C(1+|p|^2)$, for some constant $C>0$ depending only on $\delta$ and $H$. We let
    $$V_2^\delta(X):=\inf_{\alpha\in \mathcal{A}}\left\{ \mathbb{E}\left[\int_0^{+\infty}e^{-\delta t}\left( L(X_t,\alpha_t) +\alpha_t\cdot P\right)dt\right] \right\},$$
    where $X_t= X+\int_0^t\alpha_s ds$ and $\mathcal{A}:=\mathcal{A}(X,P)$ is the set  of controls
    \begin{align*}
        \mathcal{A}(X,P)=\bigg\{ \alpha\colon [0,+\infty)\times \Omega\to \R^d:\; \alpha\text{ is }\mathcal{B}([0,+\infty))\otimes &\;\sigma(X,P)\text{-measurable}\\
        &\text{ and }\int_0^{+\infty}e^{-\delta t}\mathbb{E}[|\alpha_t|^2]dt<\infty\bigg\}.
    \end{align*}
    We claim that for any $X\in\mathcal{H}$
    \be\label{claim}\nonumber
    \mathbb{E}\left[ v^\delta(X,P) \right]=V_2^\delta(X).
    \ee
    It is clear that $\mathbb{E}[v^\delta(X,P)]\le V_2^\delta(X)$, since $\alpha(\cdot,\omega)\in \mathcal{A}_{det}$ for almost every $\omega\in \Omega$, for any $\alpha\in \mathcal{A}(X,P)$. To prove the reverse inequality, we let $\e>0$ and for $n\in \mathbb{N}$ we construct simple random variables $X^n,P^n$ such that
    \begin{align*}
        &X^n(\omega)=\sum_{k=1}^{N_n}x_k{\bf 1}_{A_k}(\omega),\;\;\text{with }\;\|X^n-X\|_2\le \frac{1}{n},\text{ and }\\
        & P^n(\omega)=\sum_{k=1}^{N_n}p_k{\bf 1}_{A_k}(\omega),\;\;\text{with }\;\|P^n-P\|_2\le \frac{1}{n},
    \end{align*}
    for some partition $(A_k)_{k=1}^{N_n}$ of the probability space $\Omega$. For any $k\in \{1,\ldots,N_n\}$, there exists an $\e$-optimal control $\alpha_k$ for $v^\delta(x_k,p_k)$. We now consider the control
    $$\alpha_s^n(\omega)=\sum_{k=1}^{N_n}(\alpha_k)_s{\bf 1}_{A_k}(\omega).$$
    and the processes $(X_t^n)_{t\ge 0}$, $(X_t)_{t\ge 0}$ with
    $$X_t^n=X^n+\int_0^t\alpha^n_sds\quad\text{and}\quad X_t=X+\int_0^t\alpha^n_s ds.$$
    By the $\e$ optimality and the quadratic growth of $L$ we have for some constant $c,C$
    \begin{align*}
        \mathbb{E}\left[v^\delta(X^n,P^n) \right]&\ge \mathbb{E}\left[\int_0^{+\infty}e^{-\delta t}\left(L(X_t^n,\alpha_t^n)+\alpha_t^n\cdot P^n   \right)dt\right]-\e\\
        &\ge c\int_0^{+\infty}e^{-\delta t}\|\alpha_t^n\|_2^2dt-C(1+\|P^n\|_2^2).
    \end{align*}
    Thus, by the boundedness of $\delta v^\delta$ and the fact that $\|P^n\|_2\xrightarrow{n\to \infty}\|P\|_2$, we get the bound
    \be\label{eq: bound}
        \int_0^{+\infty}e^{-\delta t}\|\alpha_t^n\|_2^2dt\le C_{\delta, P}+\e,
    \ee
    for some constant $C_{\delta,P}$ depending only on $\delta, P$.
    \vspace{1mm}

    \nd
    By the $\e$-optimality, once again, we have by the regularity properties of $L$
    \begin{align*}
        \mathbb{E}\left[v^\delta(X^n,P^n) \right]&\ge \mathbb{E}\left[\int_0^{+\infty}e^{-\delta t}\left(L(X_t^n,\alpha_t^n)+\alpha_t^n\cdot P^n   \right)dt\right]-\e\\
        &= \mathbb{E}\left[\int_0^{+\infty}e^{-\delta t}\left(L(X_t,\alpha_t^n)+\alpha_t^n\cdot P \right)dt\right]-\e\\
        &\quad+\mathbb{E}\left[\int_0^{+\infty}e^{-\delta t}\left(L(X_t^n,\alpha_t^n)-L(X_t,\alpha^n_t) \right)dt\right]+\mathbb{E}\left[\int_0^{+\infty}e^{-\delta t}\alpha_t^n\cdot\left( P^n-P \right)dt\right]\\
        &\ge V_2^\delta(X)-\e-\mathbb{E}\left[\int_0^{+\infty}e^{-\delta t}(1+|\alpha_t^n|)|X-X^n|dt \right]-\|P^n-P\|_2\int_0^{+\infty}e^{-\delta t}\|\alpha_t^n\|_2dt\\
        &\ge V_2^\delta(X)-\e -(\|X-X^n\|_2+\|P-P^n\|_2)\int_0^{+\infty}e^{-\delta t}(1+\|\alpha_t^n\|_2)dt.
    \end{align*}
    Since $v^\delta$ is continuous and  and satisfies $|v^\delta(y,p)|\le C(1+|p|^2)$, we have by the dominted convergence theorem $\mathbb{E}[v^\delta(X^n,P^n)]\xrightarrow{n\to\infty}\mathbb{E}[v^\delta(X,P)]$, therefore, by choosing $n$ large enough,
    $$\mathbb{E}[v^{\delta}(X,P)]+\frac{\e}{2}\ge V_2^\delta(X)-\e-\frac{2}{n}\int_0^{+\infty}e^{-\delta t}(1+\|\alpha_t^n\|_2)dt.$$
    Combining this with \eqref{eq: bound} and setting $n$ to be large enough, we get $\mathbb{E}[v^\delta(X,P)]\ge V_2^\delta(X)-2\e$. The inequality $\mathbb{E}[v^\delta(X,P)]\ge V_2^\delta(X)$ follows.
    \vspace{2mm}

    \nd
    \textit{Step 2.} (Inequality \eqref{strict supersol} via dynamic programming)\\
    We rely on the dynamic programming equality satisfied by $V_2^\delta$:
    \be\label{DPP}
V_2^\delta(X)=\inf_{\alpha\in \mathcal{A}}\mathbb{E}\left[ \int_0^he^{-\delta t}(L(X_t,\alpha_t)+\alpha_t\cdot P)dt+e^{-\delta h}V_2^\delta(X_h)\right], \text{ for all }h\ge 0\text{ and }X\in\mathcal{H},
    \ee
    where $X_h=X+\int_0^h\alpha_sds$. Having \eqref{DPP} established, by classical arguments (for example see \cite[Section 2.5]{fabbri2017stochastic}), we get the viscosity supersolution inequality
    $$\delta V^\delta_2(X_0)+\sup_{a}\mathbb{E}[-L(X_0,a)-a\cdot(D\phi(X_0)+P))]\ge 0,$$
    where the supremum is taken over all square integrable and $\sigma(X_0,P)$-measurable random variables $a\colon \Omega\to\R^d$. Inequality \eqref{strict supersol} then follows from the above because
    $$\mathbb{E}[\sup_{a\in \R^d}\{-L(X_0,a)-a\cdot (D\phi(X_0)+P)\}]\ge \sup_{a}\mathbb{E}[-L(X_0,a)-a\cdot(D\phi(X_0)+P))].$$
    \vspace{0.5mm}

    \nd
    We conclude the proof by showing \eqref{DPP}. By the definition of $V_2^\delta$ at the beginning of the proof, \eqref{DPP} hold as a $\ge$ inequality. It suffices to prove the reverse. Let $\alpha^1\in\mathcal{A}=\mathcal{A}(X,P)$ and $X_t=X+\int_0^t\alpha_s^1ds$. For $\e>0$, we pick $\alpha^2\in\mathcal{A}(X_h,P)$ such that
    \begin{align}
        V_2^\delta(X_h)+\e&\ge \mathbb{E}\left[\int_0^{+\infty}e^{-\delta t}(L(\widetilde{X}_t,\alpha_t^2)+\alpha^2_t\cdot P)dt     \right]\nonumber\\
        &=\mathbb{E}\left[\int_h^{+\infty}e^{-\delta (t-h)}(L(\widetilde{X}_{t-h},\alpha_{t-h}^2)+\alpha^2_{t-h}\cdot P)dt \right],\label{inequalityDPP}
    \end{align}
    where $\widetilde{X}_t=X_h+\int_0^t\alpha^2_sds,\;\;t\ge 0$. Consider the control $\hat{\alpha}_t=\begin{cases}
        \alpha^1_t,&t\in [0,h],\\
        \alpha^2_{t-h},& t\in (h,+\infty).
    \end{cases}$ We consider the process $(Y_t)_{t\ge 0}$ with $dY_t=\hat{\alpha}_tdt,\;t\ge 0$ and $Y_0=X$. We have $Y_t=X_t,\;t\in [0,h]$ and $Y_t=\widetilde{X}_{t-h},\; t\ge h$. Since $\sigma(X_h,P)\subseteq \sigma(X,P)$, it is clear that $\hat{\alpha}\in \mathcal{A}(X,P)$, therefore we have
    \begin{align*}
        V_2^\delta(X)&\le \mathbb{E}\left[\int_0^{h}e^{-\delta t}(L(X_t,\alpha_t^1)+\alpha^1_t\cdot P)dt\right]+e^{-\delta h}\mathbb{E}\left[\int_h^{+\infty}e^{-\delta (t-h)}(L(\widetilde{X}_{t-h},\alpha_{t-h}^2)+\alpha^2_{t-h}\cdot P)dt \right]\\
        &\le \mathbb{E}\left[\int_0^{h}e^{-\delta t}(L(X_t,\alpha_t^1)+\alpha^1_t\cdot P)dt\right]+e^{-\delta h}\left( V_2^\delta(X_h)+\e\right),
    \end{align*}
    where in the last line we used \eqref{inequalityDPP}. The result follows by sending $\e\to 0$ and by minimizing over $\alpha^1\in\mathcal{A}(X,P)$.
\end{proof}
\vspace{3mm}

\begin{proof}[Proof of the viscosity property of the limit]
    By Theorem \ref{thm: unif conv}, we know that $U^\e\colon[0,T]\times \mathcal{P}_2\to \R$ converges to a function $U\colon [0,T]\times \mathcal{P}_2\to \R$ uniformly. We will show that $U$ satisfies \eqref{HJB} in the viscosity sense. Then the result will follow by uniqueness of viscosity solutions provided by Theorem \ref{comprep}.
    \vspace{2mm}

    \nd
     We will only show the subsolution property, since the supersolution will be similar. We note that $U(T,m)=G(m)$, therefore we just need to verify the second property of Definition \ref{L-visc}. Let $\Phi\in C^{1,1}([0,T]\times\mathcal{H})$ be a test function and assume that $\hat{U}(t,X)-\Phi(t,X)$ admits a strict global maximum with value zero at $(t_0,X_0)\in [0,T)\times\mathcal{H}$. We assume that $t_0>0$, so there exists $r>0$ such that $(t_0-r,t_0+r)\subset [0,T)$. The case $t_0=0$ can be treated analogously by considering $r$ such that $[0,r)\subset [0,T)$. Seeking a contradiction, we suppose that
     \be\label{"wrong" ineq}
-\partial_t\Phi(t_0,X_0)+\mathbb{E}\left[\overline{H}(D\Phi(t_0,X_0))\right]=\theta>0.
    \ee
    Our proof utilizes the idea of the perturbed test function method. We let $\delta>0$ and, for any $p\in \R^d$, we consider $v^\delta(\cdot):= v^\delta(\cdot, p)$ to be the solution of the discounted cell problem
    \be\label{discountedcell}
\delta v^\delta+ H(y, p+D_y v^\delta)=0,\;\; y\in\T^d.
    \ee
    We may consider the functions $V^\delta\colon\mathcal{H}\to \R$ with $V^\delta(X)=\mathbb{E}\left[ v^\delta(X,D\Phi(t_0,X_0)   \right]$ and $\Psi^{\e,\delta,\eta}\colon [0,T]\times \mathcal{H}\times\mathcal{H}\to \R$ with
    $$\Psi^{\e,\delta,\eta}(t,X,Y)=\Phi(t,X)+\e V^\delta(Y)+\frac{\|Y-\frac{X}{\e}\|_2^2}{\eta}+L\left(\|X-X_0\|_2^2+|t-t_0|^2  \right),$$
    Consider $r>0$ small enough (to be chosen later) and the set
    $$A_r=(t_0-r,t_0+r)\times \left\{ X\in\mathcal{H}:\; \|X-X_0\|_2< r  \right\}.$$
    By the perturbed optimization result of Ekeland-Lebourg \cite[Theorem 3.25]{fabbri2017stochastic} for every $\mathbb{N}\ni n\ge 1$, there exist $a_n\in \R$ and $X_n,Y_n\in \mathcal{H}$ such that $|a_n|+\|X_n\|_2+\|Y_n\|_2\le \frac{1}{n}$ and the function
    $$[0,T]\times\mathcal{H}\times\mathcal{H}\ni(t,X,Y)\mapsto\hat{U}^\e(t,X)-\Psi^{\e,\delta,\eta}(t,X,Y)-a_nt-\mathbb{E}[X_n\cdot X]-\mathbb{E}[Y_n\cdot Y]-\e^3 \|Y\|_2^2$$
    attains a maximum at some point $(\overline{t},\overline{X},\overline{Y})$ such that $(\overline{t},\overline{X})\in \overline{A}_r$, where $\overline{A}_r$ is the closure of the set $A_r$ defined above, and $\|\overline{Y}\|_2<\e^{-3/2}$.


    By comparing the values of the above function at $(\overline{t},\overline{X},\overline{Y})$ and at $(\overline{t},\overline{X},\overline{X}/\e)$, we obtain
    \begin{align}
        -\e V^\delta(\overline{X}/\e)-\mathbb{E}\left[ \frac{Y_n\cdot \overline{X}}{\e}   \right] -\e^3\left\|\frac{\overline{X}}{\e}\right\|_2^2\le -\e V^\delta(\overline{Y})-\frac{\|\overline{Y}-\overline{X}/\e\|_2^2}{\eta}-\mathbb{E}[Y_n\cdot \overline{Y}]-\e^3\|\overline{Y}\|_2^2.\label{ineq100}
    \end{align}
    Consequently, for $n$ large enough such that $\frac{1}{n}\le \e^5$, we discover due to the uniform Lipschitz bound for $v^\delta$
    \begin{align*}
        \e^3\|\overline{Y}\|_2^2+\frac{\|\overline{Y}-\overline{X}/\e\|_2^2}{\eta}&\le \e\left(V^\delta(\overline{X}/\e)- V^\delta(\overline{Y})\right)+\e\|\overline{X}\|_2^2-\mathbb{E}[Y_n\cdot \overline{Y}]+\frac{1}{\e}\mathbb{E}[Y_n\cdot \overline{X}]\\
        &\le \e\|Dv^{\delta}\|_{\infty}\left\|\frac{\overline{X}}{\e}-\overline{Y}  \right\|_2 +2\e(r^2+\|X_0\|_2^2)+\e^5\|\overline{Y}\|_2+\e^4\|\overline{X}\|_2\\
        &\le C\e\left\|\frac{\overline{X}}{\e}-\overline{Y}  \right\|_2+ \e^5\|\overline{Y}\|_2+4\e(r^2+\|X_0\|_2^2)\\
        &\le C\e(r^2+1) +\frac{\|\overline{X}/\e-\overline{Y}\|_2^2}{2\eta}+\frac{C^2\e^2}{2}+\frac{\e^3\|\overline{Y}\|_2^2}{2}+\frac{\e^7}{2},
    \end{align*}
    which gives
    \begin{align}
        \|\overline{Y}-\overline{X}/\e\|_2^2\le \eta\e C(r^2+1) \quad\text{and}\quad \e^2\|\overline{Y}\|_2^2\le C(r^2+1),\label{prox1}
    \end{align}

    \nd
    for some constant $C>0$ depending only on the data and $X_0$. Going back to \eqref{ineq100}, we get
    \begin{align*}
        \frac{\|\overline{Y}-\overline{X}/\e\|_2^2}{\eta}&\le \e\left(V^\delta(\overline{X}/\e)- V^\delta(\overline{Y})\right)-\mathbb{E}[Y_n(\cdot \overline{Y}-\overline{X}/\e)]+\e^3\left(\|\overline{X}/\e\|_2^2-\|\overline{Y}\|_2^2\right)\\
        &\le C\e\|\overline{X}/\e-\overline{Y}\|_2+\frac{1}{n}\|\overline{X}/\e-\overline{Y}\|_2+\|\overline{X}/\e-\overline{Y}\|_2\left(\e^2\|\overline{X}\|_2+\e^3\|\overline{Y}\|_2\right),
    \end{align*}
    which, due to \eqref{prox1}, yields
    \be\label{prox2}
\|\overline{X}/\e-\overline{Y}\|_2\le C\e\eta,
    \ee
    for some constant $C>0$ depending only on the data and $X_0$.
    \vspace{1mm}

    \nd
    We consider 2 cases depending on whether $(\overline{t},\overline{X})$ is on the boundary or in the interior of $\overline{A}_r$.
    \vspace{3mm}

    \nd
    \textit{Case 1.} We start with the case where $(\overline{t},\overline{X})$ is in the interior. Now, by the definition of viscosity solutions for $\hat{U}^\e$ we have
    \be\label{viscprop}
    \begin{split}
    -\partial_t\Phi(\overline{t},\overline{X})-a_n-2L(t&-t_0)\\
    &+\mathbb{E}\left[ H\left( \frac{\overline{X}}{\e}, \frac{2 (\overline{X}/\e-\overline{Y})}{\e\eta}+2L(\overline{X}-X_0)+X_n+D\Phi(\overline{t},\overline{X})    \right)\right]\le 0.
    \end{split}
    \ee
    In addition, since the function
    $$\mathcal{H}\ni Y\mapsto \e V^\delta(Y)+\frac{\|\overline{X}/\e-Y\|_2^2}{\eta}+\mathbb{E}[Y_n\cdot Y]+\e^3\|Y\|_2^2$$
    admits a minimum at $\overline{Y}$, Proposition \ref{supersolution} yields
    \be\label{strictsuper}
    \delta V^\delta(\overline{Y})+\mathbb{E}\left[ H\left(\overline{Y},D\Phi(t_0,X_0)+\frac{2(\overline{X}/\e-\overline{Y})}{\e\eta} -2\overline{Y}\e^2 -\frac{Y_n}{\e} \right)   \right]\ge 0.
    \ee
    Due to the regularity of $\Phi$, we may choose $r$ small enough such that
    \be\label{prox3}
    \left|\partial_t\Phi(\overline{t},\overline{X})-\partial_t\Phi(t_0,X_0)\right|+\left\|D\Phi(\overline{t},\overline{X})-D\Phi(t_0,X_0)\right\|_2\le \delta^2/3.
    \ee
    We add \eqref{viscprop} and \eqref{strictsuper}, and we get by the properties of $H$
    \begin{align}
    0&\ge -\partial_t\Phi(\overline{t},\overline{X})-\delta V^\delta(\overline{Y})-a_n-2L(t-t_0)\nonumber\\
    &\hspace{2cm}+\mathbb{E}\left[ H\left( \frac{\overline{X}}{\e}, \frac{2 (\overline{X}/\e-\overline{Y})}{\e\eta}+2L(\overline{X}-X_0)+X_n+D\Phi(\overline{t},\overline{X})    \right)\right]\nonumber\\
    &\hspace{2cm}\qquad -\mathbb{E}\left[ H\left(\overline{Y},D\Phi(t_0,X_0)+\frac{2(\overline{X}/\e-\overline{Y})}{\e\eta} -2\overline{Y}\e^2 -\frac{Y_n}{\e} \right)   \right]\nonumber\\
    &\ge -\partial_t\Phi(\overline{t},\overline{X})-\delta V^\delta(\overline{Y})-a_n-2L(t-t_0)\nonumber\\
    &\; -C_H\mathbb{E}\bigg[\left(1+2Lr+\frac{1}{n}+2\e^2\left|\overline{Y}\right|+\left|\frac{Y_n}{\e}\right|+|D\Phi(t_0,X_0)|+|D\Phi(\overline{t},\overline{X})|+\left|\frac{4(\overline{X}/\e-\overline{Y})}{\e\eta}\right|\right)\nonumber\\
    &\quad\hspace{0.5cm} \cdot \left( \left|\frac{\overline{X}}{\e}-\overline{Y}\right|+|2L(\overline{X}-X_0)|+|X_n|+|D\Phi(\overline{t},\overline{X})-D\Phi(t_0,X_{0})|+2\e^2|\overline{Y}|+\left|\frac{Y_n}{\e}\right|\right)\bigg]\nonumber
    \end{align}

    \nd
    We now use Cauchy-Schwarz, \eqref{prox3} and the fact that $\frac{1}{n}\le \e^5$ to obtain
    \begin{align}
    &0\ge -\partial_t\Phi(t_0,X_0)-\frac{\delta^2}{3}-\delta V^\delta(\overline{Y})-\frac{1}{n}-2Lr\nonumber\\
    &\;\; -C_H\mathbb{E}\bigg[ \left(1+2Lr+\e^5 + 2\e^2|\overline{Y}| +\left|\frac{Y_n}{\e}\right|+|D\Phi(t_0,X_0)|+|D\Phi(\overline{t},\overline{X})|+\left|\frac{2(\overline{X}/\e-\overline{Y}}{\e\eta}   \right|\right)^2\bigg]^{1/2}\nonumber  \\
    & \hspace{1cm} \cdot\mathbb{E}\bigg[\left( \left|\frac{\overline{X}}{\e}-\overline{Y}\right|+ 2Lr+|X_n|+|D\Phi(t_0,X_0)-D\Phi(\overline{t},\overline{X})|+2\e^2|\overline{Y}|+\left|\frac{Y_n}{\e}\right|\right)^2\bigg]^{1/2}.\label{ineq3}
    \end{align}
    However, we have
    \begin{align*}
        \mathbb{E}\bigg[ &\left(1+2Lr+\e^5 + 2\e^2|\overline{Y}| +\left|\frac{Y_n}{\e}\right|+|D\Phi(t_0,X_0)|+|D\Phi(\overline{t},\overline{X})|+\left|\frac{2(\overline{X}/\e-\overline{Y}}{\e\eta}   \right|\right)^2\bigg]^{1/2}\\
        &\hspace{0.5cm}\le  1+2Lr+\e^{5}+2\e^2\|\overline{Y}\|_2+\frac{1}{n\e} +\|D\Phi(t_0,X_0)\|_2 +\|D\Phi(\overline{t},\overline{X})\|_2 +\frac{2}{\e\eta}\|\overline{X}/\e-\overline{Y}\|_2   \\
        &\hspace{0.5cm}\le  1+(2Lr)^2+\e^{5}+ \e C\sqrt{r^2+1}+\e^4+\frac{\delta^2}{3}+2 \|D\Phi(t_0,X_0)\|_2+2C =:M,
    \end{align*}
    where we used \eqref{prox3}, \eqref{prox2} and \eqref{prox1}, and $M$ is bounded when $\e,r,\delta$ are small. In addition,
    \begin{align*}
    \mathbb{E}\bigg[&\left(\left|\frac{\overline{X}}{\e}-\overline{Y}\right|+ 2Lr+|X_n|+|D\Phi(t_0,X_0)-D\Phi(\overline{t},\overline{X})|+2\e^2|\overline{Y}|+\left|\frac{Y_n}{\e}\right|\right)^2\bigg]^{1/2}\\
    &\hspace{1cm}\le \left\|\frac{\overline{X}}{\e}-\overline{Y}\right\|_2+ 2Lr+\|X_n\|_2+\|D\Phi(t_0,X_0)-D\Phi(\overline{t},\overline{X})\|_2+\e^2\|\overline{Y}\|_2+\frac{1}{\e}\|Y_n\|_2\\
    &\hspace{1cm}\le C\eta\e+ 2Lr+\frac{1}{n}+\frac{\delta^2}{3}+\e C+\frac{1}{n\e}\le C\e+2Lr+\frac{\delta^2}{3}.
    \end{align*}
    Using the last two inequalities in \eqref{ineq3} when $Lr, \e,\eta$ and $\delta$ are chosen to be small enough we get
    $$\frac{\theta}{2}\ge -\partial_t\Phi(t_0,X_0)-\delta V^\delta(\overline{Y}).$$
    However, by Proposition~\ref{effHm}, for any $p\in \R^d$, $-\delta v^\delta(\cdot,p)$ converges uniformly to $\overline{H}(p)$ as $\delta\to 0$, hence $-\delta V^\delta(\overline{Y})$ converges to $\mathbb{E}[\overline{H}(D\Phi(t_0,X_0))]$ as $\delta\to0$.
    This contradicts \eqref{"wrong" ineq} if $\delta$ is small enough, so we may proceed with case 2.
    \vspace{2mm}

    \nd
    \textit{Case 2.} We now assume that $(\overline{t},\overline{X})$ is on the boundary of $\overline{A}_r$. Note, also, that it follows from \eqref{prox1} that $\|\overline{Y}\|_2<\e^{-3/2}$. We will get a contradiction by proving that
    \be\label{ts}
    \begin{split}
\hat{U}^\e(t,X)-\Phi(t,X)-\e V^\delta(Y)&-\frac{\|X/\e-Y\|_2^2}{\eta}-L\left(\|X-X_0\|_2^2+|t-t_0|^2\right)\\
&\quad-\e^3\|Y\|_2^2-\mathbb{E}[X\cdot X_n]-\mathbb{E}[Y\cdot Y_n]+a_nt<-\frac{Lr^2}{2},
    \end{split}
    \ee
    for any $(t,X)$ on the boundary of $\overline{A}_r$, any $Y$ such that $\|Y\|_2<\e^{-3/2}$, any $\e$ small enough and any $n\in\mathbb{N}$ with $\frac{1}{n}<\e^5$. Indeed, if \eqref{ts} is true, then since $(\overline{t},\overline{X},\overline{Y})$ is a position of maximum, we also have
    $$\hat{U}^\e(t_0,X_0)-\Phi(t_0,X_0)-\e V^\delta\left(\frac{X_0}{\e}\right)-\e\|X_0\|_2^2-\mathbb{E}\left[X_0\cdot X_n+\frac{X_0}{\e}\cdot Y_n\right] +a_nt<-\frac{Lr^2}{2}.$$
    By sending $n\to +\infty$ and then $\e\to 0$, we get $0=\hat{U}(t_0,X_0)-\Phi(t_0,X_0)<-\frac{Lr^2}{2}$, which is a contradiction.
    \vspace{1mm}

    \nd
    To show \eqref{ts}, we again argue by contradiction. Suppose that for any $\e>0$, there exist $(t^\e,X^\e)$ on the boundary of $\overline{A}_r$ and $Y^\e\in\mathcal{H}$ with $\|Y^\e\|_2<\e^{-3/2}$ such that
    \be\nonumber
    \begin{split}
\hat{U}^\e(t^\e,X^\e)-\Phi(t^\e,X^\e)-\e V^\delta(Y^\e)&-\frac{\|X^\e/\e-Y^\e\|_2^2}{\eta}-L\left(\|X^\e-X_0\|_2^2+|t^\e-t_0|^2\right)\\
&\quad-\e^3\|Y^\e\|_2^2-\mathbb{E}[X^\e\cdot X_n]-\mathbb{E}[Y^\e\cdot Y_n]+a_nt^\e\ge -\frac{Lr^2}{2},
    \end{split}
    \ee
    for some $n\in \mathbb{N}$ such that $\frac{1}{n}<\e^5$. Since $(t^\e,X^\e)$ is on the boundary of $\overline{A}_r$, we have the inequality $L\left(\|X^\e-X_0\|_2^2+|t^\e-t_0|^2\right)\ge Lr^2$. In addition, $\hat{U}(t^\e,X^\e)\le \Phi(t^\e,X^\e)$, so we get from the above
    $$\hat{U}^\e(t^\e,X^\e)-\hat{U}(t^\e,X^\e)-\e V^\delta(Y^\e)-\mathbb{E}[X^\e\cdot X_n]-\mathbb{E}[Y^\e\cdot Y_n]+a_n t^\e\ge \frac{Lr^2}{2},$$
    or
    $$\sup_{(t,X)\in \overline{A}_r}|U^\e(t,\mathcal{L}(X))-U(t,\mathcal{L}(X))|-\e V^\delta(Y^\e)+\|X^\e\|_2\|X_n\|_2+\|Y^\e\|_2 \|Y_n\|_2+|a_n||t^\e|\ge \frac{Lr^2}{2},$$
    or
    \be\label{ineq101}\nonumber
    \sup_{(t,X)\in \overline{A}_r}|U^\e(t,\mathcal{L}(X))-U(t,\mathcal{L}(X))|-\e V^\delta(Y^\e)+\|X^\e\|_2\frac{1}{n}+\|Y^\e\|_2 \frac{1}{n}+|t^\e|\frac{1}{n}\ge \frac{Lr^2}{2}.
    \ee
    We now send $\e\to 0$ and $n\to \infty$ to get a contradiction due to the uniform convergence of $U^\e$ to $U$, the boundedness of $V^\delta$ and the uniform boundedness of $(\|X^\e\|_2)_{\e>0}$ and $(t^\e)_{\e >0}$.
\end{proof}

\subsection{Proof of Theorem \ref{thm:main1} and optimality of the rate $O(\e)$.}\label{proofmain1}

\nd
We collect the results of the previous subsections to prove Theorem \ref{thm:main1}.

\begin{proof}[Proof of Theorem \ref{thm:main1}]
    Combining the result of Theorem \ref{thm: unif conv} and the proof presented in subsection \ref{charoflimit}, the family $(U^\e)_{\e>0}$ converges uniformly and with rate $O(\e)$ (equation \eqref{rate1}) to an $L$-viscosity solution $U$ of \eqref{HJB}. By Theorem \ref{comprep}, \eqref{HJB} has a unique $L$-viscosity solution. This fully characterizes the limit $U$. The optimality of the rate $O(\e)$ follows from Proposition \ref{prop:opex} below.
\end{proof}

\nd
We finally present an example in $\R^d$ adapted from \cite[Proposition 4.3]{MitakeTranYu2019}, which shows the optimality of $O(\e)$ rate in Theorem \ref{thm:main1}.
\begin{proposition}\label{prop:opex}
Let $H(x,p)=\frac{|p|^2}{2}+V(x)$, where $V\in C(\mathbb{T}^d)$ satisfies $\max_{x\in \mathbb{T}^d} V(x)=0$ and $V(x)\leq -1$ for all $x\in B\left(0,\frac{1}{3}\right)$. Suppose $G(m)\equiv 0$ for all $m\in \mathcal{P}_2$. For $\varepsilon>0$, let $U^\varepsilon$ be the solution to \eqref{HJBe} associated with this Hamiltonian $H$, and let $U$ be the solution to \eqref{HJB} with the effective Hamiltonian $\bar{H}:\mathbb{R}^d\to\mathbb{R}$ corresponding to $H$.

Then $U^\varepsilon$ converges uniformly to $U\equiv 0$ on $[0,T]\times \mathcal{P}_2$ as $\varepsilon\to 0$, and
    \[
    U^\ep(T-1,\delta_0) \geq \frac{\e}{6}.
    \]
\end{proposition}
\begin{proof}
Since $\bar H(0)=0$ (see \cite[Lemma 4.21]{tran_hamilton-jacobi_2021}), the function $U\equiv 0$ solves \eqref{HJB}. Indeed, the lifted function $\hat{U} : [0,T] \times \mathcal{H} \to \mathbb{R}$ defined by $\hat{U}(t,X) \equiv 0$ satisfies
\[ \begin{cases}
-\partial_t \hat{U}(t,X) + \hat{\bar{H}}(D\hat{U}(t,X)) = 0, & (t,X) \in [0,T) \times \mathcal{H}, \\ \hat{U}(T,X) = 0, & X \in \mathcal{H},
\end{cases}
\]
where $\hat{\bar{H}}(P)=\mathbb{E}[\bar{H}(P)]\;\; \text{for }P\in \mathcal{H}$. By the value function representation of $U^\e$,
\[
U^\e(T-1,\delta_0)
=\inf_{\alpha\in\mathcal A} \mathbb{E}\left[\int_0^1 \left(\frac{|\alpha_s|^2}{2}-V \left(\frac{X_s}{\e}\right)\right)\,ds\right],
\qquad dX_s=\alpha_s\,ds,\ \ X_0=0.
\]
Fix $\alpha\in\mathcal A$ and let $X$ be the associated controlled process. Define
\[
S=\left\{\omega\in\Omega: X_s(\omega) \in B\left(0, \frac{\ep}{3}\right) \text{ for all } s \in \left[0, \frac{\e}{3}\right]\right\}.
\]
If $\omega\in S$, then $V(X_s/\e)\le -1$ on $[0,\e/3]$, hence
\[
\int_0^1 \frac{\left|\alpha_s\right|^2}{2}-V\left(\frac{X_s}{\e}\right)\,ds \geq  \int_0^\frac{\ep}{3}-V\left(\frac{X_s}{\e}\right)\,ds \geq \frac{\e}{3}.
\]
If $\omega\in S^c$, let $t_0(\omega)\le \e/3$ be the first exit time from $B(0,\e/3)$. Then $|X_{t_0}|=\e/3$ and
\[
\int_0^1 \frac{\left|\alpha_s\right|^2}{2}-V\left(\frac{X_s}{\e}\right)\,ds \geq  \int_0^{t_0}\frac{\left|\alpha_s\right|^2}{2}\,ds \geq \frac{1}{2t_0} \left(\int_0^{t_0} \alpha_s\,ds\right)^2 =\frac{\left|X_{t_0}\right|^2}{2t_0} \geq \frac{\e}{6}.
\]
Therefore,
\[
\mathbb{E}\left[\int_0^1 \frac{\left|\alpha_s\right|^2}{2}-V\left(\frac{X_s}{\e}\right)\,ds \right] \geq \frac{\e}{3}\mathbb{P}\left(S\right)+\frac{\e}{6}\mathbb{P}\left(S^c\right) \geq \frac{\e}{6}.
\]
Taking the infimum over $\alpha$ completes the proof.
\end{proof}

\section{Extension to Macroscopic and Measure-Dependent Lagrangians}\label{sec:generall}

\nd
In this section, we assume {\rm (A1)} and {\rm (A2)}. The Hamiltonian is now allowed to depend on the slow variable $x \in \mathbb{R}^d$ and probability measure $m \in \mathcal{P}_2$, and the constants appearing in (A1') are uniform with respect to $x\in \mathbb{R}^d, m \in \mathcal{P}_2$. We extend the result from the previous case, obtained under assumptions (A1') and (A2), to this more general setting. The rate of convergence in this case is $O(\sqrt{\e})$.


\subsection{Equivalent formulations of $U$}
As in the case of $U^\varepsilon$, by the dynamic programming principle \cite{djete2022mckean1}, the solution $U$ to the limiting equation \eqref{HJB4} admits a variational representation. In particular, $U$ satisfies \eqref{formula}. We also provide an alternative formulation of $U$ and a velocity bound for nearly optimal paths. The proof is omitted as it follows the same argument as in Propositions~\ref{prop:Uepm2new} and \ref{prop:UeqU2}.

\begin{proposition}\label{prop:Uformula}
(1)  For any $t \in [0, T]$ and $m \in \mathcal{P}_2$,
\[
\begin{aligned}
U(t, m)=&\inf_{\alpha\in\mathcal{A}}\left\{\mathbb{E}\left[\int^T_t \bar{L}\left(X_s,\alpha_s, \mathcal{L}(X_s) \right)\, ds \right]+G(\mathcal{L}(X_T))\right\},\\
=&\inf_{\nu\in \mathcal{P}_2}\left\{\inf_{(X,Y)\in\mathcal{E}^\ast_{m,\nu}}\left\{\inf_{\gamma\in\Gamma(X,Y)}\mathbb{E}\left[\int^T_t\bar{L}\left(\gamma_s, \dot{\gamma}_s,\mathcal{L}(\gamma_s)\right) \, ds\right]\right\}+G(\nu)\right\}.
\end{aligned}
\]
Here, $\mathcal{A}$ is the control set defined in~\eqref{eqn:cs}, and $X_s$ solves $dX_s=\alpha_s\,ds$. $\mathcal{E}_{m,\nu}^\ast$ is defined in \eqref{eqn:estar}, and $\Gamma(X,Y)$ is defined in \eqref{eqn:gammaxy}.

\nd
(2) Let $t \in [0, T)$ and $\delta \in \left(0, T-t\right)$. Suppose $(\nu_\delta,X^\delta, Y^\delta,\gamma^\delta)$ is a $\delta$-minimizer for $U(t,m)$.
Then there exists a constant $M_1=M_1(C_{2, H}, C_G, K_0)>0$ such that
\[
    \mathbb{E} \left[\int_{t}^T\left|\dot{\gamma}^\delta_s\right|^2\,ds\right] \leq M_1(T-t).
\]
\end{proposition}

\subsection{Convergence of the cost function $h^\e$}
The rate of convergence of $h^\e$ is summarized below as a direct consequence of Proposition~\ref{prop:subsuperm}. In particular, the constant $C$ in the $O(\varepsilon)$ convergence rate is independent of $c, \mu$. Moreover, the limit can be characterized in terms of the effective Lagrangian $\bar{L}$, as established in the existing literature (see, for example, \cite{LionsPapaVara1987,Weinan1991}).

\begin{proposition}\label{prop:rocmu}
Let $\ep>0,\ 0\leq t_1 < t_2\leq T, x,y, c\in\mathbb{R}^d, \mu \in \mathcal{P}_2$. Then the limit
\[
\bar{h} \left(c,\mu;t_1,  t_2,x,y\right):=\lim_{\ep\to 0^+} \ep h\left(c,\mu; \frac{t_1}{\ep}, \frac{t_2}{\ep},\frac{x}{\ep},\frac{y}{\ep}\right)= \lim_{\ep\to 0^+}  h^\e (c,\mu;t_1, t_2, x,y )
\]
exists. Moreover, there exists a constant $C=C(d, C_{1,H},C_{2,H},K_0)>0$ independent of $c, \mu, t_1, t_2, x, y$ such that
\[
\begin{aligned}
&\left|\bar{h}\left(c,\mu; t_1,t_2,x,y\right)- h^\e\left(c,\mu;t_1, t_2, x, y\right)\right|\leq C\left(1+\frac{|x-y|^2}{(t_2-t_1)^2}\right) \ep.
\end{aligned}
\]
Finally, the limit admits the representation
\[
\begin{aligned}
\bar{h}(c,\mu;t_1,t_2,x,y)
&= (t_2 - t_1)\,
\bar{L}\!\left(c,\frac{y-x}{t_2 - t_1}, \mu\right)\\
&=\inf\left\{\int_{t_1}^{t_2} \bar{L}(c,\dot{\gamma}(s), \mu)\,ds:\gamma \in \mathrm{AC}\left([t_1, t_2];\R^d\right), \gamma(t_1)=x, \gamma(t_2)=y\right\}.
\end{aligned}
\]
\end{proposition}

\subsection{Rate of convergence of $U^\e$ to $U$}

Before proving the convergence rate, we establish the following lemma, which is analogous to Lemma \ref{lem:fixest} and shows that the same decomposition property holds for $\bar{L}$: the minimal action over a fixed time partition can be written as a sum of local minimal costs on each subinterval.

\begin{lemma}\label{lem:lbarpart}
Let $\e>0, t\in[0,T), \nu\in\mathcal P_2, (X,Y)\in\mathcal{E}_{\mu_0,\mu}^\ast$, and $\gamma\in\Gamma(X,Y) $ satisfy
\[\mathbb{E}\left[\int_t^T\left|\dot{\gamma}_s\right|^2\,
ds\right] \leq M_1(T-t),
\]
with $M_1=M_1(C_{2, H}, C_G, K_0)>0$ from Proposition~\ref{prop:Uepm2new}. Let $N \in \mathbb{N}$, $\Delta t := (T-t)/N$, and $t_k := t + k\Delta t$ for $k=0,\dots,N$. Define $\Gamma(X,Y,\gamma)$ as in Lemma \ref{lem:fixest}, that is, $\Gamma(X,Y;\gamma)
:=
\{\eta\in \Gamma(X,Y) : \eta_{t_k}=\gamma_{t_k}\,\, \text{a.s. for all } k=0,\dots,N\}$. For $\eta \in \Gamma(X,Y; \gamma)$, define
\[
\bar{J}(\eta)
:=
\sum_{k=0}^{N-1}
\mathbb E \left[
\int_{t_k}^{t_{k+1}}
\bar L \left(\gamma_{t_k},\dot{\eta}_s,\mathcal{L}\left(\gamma_{t_k}\right)\right)\,dt\right]
\]
Then
\[
\inf_{\eta \in\Gamma(X,Y;\gamma)} \bar J(\eta)
=
\sum_{k=0}^{N-1}
\mathbb{E} \left[
\bar{h} \left(\gamma_{t_k},\mathcal{L}\left(\gamma_{t_k}\right);t_k,t_{k+1},
\gamma_{t_k},\gamma_{t_{k+1}}\right)\right].
\]
The infimum is attained by the piecewise linear interpolation
\begin{equation}\label{eqn:gambar}
\bar{\gamma}(s,\omega):=\gamma(t_k,\omega)+\frac{s-t_k}{t_{k+1}-t_k}\left(\gamma(t_{k+1}, \omega)-\gamma(t_k,\omega)\right), \quad \text{ for } \omega \in \Omega, \,\,s \in \left[t_k, t_{k+1}\right], \, \,k=0,\dots,N-1,
\end{equation}
and $\displaystyle \mathbb{E}\left[\int_t^T \left|\dot{\bar{\gamma}}_s\right|^2 \,ds\right] \leq M_1(T-t)$.
\end{lemma}

\begin{proof}
By definition,
\[
\inf_{\eta \in \Gamma(X,Y;\gamma)}
\bar{J}(\eta)
\geq
\sum_{k=0}^{N-1}
\mathbb{E} \left[
\bar{h}
\left(\gamma_{t_k},\mathcal{L}(\gamma_{t_k});t_k,t_{k+1},
\gamma_{t_k},\gamma_{t_{k+1}}\right)\right].
\]
By Proposition \ref{prop:rocmu},
\[
\bar{h}
\bigl(\gamma_{t_k},\mathcal{L}(\gamma_{t_k});t_k,t_{k+1},
\gamma_{t_k},\gamma_{t_{k+1}}\bigr)=\left(t_{k+1}-t_k\right) \bar{L}\left(\gamma_{t_k}, \frac{\gamma_{t_{k+1}}-\gamma_{t_k}}{t_{k+1}-t_{k}}, \mathcal{L}(\gamma_{t_k})\right).
\]
Consider the path $\bar{\gamma}:[t,T] \times \Omega \to \R^d$ defined in \eqref{eqn:gambar}. Then $\bar{\gamma}\in \Gamma(X,Y;\gamma)$, and
\[
\inf_{\eta \in \Gamma(X,Y,\gamma)}
J(\eta) \leq \bar{J}(\bar{\gamma})=\sum_{k=0}^{N-1}
\mathbb{E} \left[\bar{h}\left(\gamma_{t_k},\mathcal{L}(\gamma_{t_k}); t_k, t_{k+1}, \gamma_{t_k}, \gamma_{t_{k+1}}\right) \right].
\]
Hence, $\bar{\gamma}$ attains the infimum.
We compute
\[
\mathbb{E} \left[\int_t^{T} \left|\dot{\bar{\gamma}}_s\right|^2\,ds\right] = \frac{1}{\Delta t}\sum_{k=0}^{N-1}\mathbb{E}\left|\gamma_{t_{k+1}}-\gamma_{t_k}\right|^2 \leq \sum_{k=0}^{N-1}\mathbb{E} \left[\int_{t_k}^{t_{k+1}} \left|\dot{\gamma}_s\right|^2\,ds\right] \leq M_1(T-t).
\]
\end{proof}

We are now ready to prove the convergence rate in the general setting under assumptions (A1) and (A2).

\begin{proof}[Proof of Theorem \ref{thm:multiscale}]
If $t_0=T$, by Propositions~\ref{prop:Uepm2new}, \ref{prop:UeqU2}, \ref{prop:U1eqUeqU2} and \ref{prop:Uformula}, the admissible pair $(X, Y)$ must satisfy $X=Y$, and the only admissible measure $\nu \in \mathcal{P}_2$ is $\nu=m$. It follows that $U^\ep(T, m)=G(m)=U(T,m)$.

Assume from now on that $t<T$ and let $\delta \in \left(0, T- t\right)$.

\nd
\textbf{Step 1: Lower bound for $U^\ep-U$.} By Propositions~\ref{prop:Uepm2new}, \ref{prop:UeqU2}, and \ref{prop:U1eqUeqU2}, let $(\nu,X,Y,\gamma)$ be $\delta$-optimal for $U^\e(t,m)$. Then
\begin{equation}\label{eqn:nopuep}
 \mathbb{E}\left[\int_{t}^T L\left( \frac{\gamma_s}{\e}, \gamma_s, \dot{\gamma}_s, \mathcal{L}(\gamma_s) \right)\, ds \right] +G(\nu) \leq U^\e(t, m)+\delta, \qquad
\mathbb{E}\left[\int_{t}^T|\dot{\gamma}_s|^2 \,ds\right] \leq M_1(T-t),
\end{equation}
    for some constant $M_1=M_1(C_{2, H}, C_G, K_0)>0$.

Consider a time partition $t_k := t + \frac{k}{N}(T-t)$ for $ k=0,\dots,N$, where $N \in \mathbb{N}$ will be chosen later. With $J_\ep$ as in Lemma~\ref{lem:fixest}, we may write
\begin{equation}\label{eqn:opfixmu}
\mathbb{E}\left[\int_{t}^T L\left( \frac{\gamma_s}{\e}, \gamma_s, \dot{\gamma}_s, \mathcal{L}(\gamma_s) \right)\, ds \right]\geq J_\e (\gamma) - \sum_{k=0}^{N-1} E_k,
\end{equation}
where
\[
E_k:=\mathbb{E}\left[\int_{t_k}^{t_{k+1}} \left|L\left( \frac{\gamma_s}{\e}, \gamma_s, \dot{\gamma}_s, \mathcal{L}(\gamma_s) \right)-
L\left(\frac{\gamma_s}{\ep}, \gamma_{t_k}, \dot{\gamma}_s, \mathcal{L}(\gamma_{t_k})\right)\right|\,ds\right]. \]
By Lemma \ref{lem:fixest},
\begin{equation}\label{eqn:tpartest}
\begin{aligned}
J_\e(\gamma) \geq \inf_{\eta\in\Gamma(X,Y;\gamma)}
J_\e (\eta) =\sum_{k=0}^{N-1}
\mathbb E\left[h^\e\left(\gamma_{t_k}, \mathcal{L}(\gamma_{t_k}); t_k,t_{k+1},
\gamma_{t_k},\gamma_{t_{k+1}}\right)\right].
\end{aligned}
\end{equation}

\nd
Next, Proposition~\ref{prop:rocmu} implies that
\begin{equation}\label{eqn:memukmbar}
\begin{split}
\mathbb{E}&\left[h^\e\left(\gamma_{t_k}, \mathcal{L}(\gamma_{t_k}); t_k,t_{k+1},
\gamma_{t_k},\gamma_{t_{k+1}}\right)\right]\\
&\qquad\geq \mathbb{E}\left[\bar{h}\left(\gamma_{t_k}, \mathcal{L}(\gamma_{t_k}); t_k,t_{k+1},
\gamma_{t_k},\gamma_{t_{k+1}}\right)\right] -C\left(1+\frac{ \mathbb{E}\left|\gamma_{k+1}-\gamma_k\right|^2}{|t_{k+1}-t_k|^2}\right) \ep,
\end{split}
\end{equation}
for some constant $C=C(d, C_{1,H},C_{2,H},K_0)>0$.
By H\"older's inequality,
\[
\sum_{k=0}^{N-1}\frac{ \mathbb{E}\left|\gamma_{k+1}-\gamma_k\right|^2}{|t_{k+1}-t_k|^2}\leq \frac{N}{T-t} \mathbb{E} \left[\int_{t}^T\left|\dot{\gamma}_s\right|^2\right]\,
ds \leq M_1 N,
\]
where we use the fact that $t_{k+1}-t_k=\frac{T-t_0}{N}$ and \eqref{eqn:nopuep}. Combining this with \eqref{eqn:memukmbar} yields
\begin{equation}\label{eqn:rocsumerr}
\sum_{k=0}^{N-1}
\mathbb E\left[h^\e\left(\gamma_{t_k}, \mathcal{L}(\gamma_{t_k}); t_k,t_{k+1},
\gamma_{t_k},\gamma_{t_{k+1}}\right)\right]\geq \sum_{k=0}^{N-1}\mathbb{E}\left[\bar{h}\left(\gamma_{t_k}, \mathcal{L}(\gamma_{t_k}); t_k,t_{k+1},
\gamma_{t_k},\gamma_{t_{k+1}}\right)\right] -C N\ep,
\end{equation}
for some constant $C=C(d, C_{1,H},C_{2,H},C_G, K_0)>0$.

Let $\bar\gamma:[t_0,T] \times \R^d \to \R^d$ be the piecewise linear interpolation of $\gamma$ on the partition $\{t_k\}_{k=0}^N$, namely
\[
\bar{\gamma}(t,\omega):=\gamma(t_k,\omega)+\frac{s-t_k}{t_{k+1}-t_k}\left(\gamma(t_{k+1}, \omega)-\gamma(t_k,\omega)\right), \qquad \text{ for } t \in \left[t_k, t_{k+1}\right].
\]
By Lemma \ref{lem:lbarpart},
\begin{equation}\label{eqn:admgabar}
\begin{aligned}
&\sum_{k=0}^{N-1}
\mathbb{E}\left[\bar{h}\left(\gamma_{t_k}, \mathcal{L}(\gamma_{t_k}); t_k,t_{k+1},
\gamma_{t_k},\gamma_{t_{k+1}}\right)\right]=\inf_{\eta\in \Gamma(X,Y;\gamma)} \bar J(\eta)=\bar{J}(\bar{\gamma})\\
\geq & \sum_{k=0}^{N-1}
\mathbb{E} \left[\int_{t_k}^{t_{k+1}}
\bar{L}\left(\bar{\gamma}_s,
\dot{\bar{\gamma}}_s,
\mathcal{L}\left(\bar{\gamma}_s\right)
\right)\,ds\right]-\sum_{k=0}^{N-1} \bar{E}_k,
\end{aligned}
\end{equation}
where
\[
\bar{E}_k:= \mathbb{E} \left[\int_{t_k}^{t_{k+1}}\left|
\bar{L}\left(\gamma_{t_k},
\dot{\bar{\gamma}}_s,
 \mathcal{L}(\gamma_{t_k})
\right)-
\bar{L}\left(\bar{\gamma}_s,
\dot{\bar{\gamma}}_s,
\mathcal{L}\left(\bar{\gamma}_s\right)
\right)\right|\,ds\right].
\]
Combining \eqref{eqn:opfixmu}, \eqref{eqn:tpartest}, \eqref{eqn:rocsumerr}, and \eqref{eqn:admgabar}, we obtain
\begin{equation}\label{eqn:intstepuep}
    \mathbb{E}\left[\int_{t}^T L\left( \frac{\gamma_s}{\e}, \gamma_s, \dot{\gamma}_s, \mathcal{L}(\gamma_s) \right)\, ds \right] \geq \mathbb{E} \left[\int_t^T
\bar{L}\left(\bar{\gamma}_s,
\dot{\bar{\gamma}}_s,
\mathcal{L}\left(\bar{\gamma}_s\right)
\right)\,ds\right]-\sum_{k=0}^{N-1} E_k -\sum_{k=0}^{N-1} \bar{E}_k -CN \ep,
\end{equation}
for some constant $C=C(d, C_{1,H},C_{2,H},C_G, K_0)>0$.
By Lemmas \ref{lem:fixest}, \ref{lem:lbarpart}, and \eqref{eqn:nopuep},
\begin{equation}\label{eqn:eksum}
\sum_{k=0}^{N-1}E_k \leq  \frac{C(1+M_1)(T-t)^2}{N}, \qquad
\sum_{k=0}^{N-1}\bar{E}_k\leq \frac{C(1+M_1)(T-t)^2}{N}.
\end{equation}Combining \eqref{eqn:nopuep}, \eqref{eqn:intstepuep}, and \eqref{eqn:eksum}, we deduce that
\[
\begin{aligned}
U^\ep(t,m)+\delta & \geq \mathbb{E} \left[\int_t^T
\bar{L}\left(\bar{\gamma}_s,
\dot{\bar{\gamma}}_s,
\mathcal{L}\left(\bar{\gamma}_s\right)
\right)\,ds\right] + G(\nu) -\frac{C(T-t)^2}{N}-CN\ep\\
& \geq U(t,m ) -\frac{C(T-t)^2}{N}-CN\ep.
\end{aligned}
\]
for some constant $C=C(d,C_{1,H},C_{2,H},C_G, K_0)>0$.
Letting $\delta \to 0$, we obtain
\[
U^\ep(t,m) \geq U(t,m) -\frac{C(T-t)^2}{N}-CN\ep.
\]
If $T-t \leq \sqrt{\ep}$, choose $N=1$. Then
\[
U^\ep(t,m) \geq U(t,m) -C(T-t)^2-C\ep \geq U(t,m) -C\ep.
\]
If $T-t > \sqrt{\ep}$, choose $N=\left\lfloor\frac{T-t}{\sqrt{\ep}}\right\rfloor+1$. Then
\[
\frac{(T-t)^2}{N} \leq (T-t)\sqrt{\ep}, \qquad \qquad N \ep \leq 2(T-t) \sqrt{\ep},
\]
and hence
\[
U^\ep(t,m) \geq U(t,m) -C(T-t) \sqrt{\ep}.
\]
Therefore,
\[
U^\ep(t,m)-U(t,m) \ge
\begin{cases}
-C(T-t)\sqrt{\ep}, & T-t \ge \sqrt{\varepsilon},\\[4pt]
-C\ep, & 0\leq T-t<\sqrt{\ep}.
\end{cases}
\]

\nd
\textbf{Step 2: Upper bound for $U^\ep-U$.} By Lemma \ref{prop:Uformula}, there exist $\nu \in \mathcal{P}_2, (X,Y)\in \mathcal{E}^\ast_{\mu_0,\mu}$, and $\gamma\in \Gamma(X,Y)$ satisfying
\begin{equation}\label{eqn:upbdvelbd}
   \mathbb{E}\left[\int^T_{t} \bar{L}\left(\gamma_s, \dot{\gamma}_s, \mathcal{L}(\gamma_s) \right) \, ds\right] +G(\nu) \leq U(t, m)+\delta, \qquad \mathbb{E}\left[ \int_{t}^T\left|\dot{\gamma}_s\right|^2\,
ds\right] \leq M_1(T-t),
\end{equation}
for some constant $M_1=M_1(C_{2, H}, C_G, K_0)>0$.

Consider the same time partition $t_k := t + \frac{k}{N}(T-t)$ for $ k=0,\dots,N$ with $N\in \mathbb{N}$ to be determined. Then
\begin{equation}\label{eqn:upbdLbar}
\mathbb{E}\left[\int^T_{t} \bar{L}\left(\gamma_s, \dot{\gamma}_s, \mathcal{L}(\gamma_s) \right) \, ds\right]\geq \bar{J}(\gamma)-\sum_{k=0}^{N-1} \bar{E}_k,
\end{equation}
where
\[
\bar{E}_k:= \mathbb{E}\left[\int_{t_k}^{t_{k+1}}\left|\bar{L}\left(\gamma_s,
\dot{\gamma}_s,
\mathcal{L}\left(\gamma_t\right)
\right)-\bar{L}\left(\gamma_{t_k},
\dot{\gamma}_s,
 \mathcal{L}(\gamma_{t_k})
\right) \right|\,ds\right].
\]
By Lemma \ref{lem:lbarpart},
\begin{equation}\label{eqn:upbdJbar}
\bar{J}(\gamma) \geq \inf_{\eta\in \Gamma(X,Y;\gamma)} \bar J(\eta) =\sum_{k=0}^{N-1}
\mathbb{E} \left[
\bar{h} \left(\gamma_{t_k}, \mathcal{L}(\gamma_{t_k});t_k,t_{k+1},
\gamma_{t_k},\gamma_{t_{k+1}}\right)\right].
\end{equation}
Using Proposition~\ref{prop:rocmu}, we obtain
\be\nonumber
\begin{split}
\mathbb{E}& \left[
\bar{h} \left(\gamma_{t_k}, \mathcal{L}(\gamma_{t_k});t_k,t_{k+1},
\gamma_{t_k},\gamma_{t_{k+1}}\right)\right]\\
 &\qquad\geq \mathbb{E} \left[h^\e \left(\gamma_{t_k}, \mathcal{L}(\gamma_{t_k});t_k,t_{k+1},
\gamma_{t_k},\gamma_{t_{k+1}}\right)\right]-C\left(1+\frac{ \mathbb{E}\left|\gamma_{k+1}-\gamma_k\right|^2}{|t_{k+1}-t_k|^2}\right) \ep,
\end{split}
\ee
for some constant $C=C(d, C_{1,H},C_{2,H},K_0)>0$. Using the velocity bound \eqref{eqn:upbdvelbd} exactly as in Step 1 yields
\begin{equation}\label{eqn:upbdsumroc}
\sum_{k=0}^{N-1}
\mathbb E\left[\bar{h}\left(\gamma_{t_k}, \mathcal{L}(\gamma_{t_k}); t_k,t_{k+1},
\gamma_{t_k},\gamma_{t_{k+1}}\right)\right]\geq \sum_{k=0}^{N-1}\mathbb{E}\left[h^\e \left(\gamma_{t_k}, \mathcal{L}(\gamma_{t_k}); t_k,t_{k+1},
\gamma_{t_k},\gamma_{t_{k+1}}\right)\right] -C N\ep,
\end{equation}
for some constant $C=C(d, C_{1,H},C_{2,H},C_G, K_0)>0$.

By Lemma \ref{lem:fixest}, there exists $\tilde{\gamma} \in \Gamma(X,Y; \gamma)$ such that
\begin{equation}\label{eqn:upbdgtvb}
\mathbb{E} \left[ \int_{t}^T \left|\dot{\tilde{\gamma}}_s\right|^2 dt \right]
\le \tilde{M} (T-t),
\end{equation}
and
\begin{equation}\label{eqn:upbdme}
\begin{aligned}
&\sum_{k=0}^{N-1}\mathbb{E}\left[h^\e \left(\gamma_{t_k}, \mathcal{L}(\gamma_{t_k}); t_k,t_{k+1},
\gamma_{t_k},\gamma_{t_{k+1}}\right)\right]+\delta \geq J_\ep (\tilde{\gamma})\\
& \qquad \qquad\geq \sum_{k=0}^{N-1}
\mathbb{E} \left[\int_{t_k}^{t_{k+1}}
L\left(\frac{\tilde{\gamma}_s}{\e},\tilde{\gamma}_s,
\dot{\tilde{\gamma}}_s,
\mathcal{L}\left(\tilde{\gamma}_s\right)
\right)\,ds\right]-\sum_{k=0}^{N-1} E_k,
\end{aligned}
\end{equation}
where
\[
    E_k:= \mathbb{E} \left[ \int_{t_k}^{t_{k+1}}\left| L\left(\frac{\tilde{\gamma}_s}{\e},\gamma_{t_k},\dot{\tilde{\gamma}}_s, \mathcal{L}(\gamma_{t_k}) \right)-L\left(\frac{\tilde{\gamma}_s}{\e},\tilde{\gamma}_s,\dot{\tilde{\gamma}}_s,\mathcal{L}\left(\tilde{\gamma}_s\right)\right)\right|\,ds \right],
\]
and $\tilde{M}=\tilde{M}(C_{1,H}, C_{2,H}, K_0, M_1)=\tilde{M}(C_{1,H}, C_{2,H}, K_0, C_G)>0$ is some constant.

Combining \eqref{eqn:upbdLbar}, \eqref{eqn:upbdJbar}, \eqref{eqn:upbdsumroc}, and \eqref{eqn:upbdme}, we obtain
\[
\mathbb{E}\left[\int^T_{t} \bar{L}\left(\gamma_s, \dot{\gamma}_s, \mathcal{L}(\gamma_s) \right) \, ds\right] \geq \mathbb{E} \left[\int_t^T
L\left(\frac{\tilde{\gamma}_s}{\e},\tilde{\gamma}_s,
\dot{\tilde{\gamma}}_s,
\mathcal{L}\left(\tilde{\gamma}_s\right)
\right)\,ds\right]-\sum_{k=0}^{N-1} E_k -\sum_{k=0}^{N-1} \bar{E}_k -CN \ep -\delta.
\]
Using the velocity bounds \eqref{eqn:upbdvelbd} and \eqref{eqn:upbdgtvb} as in Step 1,
\[
\sum_{k=0}^{N-1}E_k+\sum_{k=0}^{N-1}\bar E_k
\le \frac{C(T-t)^2}{N},
\]
and hence
\[
U^\ep(t,m) - U(t,m)\leq  \frac{C(T-t)^2}{N}+CN\e
\]
for some constant $C=C(d,C_{1,H},C_{2,H},C_G, K_0)>0$.

Choosing $N$ exactly as in Step 1, we conclude that
\[
U^\ep(t,m) - U(t,m)\leq \begin{cases}
C(T-t)\sqrt{\ep}, & T-t \ge \sqrt{\varepsilon},\\[4pt]
C\ep, & 0\leq T-t<\sqrt{\ep}.
\end{cases}
\]
Consequently,
\[
    \left|U^\ep(t,m)-U(t,m)\right| \le
\begin{cases}
C(T-t)\sqrt{\ep}, & T-t \ge \sqrt{\ep},\\[4pt]
C\ep, & 0\leq T-t<\sqrt{\ep}.
\end{cases}
\]

Finally, by Corollary \ref{cor:lipschitzestimate}, we have
\[
\left|U^\ep(t,m)-G(m)\right|=\left|U^\ep(t,m)-U^\ep(T,m)\right| \leq C_0(T-t),
\]
and similarly
\[
\left|U(t,m)-G(m)\right|=\left|U(t,m)-U(T,m)\right| \leq C_0(T-t),
\]
for some constant $C_0>0$. Therefore,
\[
\left|U^\ep(t,m)-U(t,m)\right|  \leq 2 C_0(T-t).
\]
In particular, if $0\leq T-t<\sqrt{\ep}$, we obtain
\[
\left|U^\ep(t,m)-U(t,m)\right| \leq C \min\{T-t, \ep\}
\]
for some constant $C=C(d,C_{1,H},C_{2,H},C_G, C_H,K_0)>0$, and
\[
\left|U^\ep(t,m)-U(t,m)\right|  \le
\begin{cases}
C(T-t)\sqrt{\varepsilon}, & T-t \ge \sqrt{\varepsilon},\\[4pt]
C\min\{T-t,\varepsilon\}, & 0\leq T-t<\sqrt{\varepsilon}.
\end{cases}
\]
\end{proof}




\section{Dynamic Optimal Transport Problem}\label{DOT}

\nd
In this section, we prove our results for the dynamic optimal transport problem $U^\e_{ot}$. Recall that for $\mu,\nu\in\mathcal{P}_2$ and $t\in [0,T)$, $U^\e(t,\mu,\nu)$ is the value function

\begin{equation}\nonumber
U_{ot}^\varepsilon(t,\mu,\nu)
\;:=\;
\inf_{\alpha\in\mathcal{A}_{t,\mu,\nu}}\;
\mathbb E\Bigg[\int_t^T L\Big(\frac{X_s}{\varepsilon},X_s,\alpha_s,\mathcal{L}(X_s)\Big)\,ds\Bigg],
\end{equation}
where $(X_s)_{s\in [t,T]}$ satisfies $dX_s=\alpha_s ds$, $\mathcal{L}(X_t)=\mu$ and $\mathcal{L}(X_T)=\nu$ and $\mathcal{A}_{t,\mu,\nu}$ was defined in \eqref{controlsetot}. As in the previous sections, we consider the control sets
$$\mathcal{A}_{t,\mu,\nu}^1:=\left\{ \alpha\in\mathcal{A}_1:\; \exists X_t\text{ such that }\mathcal{L}(X_t)=\mu\text{ and }\mathcal{L}(X_t+\int_t^T\alpha_sds)=\nu \right\}$$
and
$$\mathcal{A}_{t,\mu,\nu}^2:=\left\{ \alpha\in\mathcal{A}_2:\; \exists X_t\text{ such that }\mathcal{L}(X_t)=\mu\text{ and }\mathcal{L}(X_t+\int_t^T\alpha_sds)=\nu \right\}$$
where the sets $\mathcal{A}_1,\mathcal{A}_2$ were introduced in Section \ref{altcontrset}. The corresponding value functions are denoted by $U_{1,ot}^{\e}(t,\mu,\nu)$ and $U_{2,ot}^\e(t,\mu,\nu)$, respectively. Since $\mathcal{A}_2\subseteq \mathcal{A}\subseteq \mathcal{A}_1$, it is clear that
\be\label{eq1001}
U_{1,ot}^{\e}(t,\mu,\nu)\le U_{ot}^{\e}(t,\mu,\nu)\le U_{2,ot}^{\e}(t,\mu,\nu),\;\;\text{for any }(t,\mu,\nu)\in [0,T)\times \mathcal{P}_2\times\mathcal{P}_2.
\ee
We observe that $\mathcal{A}_{t,\mu,\nu}^1\supseteq\mathcal{A}_{t,\mu,\nu}^2\neq\emptyset$, since the $\sigma$-algebra $\mathcal{F}_0$ is assumed to be rich enough. Indeed, we may pick $\mathcal{F}_0$-measurable $X,Y\in L^2$ with $\mathcal{L}(X)=\mu$ and $\mathcal{L}(Y)=\nu$ and $X_s= \frac{s-t}{T-t}Y+\frac{T-s}{T-t}X$ is an admissible state process.\vspace{2mm}

\nd
The following lemma is a consequence of the quadratic growth of $L$ (equation \eqref{eqn:Lmquad}).

\begin{lemma}\label{naive estimate}
Fix $t\in [0,T)$ and $\mu,\nu\in \mathcal{P}_2$. Then,
$$\frac{1}{2C_{2,H}(T-t)}{\bf d}_2^2(\mu,\nu)-K_0(T-t)\le U_{1,ot}^\e(t,\mu,\nu)\le U_{2,ot}^\e(t,\mu,\nu) \le \frac{1}{2C_{1,H}(T-t)}{\bf d}_2^2(\mu,\nu)+K_0(T-t).$$
\end{lemma}
\vspace{3mm}

\nd
In order to prove Theorem \ref{main3}, our first step will be to show that $U^\e_{1,ot}=U^\e_{ot}$. Then we proceed with a similar proof strategy as the one we used for Theorems \ref{thm:main1} and \ref{thm:multiscale} (see Section \ref{sec:generall} for the proof of the latter). Recalling the equality $U^\e(t,\mu)=\inf_{\nu\in \mathcal{P}_2}\{ U^\e_{ot}(t,\mu,\nu)+G(\nu)\}$, the proof of the first step relies on the following lemma.

\begin{lemma}\label{infeq}
    Let $f_1,f_2\colon\mathcal{P}_2(\R^d)\to \R$ be two lower bounded functions such that
    \be\label{infcondition}
\inf_{m\in\mathcal{P}_2(\R^d)}\left\{ f_1(m)+G(m)\right\}=\inf_{m\in\mathcal{P}_2(\R^d)}\left\{ f_2(m)+G(m)\right\},
    \ee
    for any bounded and ${\bf d}_2$-Lipschitz function $G\colon \mathcal{P}_2(\R^d)\to \R$. If $f_1$ is assumed to be lower semi-continuous, then $f_1(m)\le f_2(m)$, for all $m\in \mathcal{P}_2(\R^d)$. In particular, if both $f_1,f_2$ are lower semi-continuous, then $f_1(m)=f_2(m)$, for all $m\in \mathcal{P}_2(\R^d)$.
\end{lemma}

\begin{proof}
    We argue by contradiction. Assume that $f_1(m_0)>f_2(m_0)$, for some $m_0\in \mathcal{P}_2$. For $N\in \mathbb{N}$, we let $G_N(m)=N\min\{1,{\bf d}_2(m,m_0)\}$, which is bounded and ${\bf d}_2$-Lipschitz with $G_N(m_0)=0$ and $G_{N} \ge 0$. Then, \eqref{infcondition} implies
    $$\inf_{m\in \mathcal{P}_2}\{f_1(m)+G_N(m)\}\le f_2(m_0)+G_N(m_0)=f_2(m_0).$$
    Let $\delta>0$ such that $\delta<f_1(m_0)-f_2(m_0)$. Then, there exists $m^{\delta,N}$ such that
    \be\label{ineq10001}
    f_1(m^{\delta,N})+G_N(m^{\delta,N})-\delta\le f_2(m_0).
    \ee
    Since $f_1$ is lower bounded, this inequality gives a uniform in $N$ upper bound for $G_N(m^{\delta,N})$, which in turn implies $m^{\delta,N}\xrightarrow{N\to \infty}m_0$ in $\textbf{d}_2$ metric. Sending $N\to \infty$ in \eqref{ineq10001} we get, by the semi-continuity of $f_1$ and the non-negativity of $G_N$,
    $$f_1(m_0)-\delta\le \liminf_{N\to \infty} f_1(m^{\delta,N})-\delta\le \liminf_{N\to \infty} (f_1(m^{\delta,N})+G_N(m^{\delta,N}))-\delta \le f_2(m_0).$$
    This contradicts the choice of $\delta$.
    \vspace{2mm}

    \nd
    If $f_2$ is also lower semi-continuous, then $f_2\le f_1$. Therefore, combining both inequalities, $f_1=f_2$.
\end{proof}

\subsection{Continuity properties of the value function}\label{DOTcont}
We prove that $U^\e_{1,ot},U^\e_{2,ot}\colon [0,T)\times \mathcal{P}_2\times\mathcal{P}_2\to \R$ are continuous functions. To lighten the notation in the proofs, for this subsection only, we assume that $H$ satisfies Assumption (A1'). The case where $H$ satisfies (A1) can be treated by adapting the arguments that follow.
\vspace{2mm}

\nd
For $i=1,2$, we consider the sets
$$\mathcal{A}^{i,\delta}_{t,\mu,\nu}=\left\{ \alpha\in\mathcal{A}_{t,\mu,\nu}^i:\; \alpha_s=0,\; s\in [t,t+\delta]\right\},\;\; \text{for }\delta\in (0,T-t).$$
We first show that the value function of the optimization problems $U^\e_{i,ot},\;i=1,2$ remains the same if we restrict to controls in $\mathcal{A}_{t,\mu,\nu}^{i,\delta},\;i=1,2$, for $\delta>0$. This property will later allow us to prove the continuity with respect to the measure arguments.

\begin{proposition}\label{altformulaot}
    Fix $\mu,\nu\in \mathcal{P}_2$, $t<T$ and $i\in \{1,2\}$. Then, in the above setup,
    \be\label{approxot}
    U_{i,ot}^\e(t,\mu,\nu)=\inf_{\delta>0,\;\alpha\in \mathcal{A}^{i,\delta}_{t,\mu,\nu}}\left\{ \mathbb{E}\left[ \int_t^TL\left( \frac{X_s}{\e},\alpha_s\right)ds\right]\right\}.
    \ee
    Furthermore, for any $R>0$ such that $R>\int|x|^2\mu(dx)$, there exists a constant $C:=C_{t,\nu,R}$ depending on the quantity $\sup_{m\in \mathcal{P}_2:\; {\bf d}_2^2(\delta_0,m)\le R}U_{i,ot}^\e(t,m,\nu)$ and $T-t$ such that for any $\eta>0$ with $C\eta<T-t$, there exists an $\eta$-optimal control $\alpha^\eta\in \mathcal{A}^{i,C\eta}_{t,\mu,\nu}$.
\end{proposition}

\begin{proof}
    The fact that $U^\e_{i,ot}(t,\mu,\nu)$ is less or equal that the right hand side of \eqref{approxot} is obvious, so it suffices to prove the reverse inequality. Let $\mathbb{N}\ni n\ge 1$, $\delta\in (0,T-t)$ and $(\alpha_s^n)_{s\in [t,T]}$ be a $\frac{1}{n}$-optimal control for $U^\e_{i,ot}$ and $(X^n_s)_{s\in [t,T]}$ the associated state process. We denote $\psi(s)=\frac{T-t}{T-t-\delta}(s-t-\delta)+t$ and we define the control
    $$\alpha_s=\begin{cases}
        \frac{T-t}{T-t-\delta}\alpha_{\psi(s)}^n, &\text{ if }s\in [t+\delta, T],\\
        0, &\text{ if }s\in [t,t+\delta],
    \end{cases}$$
    and the process
    $$X_s=\begin{cases}
        X_{\psi(s)}^n,&\text{ if }s\in [t+\delta,T],\\
        X_t^n,&\text{ if }s\in [t,t+\delta].
    \end{cases}$$
    It straightforward to see that $dX_s=\alpha_sds$ with $X_t=X_t^n\sim \mu$ and $X_T=X^n_{\psi(T)}=X^n_T\sim \nu$, therefore $(\alpha_s)_{s\in [t,T]}\in \mathcal{A}^{i,\delta}_{t,\mu,\nu}$. In the following inequalities the constant $C$ might become larger from line to line and remains independent of $n,\delta$. We have
    \begin{align*}
        \mathbb{E}\bigg[ \int_t^TL\big( \frac{X_s}{\e},&\;\alpha_s\big)ds \bigg]=\mathbb{E}\left[ \int_t^{t+\delta}L\left( \frac{X_s}{\e},0\right)ds + \int_{t+\delta}^TL\left( \frac{X_s}{\e},\alpha_s\right)ds \right]\\
        &=\delta\mathbb{E}[L(X_t/\e,0)]+\frac{T-t-\delta}{T-t}\mathbb{E}\left[ \int_t^TL\left( \frac{X_s^n}{\e},\alpha_s^n\frac{T-t}{T-t-\delta}\right)ds \right]\\
        &\le \delta K_0+\frac{T-t-\delta}{T-t}\mathbb{E}\left[ \int_t^TL\left( \frac{X_s^n}{\e},\alpha_s^n\right)ds \right]\\
        &\qquad+\frac{T-t-\delta}{T-t}\mathbb{E}\left[ \int_t^T\left(L\left( \frac{X_s^n}{\e},\alpha_s^n\frac{T-t}{T-t-\delta}\right)  - L\left( \frac{X_s^n}{\e},\alpha_s^n\right)\right)ds  \right]\\
        &\le \delta K_0+\frac{T-t-\delta}{T-t}(U_{i,ot}^\e(t,\mu,\nu)+1/n)\\
        &\quad +\frac{T-t-\delta}{T-t}\cdot \frac{\delta C_0}{T-t-\delta}\mathbb{E}\left[ \int_t^T\left(1+\frac{\delta}{T-t-\delta}|\alpha_s^n|\right)|\alpha_s^n|ds\right]\\
        &\le \delta \left(K_0-\frac{U_{i,ot}^\e(t,\mu,\nu)}{T-t}\right)+\frac{1}{n}+U_{ot}^\e(t,\mu,\nu) +\frac{\delta C}{T-t}\mathbb{E}\left[ \int_t^T(|\alpha_s^n|^2+1)ds\right]\\
        &\le\delta \left(K_0-\frac{U_{i,ot}^\e(t,\mu,\nu)}{T-t}\right)+\frac{1}{n}+U_{i,ot}^\e(t,\mu,\nu) +\frac{C\delta}{T-t}\left(T-t+\mathbb{E}\left[\int_t^T|\alpha_s^n|^2ds\right]\right)\\
        &\le U_{i,ot}^\e(t,\mu,\nu)+\delta \left(K_0-\frac{U_{i,ot}^\e(t,\mu,\nu)}{T-t}\right)+\frac{1}{n}\\
        &\qquad+\frac{C\delta}{T-t}\left( \mathbb{E}\left[\int_t^TL\left( \frac{X_s^n}{\e},\alpha_s^n\right)ds\right]+K_0\right)\\
        &\le U_{i,ot}^\e(t,\mu,\nu)+\frac{1}{n}+\delta C\left(1+\frac{1}{n(T-t)}+\frac{U_{i,ot}^\e(t,\mu,\nu)}{T-t}\right)
    \end{align*}
    where in the last inequality we used Lemma \ref{naive estimate}. For $n$ large and $\delta$ small enough, we see that the last two terms on the right hand side become very small. This implies that for $\e_0>0$, we can choose $n,\delta$ such that
    $$\mathbb{E}\bigg[ \int_t^TL\big( \frac{X_s}{\e},\;\alpha_s\big)ds \bigg]\le U^\e_{i,ot}(t,\mu,\nu)+\e_0.$$
    The reverse inequality follows.
    \vspace{2mm}

    \nd
    To prove the second statement, we observe that the above chain of inequalities implies
    $$\mathbb{E}\bigg[ \int_t^TL\big( \frac{X_s}{\e},\;\alpha_s\big)ds \bigg]\le U_{i,ot}^\e(t,\mu,\nu)+\frac{1}{n}+\delta C\left(1+\frac{1}{n(T-t)}+\frac{U_{i,ot}^\e(t,\mu,\nu)}{T-t}\right).$$
    Choosing $n$ to be large enough so that $\frac{1}{n}<\delta$ and since $\frac{1}{n}\le 1$ we derive
    $$\mathbb{E}\bigg[ \int_t^TL\big( \frac{X_s}{\e},\;\alpha_s\big)ds \bigg]\le U_{i,ot}^\e(t,\mu,\nu)+\delta+\delta C\left(1+\frac{1}{(T-t)}+\frac{|U_{i,ot}^\e(t,\mu,\nu)|}{T-t}\right)\le U_{i,ot}^\e(t,\mu,\nu)+\delta C_0,$$
    where $C_0:=1+C\left(1+\frac{1}{(T-t)}+\frac{|\sup_{m\in\mathcal{P}_2:{\bf d}_2^2(\delta_0,m)\le R}U_{i,ot}^\e(t,m,\nu)|}{T-t}\right)$. Here, we used the fact that $\mu$ satisfies ${\bf d}_2^2(\delta_0,\mu)\le \int|x|^2\mu(dx)\le R$.
    \vspace{1mm}

    \nd
    We let $\eta>0$ so that $\eta C_0^{-1}\le T-t$. Choosing $\delta$ such that $\delta C_0=\eta$, the above inequality gives
    $$\mathbb{E}\bigg[ \int_t^TL\big( \frac{X_s}{\e},\;\alpha_s\big)ds \bigg]\le U_{i,ot}^\e(t,\mu,\nu)+\eta.$$
    This means that $(\alpha_s)_{s\in [t,T]}$ is an $\eta$-optimal control for $U^\e_{i,ot}(t,\mu,\nu)$ and, since it is also in $\mathcal{A}_{t,\mu,\nu}^{i,\delta}$, the equality $\delta=C_0^{-1}\eta$ implies $(\alpha_s)_{s\in [t,T]}\in \mathcal{A}_{t,\mu,\nu}^{i, C_0^{-1}\eta}$. So the constant $C=C_{t,\nu,R}$ in the second statement is $C_0^{-1}$ from above.
\end{proof}


\nd
We now study the dependence of $U_{ot}^\varepsilon$ with respect to the initial time variable $t$.

\begin{theorem}\label{thm:lipschitz-time}
Fix $\mu,\nu\in\mathcal P_2(\mathbb R^d)$, $\delta\in (0,T)$ and $i\in \{1,2\}$.
Then there exists a constant $C=C(\mu,\nu,\delta)>0$ such that
\[
|U_{i,ot}^\varepsilon(t,\mu,\nu)-U_{i,ot}^\varepsilon(t',\mu,\nu)|
\;\le\;
C\,|t-t'|
\]
for all $t,t'\in[0,T-\delta]$.
\end{theorem}

\begin{proof}
We split the proof into two steps. Assume $t<t'$

\medskip
\noindent\emph{Step 1: Upper bound.} Since $\mu,\nu$ are fixed, we let $C=C_{t,\nu,R}$ (for $R>\int |x|^2\mu(dx)$) be the constant given by Proposition \ref{altformulaot}.
We also consider $\alpha\in \mathcal{A}_{t,\mu,\nu}^{i,t'-t}$ a $\frac{t'-t}{C}$-optimal control and $(X_s)_{s\in [t,T]}$ the associated state process. By definition the definition of $\mathcal{A}_{t,\mu,\nu}^{i,t'-t}$, it is true that $\alpha|_{s\in [t',T]}\in \mathcal{A}_{t',\mu,\nu}^i$, therefore we compute
\begin{align*}
    U_{i,ot}^\e(t',\mu,\nu)&\le\mathbb{E}\left[\int_{t'}^TL\left(\frac{X_s}{\e},\alpha_s\right) ds \right]=\mathbb{E}\left[\int_{t}^TL\left(\frac{X_s}{\e},\alpha_s\right) ds \right]-\mathbb{E}\left[\int_{t}^{t'}L\left(\frac{X_s}{\e},0\right) ds \right]\\
    &\le U_{i,ot}^\e(t,\mu,\nu)+\frac{t'-t}{C}+K_0(t'-t)=U_{i,ot}^\e(t,\mu,\nu)+ (t'-t)(K_0+1/C).
\end{align*}

\medskip
\noindent\emph{Step 2: Lower bound.}
Let $\alpha'$ be an admissible control for $U_{i,ot}^\varepsilon(t',\mu,\nu)$.
We extend it to $[t,T]$ by setting $\alpha'_s:=0$ for $s\in[t,t')$ and
$X_s:=X_{t'}$ on $[t,t']$.
Then $X_t\sim\mu$ and $X_T\sim\nu$.
\vspace{1mm}

\nd
Using \eqref{eqn:Lmquad}, we obtain
\begin{align*}
U_{i,ot}^\varepsilon(t,\mu,\nu)\le
\mathbb E\int_t^{t'} L\Big(\frac{X_s}{\varepsilon},0\Big)\,ds
+
\mathbb E\int_{t'}^T L\Big(\frac{X_s}{\varepsilon},\alpha'_s\Big)\,ds \le
K_0(t'-t)+U_{i,ot}^\varepsilon(t',\mu,\nu).
\end{align*}

\medskip

\nd
Combining the two estimates yields the desired Lipschitz continuity.
\end{proof}

\begin{proposition}\label{measure lip}
Fix $t<T$, $\nu\in\mathcal P_2(\mathbb R^d)$ and $i\in \{1,2\}$. Then the map $\mu\mapsto U_{i,ot}^\varepsilon(t,\mu,\nu)$ is continuous on $\mathcal P_2(\mathbb R^d)$ with respect to the ${\bf d}_2$-distance.
Moreover, for every $R>0$ there exists a constant $\overline{C}=\overline{C}(R,T-t,C_{t,\nu,R})>0,$ where $C_{t,\nu,R}$ is the constant constructed in Proposition \ref{altformulaot}, such that
\[
|U_{i,ot}^\varepsilon(t,\mu_1,\nu)-U_{i,ot}^\varepsilon(t,\mu_2,\nu)|
\le \overline{C}\,{\bf d}_2(\mu_1,\mu_2),
\]
for all $\mu_1,\mu_2\in\mathcal P_2(\mathbb R^d)$
with $\int|x|^2\,d\mu_i\le R,\;\; i=1,2$.
\end{proposition}

\begin{proof}
    We fix $R>0$ and $\mu_1,\mu_2\in \mathcal{P}_2$ such that $\int_{\R^d}|x|^2\;\mu_i(dx)\le R,\;i=1,2$.
    \vspace{1mm}

    \nd
    Let $\delta\in (0,T-t)$ and $\alpha^\delta\in \mathcal{A}_{t,\mu_1,\nu}^{i,\delta}$ a $C^{-1}\delta$-optimal control for $U^\e_{i,ot}(t,\mu_1,\nu)$ provided by Proposition  \ref{altformulaot}, where the constant $C=C_{t,\nu,R}$ was also constructed in Proposition \ref{altformulaot}. We denote by $(X_s^\delta)_{s\in [t,T]}$ the associated state process. By the $\delta$-optimality and the quadratic growth of $L$, we have
    \be\label{L2bound}
\frac{1}{4C_{2,H}}\mathbb{E}\left[\int_{t+\delta}^T|\alpha_s^\delta|^2ds\right]-K_0(T-t)\le U^\e_{i,ot}(t,\mu_1,\nu)+C^{-1}\delta\le \frac{1}{2C_{1,H}(T-t)}{\bf d}_2^2(\mu_1,\nu)+K_0(T-t)+C^{-1}\delta.
    \ee
    We consider $Z\sim N(0,I_d)$ to be a standard normal $d$-dimensional random variable which is independent of $X_{t+\delta}^\delta$. Then, for $r>0$, $X_{t+\delta}^\delta+rZ$ has absolutely continuous distribution with respect to the Lebesgue measure; we denote $\mathcal{L}(X_{t+\delta}^\delta+rZ)=\mu_1^r$. Since $\mu_1^r$ is absolutely continuous, we can also pick $(\alpha_s^{ot})_{s\in [t,t+\delta]}$ an optimal control for the optimal transport problem from $\mu_2$ to $\mu_1^r$ with quadratic cost $L(x,v)=|v|^2$ such that the associated state process $(X_s^{ot})_{s\in [t,t+\delta]}$ satisfies $X_{t+\delta}^{ot}=X_{t+\delta}^\delta+rZ\sim \mu_1^r$. In particular, we may take
    $$X_s^{ot}=\left(1-\frac{s-t}{\delta}\right)T(X_{t+\delta}^\delta+rZ)+\frac{s-t}{\delta}(X_{t+\delta}^\delta+rZ),$$
    where $T:\R^d\to \R^d$ is the Brenier map from $\mu_1^r$ to $\mu_2$. We also consider the control $\alpha$ and the associated state process as follows
    $$\alpha_s=\begin{cases}
        \alpha_s^\delta-\frac{rZ}{T-t-\delta},&\text{if }s\in [t+\delta,T],\\
        \alpha_s^{ot},&\text{if }s\in [t,t+\delta)
    \end{cases},\quad\text{and}\quad X_s=X_t^{ot}+\int_t^s\alpha_\tau d\tau,\;\;s\in[t,T].$$
    It is  clear that $\alpha \in\mathcal{A}_{t,\mu_2,\nu}^i$, because $X_t=X_t^{ot}=T(X_{t+\delta}^\delta+rZ)\sim \mu_2$ and $X_T=X_T^\delta\sim\nu$. We have
    \begin{align*}
        U_{i,ot}^\e(t,&\mu_2,\nu)\le \mathbb{E}\left[ \int_t^T L\left(\frac{X_s}{\e},\alpha_s\right)ds    \right]=\mathbb{E}\left[ \int_t^{t+\delta} L\left(\frac{X_s^{ot}}{\e},\alpha_s^{ot}\right)ds +\int_{t+\delta}^TL\left(\frac{X_s}{\e},\alpha_s\right)ds   \right]\\
        &\le \delta K_0 +\frac{{\bf d}_2^2(\mu_1^r,\mu_2)}{2 C_{1,H} \delta}+ \mathbb{E}\left[ \int_{t+\delta}^T L\left(\frac{X_s^\delta}{\e},\alpha_s^\delta\right)ds \right] +\mathbb{E}\left[\int_{t+\delta}^T\left|L\left(\frac{X_s}{\e},\alpha_s\right)-L\left(\frac{X_s^\delta}{\e},\alpha_s^\delta\right)\right|  ds\right]\\
        &\le \delta K_0 +\frac{{\bf d}_2^2(\mu_1^r,\mu_2)}{2 C_{1,H} \delta}+\mathbb{E}\left[ \int_{t}^T L\left(\frac{X_s^\delta}{\e},\alpha_s^\delta\right)ds \right] -\mathbb{E}\left[ \int^{t+\delta}_t L\left(\frac{X_s^\delta}{\e},\alpha_s^\delta\right)ds \right]\\
        &\hspace{2cm} +C_L\mathbb{E}\left[\int_{t+\delta}^T(1+|\alpha_s|+|\alpha_s^\delta|)\left( \left|\frac{X_s-X_s^\delta}{\e}\right|+|\alpha_s-\alpha_s^\delta|   \right)ds\right]\\
        &\le U_{i,ot}^\e(t,\mu_1,\nu)+C^{-1}\delta +\frac{{\bf d}_2^2(\mu_1^r,\mu_2)}{2 C_{1,H} \delta}+2\delta K_0\\
         &\hspace{2cm}+C_L\mathbb{E}\left[\int_{t+\delta}^T(1+|\alpha_s|+|\alpha_s^\delta|)\left( \left|\frac{X_s-X_s^\delta}{\e}\right|+|\alpha_s-\alpha_s^\delta|   \right)ds\right],
        \end{align*}
        where we repeatedly used the regularity and growth properties of $L.$ However, for $s\in [t+\delta,T]$,
        $$X_s-X_s^\delta=rZ-rZ\frac{s-t-\delta}{T-t-\delta}\quad\text{and}\quad\alpha_s-\alpha_s^\delta=\frac{rZ}{T-t-\delta},$$
        hence the last term in the above inequality can be bounded by
        \begin{align*}
            C_L\mathbb{E}\left[\int_{t+\delta}^T\left(1+|\alpha_s^\delta|+\frac{r|Z|}{T-t-\delta}\right)\left(\frac{r|Z|}{\e}\frac{T-s}{T-t-\delta}+\frac{r|Z|}{T-t-\delta}  \right)ds\right]\le rC'\left(\frac{1}{\e}+1\right),
        \end{align*}
        for some constant $C'$ depending only on $R,\nu$ and $T-t$, where in the last step we also used \eqref{L2bound} and the Cauchy-Schwarz inequality. Plugging this back in the above yields
        $$U^\e_{i,ot}(t,\mu_2,\nu)\le U_{i,ot}^\e(t,\mu_1,\nu)+C^{-1}\delta +\frac{{\bf d}_2^2(\mu_1^{r},\mu_2)}{2 C_{1,H} \delta}+2\delta K_0+rC'\left(\frac{1}{\e}+1\right).$$
        For $r\to 0^+$, this implies
        $$U^\e_{i,ot}(t,\mu_2,\nu)\le U_{i,ot}^\e(t,\mu_1,\nu)+C^{-1}\delta +\frac{{\bf d}_2^2(\mu_1,\mu_2)}{2 C_{1,H} \delta}+2\delta K_0.$$
     Now we choose $\delta=\frac{T-t}{4R^2}{\bf d}_2(\mu_1,\mu_2)$. We observe $\delta< T-t$, because ${\bf d}_2(\mu_1,\mu_2)<4R^2$, and that this choice in the above inequality gives
    $$U_{i,ot}^\e(t,\mu_2,\nu)\le U_{i,ot}^\e(t,\mu_1,\nu)+\overline{C}{\bf d}_2(\mu_1,\mu_2),$$
    for some constant $\overline{C}>0$ depending only on $R,\; C\text{ and }T-t$. We may similarly prove the inequality $U_{i,ot}^\e(t,\mu_1,\nu)\le U_{i,ot}^\e(t,\mu_2,\nu)+\overline{C}{\bf d}_2(\mu_1,\mu_2),$ and this finishes the proof.
\end{proof}

\begin{proposition}\label{measure lip2}
    Fix $t<T$, $\mu\in\mathcal P_2(\mathbb R^d)$ and $i\in \{1,2\}$.
    Then the map $\mathcal{P}_2(\R^d)\ni\nu\mapsto U_{i,ot}^\varepsilon(t,\mu,\nu)$ is continuous on
    $\mathcal P_2(\mathbb R^d)$ with respect to the ${\bf d}_2$-distance. Moreover, for every $R>0$ there exists $C=C(R,T-t)>0$ such that
    \[
    |U_{i,ot}^\varepsilon(t,\mu,\nu_1)-U_{i,ot}^\varepsilon(t,\mu,\nu_2)|
    \le C\,{\bf d}_2(\nu_1,\nu_2),
    \]
    for all $\nu_1,\nu_2\in\mathcal P_2(\mathbb R^d)$
    with $\int|x|^2\,d\nu_i\le R,\; i=1,2$.
\end{proposition}

\begin{proof}
    The proof is identical to Proposition \ref{measure lip} after reversing the time.
\end{proof}

\subsection{Convergence of $U_{ot}^\e$}

\begin{proof}[Proof of Theorem \ref{main3}]
    \textit{Step 1.} ($U^\e_{ot}=U^\e_{1,ot}$ and $U_{ot}=U_{1,ot}$)\\[0.5ex]
    We first notice that Proposition \ref{prop:U1eqUeqU2} implies
    $$U_1^\e(t,\mu)=\inf_{\nu\in \mathcal{P}_2}\left\{ U^\e_{1,ot}(t,\mu,\nu)+F(\nu)\right\}=\inf_{\nu\in \mathcal{P}_2}\left\{ U^\e_{2,ot}(t,\mu,\nu)+F(\nu)\right\}=U_2^\e(t,\mu),$$
    for any $F\colon\mathcal{P}_2\to \R$ bounded and Lipschitz and any $(t,\mu)\in [0,T)\times \mathcal{P}_2$.
    Since $U^\e_{1,ot},U^\e_{2,ot}$ are lower bounded (Lemma \ref{naive estimate}) and the maps $\nu\mapsto U^\e_{1,ot}(t,\mu,\nu),\; \nu\mapsto U^\e_{2,ot}(t,\mu,\nu)$ are continuous (Proposition \ref{measure lip2}), it follows from Lemma \ref{infeq} that $U^\e_{1,ot}=U^\e_{2,ot}$. This equality yields $U^\e_{1,ot}=U^\e_{ot}=U^\e_{2,ot}$, because of \eqref{eq1001}.
    \vspace{1mm}

    \nd
    By the regularity and growth properties of the effective cost $\overline{L}$ in Proposition \ref{prop:Lmprop}, we may argue similarly for the value function $U_{ot}$ from \eqref{DOTlimit} to derive
    $$U_{ot}(t,\mu,\nu)=U_{1,ot}(t,\mu,\nu):=\inf_{\alpha\in \mathcal{A}^1_{t,\mu,\nu}}\left\{\mathbb{E}\left[\int_t^T\overline{L}(X_s,\alpha_s,\mathcal{L}(X_s))ds   \right]\right\},$$
    for any $(t,\mu,\nu)\in [0,T)\times\mathcal{P}_2\times\mathcal{P}_2$.
    \vspace{2.5mm}

    \nd
    \textit{Step 2.} (Convergence and rate of convergence under (A1))\\[0.3ex]
    Fix $(t,\mu,\nu)\in [0,T)\times\mathcal{P}_2\times\mathcal{P}_2$ and set $M_0:={\bf d}_2^2(\mu,\nu)$. Using the notation introduced in subsection \ref{formulations}, we write
    $$U^\e_{1,ot}(t,\mu,\nu)=\inf_{(X, Y) \in \mathcal{E}_{\mu,\nu}} \left\{ \inf_{\gamma \in \Gamma(X, Y)} \mathbb{E}\left[\int_{t}^T L\left( \frac{\gamma_s}{\e}, \gamma_s, \dot{\gamma}_s, \mathcal{L}(\gamma_s) \right)\, ds \right] \right\}.$$
    We argue as in the proof of Theorem \ref{thm:multiscale}. Let $(X,Y,\gamma)$ be $\delta$-optimal for $U^\e_{1,ot}(t,\mu,\nu)$. Then
\begin{equation}\label{eqn:nopuepot}
 \mathbb{E}\left[\int_{t}^T L\left( \frac{\gamma_s}{\e}, \gamma_s, \dot{\gamma}_s, \mathcal{L}(\gamma_s) \right)\, ds \right] \leq U^\e_{1,ot}(t, \mu,\nu)+\delta.
\end{equation}
By Lemma \ref{naive estimate} and the growth properties of $L$, we discover that there exists a constant $C=C(d, C_{1,H},C_{2,H},C_G, K_0)>0$ such that
\be\label{gammaineq}
\frac{1}{T-t}\mathbb{E}\left[\int_t^T|\dot{\gamma}_s|^2ds   \right]\le C\left(1+\frac{M_0}{(T-t)^2}\right).
\ee
Consider a time partition $t_k := t + \frac{k}{N}(T-t)$ for $ k=0,\dots,N$, where $N \in \mathbb{N}$ will be chosen later. With $J_\e$ as in Lemma~\ref{lem:fixest}, we may write
\begin{equation}\label{eqn:opfixmuot}
\mathbb{E}\left[\int_{t}^T L\left( \frac{\gamma_s}{\e}, \gamma_s, \dot{\gamma}_s, \mathcal{L}(\gamma_s) \right)\, ds \right]\geq J_\e (\gamma) - \sum_{k=0}^{N-1} E_k,
\end{equation}
where
\[
E_k:=\mathbb{E}\left[\int_{t_k}^{t_{k+1}} \left|L\left( \frac{\gamma_s}{\e}, \gamma_s, \dot{\gamma}_s, \mathcal{L}(\gamma_s) \right)-
L\left(\frac{\gamma_s}{\e}, \gamma_{t_k}, \dot{\gamma}_s, \mathcal{L}(\gamma_{t_k})\right)\right|\,ds\right]. \]
By Lemma \ref{lem:fixest},
\begin{equation}\label{eqn:tpartestot}
\begin{aligned}
J_\e(\gamma) \geq \inf_{\eta\in\Gamma(X,Y;\gamma)}
J_\e (\eta) =\sum_{k=0}^{N-1}
\mathbb E\left[h^\e\left(\gamma_{t_k}, \mathcal{L}(\gamma_{t_k}); t_k,t_{k+1},
\gamma_{t_k},\gamma_{t_{k+1}}\right)\right].
\end{aligned}
\end{equation}

\nd
Next, Proposition~\ref{prop:rocmu} implies that
\begin{equation}\label{eqn:memukmbarot}
\begin{split}
\mathbb{E}&\left[h^\e\left(\gamma_{t_k}, \mathcal{L}(\gamma_{t_k}); t_k,t_{k+1},
\gamma_{t_k},\gamma_{t_{k+1}}\right)\right]\\
&\hspace{2cm}\geq \mathbb{E}\left[\bar{h}\left(\gamma_{t_k}, \mathcal{L}(\gamma_{t_k}); t_k,t_{k+1},
\gamma_{t_k},\gamma_{t_{k+1}}\right)\right] -C\left(1+\frac{ \mathbb{E}\left|\gamma_{k+1}-\gamma_k\right|^2}{|t_{k+1}-t_k|^2}\right) \e,
\end{split}
\end{equation}
for some constant $C=C(d, C_{1,H},C_{2,H},K_0)>0$. By H\"older's inequality,
\[
\sum_{k=0}^{N-1}\frac{ \mathbb{E}\left|\gamma_{k+1}-\gamma_k\right|^2}{|t_{k+1}-t_k|^2}\leq \frac{N}{T-t} \mathbb{E} \left[\int_{t}^T\left|\dot{\gamma}_s\right|^2\right]\,
ds \leq NC\left(1+\frac{M_0}{(T-t)^2}\right),
\]
where we use the fact that $t_{k+1}-t_k=\frac{T-t}{N}$ and \eqref{gammaineq}. Combining this with \eqref{eqn:memukmbarot} yields
\begin{equation}\label{eqn:rocsumerrot}
\begin{split}
\sum_{k=0}^{N-1}
\mathbb E&\left[h^\e\left(\gamma_{t_k}, \mathcal{L}(\gamma_{t_k}); t_k,t_{k+1},
\gamma_{t_k},\gamma_{t_{k+1}}\right)\right]\\
&\hspace{1cm}\ge\sum_{k=0}^{N-1}\mathbb{E}\left[\bar{h}\left(\gamma_{t_k}, \mathcal{L}(\gamma_{t_k}); t_k,t_{k+1},
\gamma_{t_k},\gamma_{t_{k+1}}\right)\right] -C N\e\left(1+\frac{M_0}{(T-t)^2}\right).
\end{split}
\end{equation}
Let $\bar\gamma:[t_0,T] \times \R^d \to \R^d$ be the piecewise linear interpolation of $\gamma$ on the partition $\{t_k\}_{k=0}^N$, namely
\[
\bar{\gamma}(t,\omega):=\gamma(t_k,\omega)+\frac{s-t_k}{t_{k+1}-t_k}\left(\gamma(t_{k+1}, \omega)-\gamma(t_k,\omega)\right), \qquad \text{ for } t \in \left[t_k, t_{k+1}\right].
\]
By Lemma \ref{lem:lbarpart} (for $\left(1+\frac{M_0}{(T-t)^2}\right)$ in place of $M_1$),
\begin{equation}\label{eqn:admgabarot}
\begin{aligned}
&\sum_{k=0}^{N-1}
\mathbb{E}\left[\bar{h}\left(\gamma_{t_k}, \mathcal{L}(\gamma_{t_k}); t_k,t_{k+1},
\gamma_{t_k},\gamma_{t_{k+1}}\right)\right]=\inf_{\eta\in \Gamma(X,Y;\gamma)} \bar J(\eta)=\bar{J}(\bar{\gamma})\\
 & \hspace{2cm}\geq\sum_{k=0}^{N-1}
\mathbb{E} \left[\int_{t_k}^{t_{k+1}}
\bar{L}\left(\bar{\gamma}_s,
\dot{\bar{\gamma}}_s,
\mathcal{L}\left(\bar{\gamma}_s\right)
\right)\,ds\right]-\sum_{k=0}^{N-1} \bar{E}_k,
\end{aligned}
\end{equation}
where
\[
\bar{E}_k:= \mathbb{E} \left[\int_{t_k}^{t_{k+1}}\left|
\bar{L}\left(\gamma_{t_k},
\dot{\bar{\gamma}}_s,
 \mathcal{L}(\gamma_{t_k})
\right)-
\bar{L}\left(\bar{\gamma}_s,
\dot{\bar{\gamma}}_s,
\mathcal{L}\left(\bar{\gamma}_s\right)
\right)\right|\,ds\right].
\]
Combining \eqref{eqn:opfixmuot}, \eqref{eqn:tpartestot}, \eqref{eqn:rocsumerrot}, and \eqref{eqn:admgabarot}, we obtain
\begin{equation}\label{eqn:intstepuepot}
    \mathbb{E}\left[\int_{t}^T L\left( \frac{\gamma_s}{\e}, \gamma_s, \dot{\gamma}_s, \mathcal{L}(\gamma_s) \right)\, ds \right] \geq \mathbb{E} \left[\int_t^T
\bar{L}\left(\bar{\gamma}_s,
\dot{\bar{\gamma}}_s,
\mathcal{L}\left(\bar{\gamma}_s\right)
\right)\,ds\right]-\sum_{k=0}^{N-1} E_k -\sum_{k=0}^{N-1} \bar{E}_k -C_0N \e,
\end{equation}
where $C_0:=C\left(1+\frac{M_0}{(T-t)^2}\right)$. By Lemmas \ref{lem:fixest}, \ref{lem:lbarpart} (for $M=M_1=1+\frac{M_0}{(T-t)^2}=C_0/C$), and \eqref{gammaineq},
\begin{equation}\label{eqn:eksumot}
\sum_{k=0}^{N-1}E_k \leq  \frac{C_2(1+C_0/C)(T-t)^2}{N}, \qquad
\sum_{k=0}^{N-1}\bar{E}_k\leq \frac{C_2(1+C_0/C)(T-t)^2}{N},
\end{equation}
for some constant $C_2=C(d,C_{1,H},C_{2,H},C_G, K_0)>0$. Combining \eqref{eqn:nopuepot}, \eqref{eqn:intstepuepot}, and \eqref{eqn:eksumot}, we deduce that
\[
\begin{aligned}
U^\e_{1,ot}(t,\mu,\nu))+\delta & \geq \mathbb{E} \left[\int_t^T
\bar{L}\left(\bar{\gamma}_s,
\dot{\bar{\gamma}}_s,
\mathcal{L}\left(\bar{\gamma}_s\right)
\right)\,ds\right] -\frac{2C_2(1+C_0/C)(T-t)^2}{N}-C_0N\e\\
& \geq U_{1,ot}(t,\mu,\nu) -\frac{2C_2(1+C_0/C)(T-t)^2}{N}-C_0N\e.
\end{aligned}
\]
where $C_2,C_0,C$ are the constants obtained above. Letting $\delta\to 0$ and optimizing with respect to $N$ we obtain
$$U^\e_{1,ot}(t,\mu,\nu)-U_{1,ot}(t,\mu,\nu)\ge -C_1
\left(2(T-t)+\frac{M_0}{T-t}  \right)\sqrt{\e},$$
for some constant $C_1>0$ depending only on $C_2,C$ considered above and hence on $d,C_{1,H},C_{2,H},C_G$ and $K_0$. With a similar argument, as in the proof of Theorem \ref{thm:multiscale}, we establish
$$U^\e_{1,ot}(t,\mu,\nu)-U_{1,ot}(t,\mu,\nu)\le C_1'
\left(2(T-t)+\frac{M_0}{T-t}  \right)\sqrt{\e},$$
for some constant $C_1'>0$ depending only on $d,C_{1,H},C_{2,H},C_G$ and $K_0$. We omit the details of the last calculation. The first result follows.
\vspace{2.5mm}

\nd
    \textit{Step 3.} (Convergence and rate of convergence under (A1'))\\
    Fix $(t,\mu,\nu)\in [0,T)\times\mathcal{P}_2\times\mathcal{P}_2$. Combining the equality $U^\e_1(t,\mu)=\inf_{\nu\in \mathcal{P}_2}\left\{ U^\e_{1,ot}(t,\mu,\nu)+F(\nu) \right\}$ with the equality of Lemma \ref{lem:Uep1costh}, we obtain
    $$\inf_{\nu\in \mathcal{P}_2}\left\{ U^\e_{1,ot}(t,\mu,\nu)+F(\nu) \right\}=\inf_{\nu\in \mathcal{P}_2}\left\{ \inf_{(X, Y) \in \mathcal{E}_{\mu,\nu}}
\mathbb{E}[h^\e(t,T,X,Y)]+F(\nu) \right\},$$
    for any bounded and ${\bf d}_2$-Lipschitz terminal cost $F:\mathcal{P}_2\to \R$. By the continuity and the quadratic growth estimates proved for $h^\e$ in Proposition \ref{prop:mmeas}, we derive from \cite[Theorem 5.20]{villani2009optimal} that the optimal transport value function $\nu\mapsto \inf_{(X, Y) \in \mathcal{E}_{\mu,\nu}}
\mathbb{E}[h^\e(t,T,X,Y)]$ is continuous. Thus, as in Step 1 by Lemma \ref{infeq}, we have
$$U^\e_{ot}(t,\mu,\nu)=U^\e_{1,ot}(t,\mu,\nu)=\inf_{(X, Y) \in \mathcal{E}_{\mu,\nu}}\mathbb{E}[h^\e(t,T,X,Y)].$$
We can similarly show that
$U_{ot}(t,\mu,\nu)=U_{1,ot}(t,\mu,\nu)=\inf_{(X, Y) \in \mathcal{E}_{\mu,\nu}}\mathbb{E}[\overline{h}(t,T,X,Y)],$
where $\overline{h}$ is from Proposition \ref{prop:subsuperm}. We finish the proof by applying the estimate for the difference of $\overline{h},h^\e$ from Proposition \ref{prop:subsuperm}(3):
\begin{align*}
    \inf_{(X, Y) \in \mathcal{E}_{\mu,\nu}}\mathbb{E}[h^\e(t,T,X,Y)]&\ge \inf_{(X, Y) \in \mathcal{E}_{\mu,\nu}}\left\{\mathbb{E}[\overline{h}(t,T,X,Y)]-\e C\left(1+\frac{\mathbb{E}[|X-Y|^2]}{(T-t)^2}\right)\right\}\\
    &\ge  \inf_{(X, Y) \in \mathcal{E}_{\mu,\nu}}\left\{\mathbb{E}[\overline{h}(t,T,X,Y)]-\e C\left(1+\frac{{\bf d}_2^2(\mu,\nu)}{(T-t)^2}\right)\right\},
\end{align*}
for some constant \(C=C(d, C_{1, H}, C_{2, H}, K_0)>0\). Thus, $U^\e_{1,ot}(t,\mu,\nu)\ge U_{1,ot}(t,\mu,\nu)-\e C\left(1+\frac{{\bf d}_2^2(\mu,\nu)}{(T-t)^2}\right)$. We show $U_{1,ot}(t,\mu,\nu)\ge U^\e_{1,ot}(t,\mu,\nu)-\e C\left(1+\frac{{\bf d}_2^2(\mu,\nu)}{(T-t)^2}\right)$ similarly.
\end{proof}

\appendix

\section{Proofs of preliminary results}\label{proofofprelim}
\begin{proof}[Proof of Proposition \ref{prop:Lmprop}]
We only prove (1) and (4). The proof of the remaining statements can be found in \cite{tran_hamilton-jacobi_2021}.

\noindent
(1) Fix \(y \in \mathbb{T}^d, x\in \R^d\), \( v \in \mathbb{R}^d\), and \(m\in \mathcal{P}_2\). By \eqref{quadgrowthHm},
        \[
        0 \cdot v-H(y, x, 0, m)  \geq -K_0.
        \]
        Moreover, again by \eqref{quadgrowthHm}, for any $p \in \mathbb{R}^d$ with $\displaystyle|p|\geq \frac{|v|+2K_0}{C_{1,H}}+1$,
\[
-p \cdot v-H(y,x, p,m)
\leq  |p|\,|v|-C_{1,H}|p|^2+K_0\leq |p|\left(|v|-C_{1,H}|p|\right)+K_0 \leq -K_0.
\]
Therefore,
\[
L(y,x,v,m)=\sup_{p\in \mathbb{R}^d} \big( -p \cdot v-H(y,x,p,m)\big) = \max_{|p|\leq \frac{|v|+2K_0}{C_{1,H}}+1} \big(-p \cdot v-H(y,x, p, m)\big) <+\infty.
\]

\noindent
(4) Let $y,y' \in \mathbb{T}^d$, $x,x', v, v' \in \mathbb{R}^d$. From part (1), there exists \(p_{x,y,m} \in \R^d\) with \(|p_{x,y,m}| \leq \frac{|v|+2K_0}{C_{1,H}}+1 \) such that
\[
L(y,x,v,m) = -p_{x,y,m} \cdot v - H(y,x, p_{x,y,m},m).
\]
We estimate:
\[
\begin{aligned}
&L(y,x,v,m)-L(y',x',v',m')\\
&\leq -p_{x,y,m} \cdot v - H(y,x,p_{x,y,m},m) - \big(-p_{x,y,m} \cdot v' - H(y',x',p_{x,y,m},m')\big) \\
&= H(y',x', p_{x,y,m}, m') - H(y,x,p_{x,y,m},m) + p_{x,y,m} \cdot (v'-v) \\
&\leq H(y,x, p_{x,y,m},m') - H(y,x,p_{x,y,m},m) + \left(\frac{|v|+2K_0}{C_{1,H}}+1\right) |v-v'|.\\
& \leq C_H\left(1 + 2\left(\frac{|v|+2K_0}{C_{1,H}}+1\right) \right)\left(|x-x'|+|y-y'|+{\bf d}_2(m,m')\right)+ \left(\frac{|v|+2K_0}{C_{1,H}}+1\right) |v-v'|\\
& \leq C\left(1+|v|\right)\left(|x-x'|+|y-y'|+|v-v'|+{\bf d}_2(m,m')\right),
\end{aligned}
\]
where the second-to-last inequality follows from \eqref{ineqHm} and $C=C\left(C_{1, H}, C_H, K_0\right)>0$.

By symmetry, exchanging \((x,v)\) and \((x',v')\) gives
\[
L(y',x',v',m')-L(y,x,v,m) \leq C\left(1+|v'|\right)\left(|x-x'|+|y-y'|+|v-v'|+{\bf d}_2(m,m')\right).
\]
Hence,
\[
\left|L(y,x,v,m)-L(y',x',v',m')\right| \leq C_L\left(1+|v|+|v'|\right)\left(|x-x'|+|v-v'|\right),
\]
for some constant $C_L=C_L\left(C_{1, H}, C_H, K_0\right)>0$.
\end{proof}

\begin{proof}[Proof of Proposition \ref{effHm}]
    We only prove \eqref{eqn:Hbarmgrowth} and \eqref{eqn:Hbarmreg}; the remaining assertions follow from the proof of \cite[Theorems 4.3, 4.14]{tran_hamilton-jacobi_2021}.

    From \cite[(4.17)]{tran_hamilton-jacobi_2021}, we know that
    \[
    \min_{y \in \T^d} H(y,x, p,m) \leq \bar{H}(x,p,m) \leq \max_{y\in \T^d} H(y,x, p, m), \qquad \text{for all }x, p \in \R^d, m \in \mathcal{P}_2.
    \]
Together with \eqref{quadgrowthHm}, this yields
\begin{equation}\label{eqn:Hbarquad}
    C_{1,H}|p|^2-K_0 \leq \bar{H}(x,p,m) \leq C_{2,H}|p|^2+K_0.
\end{equation}
    Let $x', p'\in \R^d$, $m' \in \mathcal{P}_2$, and $u^\delta: \T^d \to \R$ be the unique solution of
    \[
    \delta u^\delta(y) + H \left(y, x',Du^\delta(y)+p', m'\right)=0, \qquad y \in \T^d.
    \]
    As $-\delta u^\delta(\cdot) \to \bar{H}(x', p', m')$ uniformly on $\T^d$, for $\delta>0$ small enough, by \eqref{eqn:Hbarquad} and \eqref{quadgrowthHm}, we have
    \[
    C_{1, H}\left|Du^\delta+p'\right|^2-K_0\leq H(y,x', Du^\delta+p',m')=-\delta u^\delta (y) \leq \bar{H}(p',m')+1 \leq C_{2,H}|p|^2+K_0+1,
    \]
    which implies
    \[
    \left|Du^\delta+p'\right|^2\leq \frac{C_{2,H}}{C_{1,H}} |p'|^2+\frac{2K_0+1}{C_{1,H}}.
    \]
    Hence,
    \begin{equation}\label{eqn:udlipm}
        \left|Du^\delta\right| \leq C\left(1+\left|p'\right|\right)
    \end{equation}
    for some constant $C=C\left(C_{1,H}, C_{2,H}, K_0\right)>0$.
    We estimate
    \begin{equation}\label{eqn:Hudppprime}
    \begin{aligned}
    &\left|H\left(y,x', Du^\delta(y) +p',m'\right)-H\left(y,x,Du^\delta(y) +p,m\right)\right| \\\leq & C_H\left(1+2\left|Du^\delta\right|+|p|+|p'|\right)\left(\left|x-x'\right|+\left|p-p'\right|+{\bf d}_2(m,m')\right)\\
    \leq & C_{\bar{H}} \left(1+|p|+|p'|\right)\left(\left|x-x'\right|+\left|p-p'\right|+{\bf d}_2(m,m')\right),
    \end{aligned}
    \end{equation}
where $C_{\bar{H}}=C_{\bar{H}}\left(C_{1,H}, C_{2,H}, C_H, K_0\right)>0$ is a constant, and the second inequality follows from \eqref{eqn:udlipm}. Consider $\displaystyle u^\delta + \frac{C_{\bar{H}}\left(1+|p|+|p'|\right)\left(\left|x-x'\right|+\left|p-p'\right|+{\bf d}_2(m,m')\right)}{\delta}$. This is a supersolution to \eqref{discountcellm}, since by \eqref{eqn:Hudppprime},
\[
\begin{aligned}
&\delta u^\delta+ C_{\bar{H}}\left(1+|p|+|p'|\right)\left(\left|x-x'\right|+\left|p-p'\right|+{\bf d}_2(m,m')\right)+ H\left(y,x, Du^\delta(y) +p,m\right)\\
\geq & \delta u^\delta+H\left(y,x', Du^\delta(y) +p',m'\right)=0.
\end{aligned}
\]
Similarly, $\displaystyle u^\delta -\frac{C_{\bar{H}}\left(1+|p|+|p'|\right)\left(\left|x-x'\right|+\left|p-p'\right|+{\bf d}_2(m,m')\right)}{\delta}$ is a subsolution to \eqref{discountcellm}. By the comparison principle for \eqref{discountcellm}, we obtain
\[
\begin{aligned}
    &u^\delta -\frac{C_{\bar{H}}\left(1+|p|+|p'|\right)\left(\left|x-x'\right|+\left|p-p'\right|+{\bf d}_2(m,m')\right)}{\delta}  \\
    &\qquad \qquad\qquad\qquad\qquad\qquad \leq v^\delta  \leq u^\delta +\frac{C_{\bar{H}}\left(1+|p|+|p'|\right)\left(\left|x-x'\right|+\left|p-p'\right|+{\bf d}_2(m,m')\right)}{\delta}.
\end{aligned}
\]
Multiplying by $\delta$ and sending $\delta \to 0 $, we deduce
\[
 \left|\bar{H}(x,p,m) - \bar{H}(x', p',m')\right| \leq C_{\bar{H}}\left(1+|p|+|p'|\right)\left(\left|x-x'\right|+\left|p-p'\right|+{\bf d}_2(m,m')\right).
\]
\end{proof}

\begin{proof}[Proof of Proposition \ref{prop:subsuperm}, parts (1) and (3)]
We begin by proving part (1). Without loss of generality, we assume
$t_1=0$, $t_2=t$ for some $t\in(0,T]$, and $x=0$. The general case
follows by the same argument.

Since, by definition, $h(0,2t, 0, 2y) \leq h(0,t,0,y)+h(0,t,y,2y)$, it suffices to show
\begin{equation}\label{eqn:m2ygoal}
h(0,t,y,2y) \leq h(0,t,0,y)+C\left(1+\frac{|y|^2}{t^2}\right)
\end{equation}
for some constant $C=C(d, C_{1, H}, C_{2, H}, K_0)>0$.

Consider the straight line $\gamma:[0,t]\to \R^d$ connecting $y$ and $2y$, defined by $\gamma(s)=\frac{s}{t}y+y$. By the definition of $h(0,t,y,2y)$,
\begin{equation}\label{eqn:m2yupbd}
    h(0,t,y,2y) \leq \int_0^t L\left(\frac{s}{t}y+y, \frac{y}{t}\right) \, ds \leq \left(\frac{1}{4C_{1,H}} \cdot \frac{|y|^2}{t^2}+K_0\right)t,
\end{equation}
where the second inequality follows from \eqref{eqn:Lmquad}.

Let $\eta:[0,t] \to \R^d$ be a minimizing curve of $h(0,t,0, y)$, that is, $\eta(0)=0, \eta(t)=y$, and
\begin{equation}\label{eqn:mylobd}
h(0,t,0,y)=\int_0^t L \left(\eta(s), \dot{\eta}(s)\right)ds \geq -K_0t,
\end{equation}
where the second inequality follows from \eqref{eqn:Lmquad}. Consider the straight line $\tilde{\gamma}:[0, t] \to \R^d$ connecting $0$ and $y$, defined by $\tilde{\gamma}(s)=\frac{s}{t}y$. Then
\begin{equation}\label{eqn:myupbd}
h(0,t,0,y)=\int_0^t L \left(\eta(s), \dot{\eta}(s)\right)ds \leq \int_0^t L\left(\frac{s}{t}y, \frac{y}{t}\right) \, ds \leq \left(\frac{1}{4C_{1,H}} \cdot \frac{|y|^2}{t^2}+K_0\right)t,
\end{equation}
where the last inequality follows from \eqref{eqn:Lmquad}.

If $t \leq 4$, then by \eqref{eqn:m2yupbd} and \eqref{eqn:mylobd} we have
\[
h(0,t,y,2y) \leq \frac{1}{C_{1,H} }\cdot \frac{|y|^2}{t^2}+4K_0, \qquad h(0,t,0,y) \geq -4K_0.
\]
Consequently,
\[
h(0,t,y,2y) \leq h(0,t, 0,y) +\frac{1}{C_{1,H} }\cdot \frac{|y|^2}{t^2}+8K_0,
\]
and therefore \eqref{eqn:m2ygoal} holds.

Suppose $t>4$. We claim that there exists $d \in \left\{0, 1, \cdots, \lfloor t \rfloor - 1\right\}$ such that
\begin{equation}\label{eqn:xivelupbd}
\frac{1}{4C_{2,H}}\int_d^{d+1} \left|\dot{\eta}(s)\right|^2 \,ds \leq \frac{1}{C_{1,H} }\cdot \frac{|y|^2}{t^2} + 4K_0.
\end{equation}
Indeed, if no such $d$ existed, then
\[
\frac{1}{4C_{2,H}}\int_0^{\lfloor t \rfloor} \left|\dot{\eta}(s)\right|^2 \,ds \geq \lfloor t \rfloor\left(\frac{1}{C_{1,H} }\cdot \frac{|y|^2}{t^2} + 4K_0\right).
\]
Combining this with \eqref{eqn:myupbd}, we obtain
\begin{equation}\label{eqn:tfloorine}
\begin{aligned}
\left(\frac{1}{4C_{1,H}} \cdot \frac{|y|^2}{t^2}+K_0\right)t &\geq \int_0^{\lfloor t\rfloor} L \left(\eta(s), \dot{\eta}(s)\right)ds+\int_{\lfloor t\rfloor}^t L \left(\eta(s), \dot{\eta}(s)\right)ds\\
& \geq \frac{1}{4C_{2,H}}\int_0^{\lfloor t \rfloor} \left|\dot{\eta}(s)\right|^2 \,ds -K_0 t\\
& \geq \lfloor t \rfloor\left(\frac{1}{C_{1,H} }\cdot \frac{|y|^2}{t^2} + 4K_0\right)-K_0 t,
\end{aligned}
\end{equation}
where the second inequality follows from \eqref{eqn:Lmquad}. Since $t>4$ and $\lfloor t \rfloor  - t > -1 $, we have $4 \lfloor t \rfloor -2t \geq 2 \lfloor t \rfloor -2 >0$, and hence
\[
\begin{aligned}
\lfloor t \rfloor\left(\frac{1}{C_{1,H} }\cdot \frac{|y|^2}{t^2} + 4K_0\right)-K_0 t & =\lfloor t \rfloor\left(\frac{1}{C_{1,H} }\cdot \frac{|y|^2}{t^2} \right) + (4 \lfloor t \rfloor -t )K_0 \\
&> \frac{t}{4}\left(\frac{1}{C_{1,H} }\cdot \frac{|y|^2}{t^2} \right) +tK_0,
\end{aligned}
\]
which contradicts \eqref{eqn:tfloorine}. Hence, \eqref{eqn:xivelupbd} holds.

Let $w \in \mathbb{Z}^d$ such that $y-w \in [0,1]^d$. Define a path $\tilde{\eta}:[0, t] \to \R^d$ by
\[
    \tilde{\eta}(s):=\left
    \{\begin{aligned}
    &w+ 4\left(\frac{1}{4}-s\right)(y-w), \qquad \qquad  \quad \,\,\, \text{if} \,\, \, 0 \leq s \leq \frac{1}{4},\\
    & \eta \left(s -\frac{1}{4}\right) +w, \qquad \qquad  \qquad \qquad \quad \text{if} \,\, \,\frac{1}{4} \leq s \leq d+\frac{1}{4},\\
    & \eta\left(2\left(s-d-\frac{1}{4}\right)+d\right)+w, \qquad  \quad \,\, \text{if} \,\,\, d+\frac{1}{4} \leq s \leq d +\frac{3}{4}\\
    & \eta \left(s+\frac{1}{4}\right)+w, \qquad \qquad\qquad \qquad \quad \text{if} \,\,\, d+\frac{3}{4} \leq s \leq t-\frac{1}{4},\\
    & y+w +4\left(s-t+\frac{1}{4}\right)\left(y-w\right), \qquad \text{if} \,\,\, t-\frac{1}{4} \leq s \leq t.\\
    \end{aligned}
    \right.
\]
Since $\tilde{\eta}(0)=y$ and $\tilde{\eta}(t)=2y$, $\tilde{\eta}$ is an admissible path for $h(0,t, y, 2y)$. Thus,
\begin{equation}\label{eqn:mltilxi}
h(0,t,y,2y) \leq \int_0^tL \left(\tilde{\eta}(s), \dot{\tilde{\eta}} (s)\right) \, ds.
\end{equation}
We decompose the integral as
\begin{equation}\label{eqn:decomptilxi}
\int_0^tL \left(\tilde{\eta}(s), \dot{\tilde{\eta}} (s)\right) \, ds= \mathrm{I}+\mathrm{II}+\mathrm{III},
\end{equation}
where
\[
\mathrm{I}:=\int_0^\frac{1}{4}L \left(\tilde{\eta}(s), \dot{\tilde{\eta}} (s)\right) \, ds+\int_{t-\frac{1}{4}}^{t}L \left(\tilde{\eta}(s), \dot{\tilde{\eta}} (s)\right) \, ds,
\]

\[
\mathrm{II}:= \int_\frac{1}{4}^{d+\frac{1}{4}} L\left(\tilde{\eta}(s), \dot{\tilde{\eta}} (s)\right) \, ds+\int_{d+\frac{3}{4}}^{t-\frac{1}{4}}L\left(\tilde{\eta}(s), \dot{\tilde{\eta}} (s)\right) \, ds, \quad \text{and} \quad \mathrm{III}:=\int_{d+\frac{1}{4}}^{d+\frac{3}{4}}L\left(\tilde{\eta}(s), \dot{\tilde{\eta}} (s)\right) \, ds.
\]
We first estimate $\mathrm{I}$. By the definition of $\tilde{\eta}$,
\begin{equation}\label{eqn:tilxiIest}
\begin{aligned}
\mathrm{I} = &\int_0^\frac{1}{4} L\left(w+ 4\left(\frac{1}{4}-s\right)(y-w), -4(y-w)\right)\, ds\\
&\qquad \qquad\qquad\qquad\qquad +\int_{t-\frac{1}{4}}^t L\left(y+w +4\left(s-t+\frac{1}{4}\right)\left(y-w\right), 4(y-w)\right)\, ds\\
\leq &\frac{1}{4}  \left(\frac{1}{4C_{1,H}}\left(4 \sqrt{d}\right)^2+K_0\right)\cdot 2,
\end{aligned}
\end{equation}
where the inequality follows from \eqref{eqn:Lmquad} and the fact that $|y-w| \leq \sqrt{d}$ since $y-w \in [0, 1]^d$.
Next, we estimate $\mathrm{II}$:
\begin{equation}\label{eqn:tilxiIIest}
\begin{aligned}
\mathrm{II} =&\int_\frac{1}{4}^{d+\frac{1}{4}} L\left(\eta \left(s -\frac{1}{4}\right) +w, \dot{\eta} \left(s -\frac{1}{4}\right)\right)\, ds +\int_{d+\frac{3}{4}}^{t-\frac{1}{4}} L\left(\eta \left(s+\frac{1}{4}\right)+w, \dot{\eta} \left(s +\frac{1}{4}\right)\right)\, ds\\
= & \int_0^d L \left(\eta(s), \dot{\eta}(s)\right) \, ds+ \int_{d+1}^t L\left(\eta \left(s\right), \dot{\eta} \left(s \right)\right)\, ds\\
=& \int_0^t  L \left(\eta(s), \dot{\eta}(s)\right) \, ds - \int_d^{d+1}L\left(\eta \left(s\right), \dot{\eta} \left(s \right)\right)\, ds\\
\leq & h(0,t,0,y) +K_0,
\end{aligned}
\end{equation}
where the second line follows from the fact that $ x \mapsto L(x , v)$ is periodic for every $v \in \R^d$, and the last inequality follows from \eqref{eqn:Lmquad}. Finally, we estimate $\mathrm{III}$:
\begin{equation}\label{eqn:tilxiIIIest}
\begin{aligned}
\mathrm{III} = &\int_{d+\frac{1}{4}}^{d+\frac{3}{4}}L\left(\eta\left(2\left(s-d-\frac{1}{4}\right)+d\right)+w, 2\dot{\eta}\left(2\left(s-d-\frac{1}{4}\right)+d\right)\right) \, ds\\
= & \frac{1}{2}\int_d^{d+1} L\left(\eta\left(s\right), 2\dot{\eta}\left(s\right)\right) \, ds \leq  \frac{1}{2C_{1,H}} \int_d^{d+1} \left|\dot{\eta}(s)\right|^2 \,ds+\frac{K_0}{2}\\
\leq & \frac{2C_{2, H}}{\left(C_{1, H}\right)^2} \cdot \frac{|y|^2}{t^2}+ \left(\frac{8C_{2,H}}{C_{1,H}}+\frac{1}{2}\right) K_0,
\end{aligned}
\end{equation}
where the last inequality follows from \eqref{eqn:xivelupbd}. Now combining \eqref{eqn:mltilxi}, \eqref{eqn:decomptilxi}, \eqref{eqn:tilxiIest}, \eqref{eqn:tilxiIIest}, and \eqref{eqn:tilxiIIIest}, we have
\[
h(0,t,y,2y) \leq h(0,t,0,y)+ \frac{2C_{2,H}}{(C_{1,H})^2}\cdot \frac{|y|^2}{t^2}+\left(\frac{8C_{2,H}}{C_{1,H}}+2\right) K_0 + \frac{2d}{C_{1, H}}.
\]
This completes the proof of part (1).

Now we prove part (3). By part (1) and Fekete's lemma, $\lim_{\ep \to 0^+} \ep m\left(\frac{t_1}{\ep}, \frac{t_2}{\ep}, \frac{x}{\ep},\frac{y}{\ep}\right)$ exists. For simplicity, set $C':=C\left(1+
\frac{|y-x|^2}{(t_2-t_1)^2}\right)$. By subadditivity, we have
\[h(2t_1,2t_2, 2x, 2y)+C' \leq 2 (h(t_1, t_2, x, y)+C').\]
Iterating this inequality $k$ times yields
\[h \left(2^k t_1,2^k t_2, 2^k x, 2^k y\right) +C' \leq 2^k \left(h(t_1,t_2, x, y)+C'\right).\]
Dividing both sides by $2^k$ and letting $k \to \infty$, we obtain
\[\bar{h}(t_1,t_2, x, y) \leq h(t_1,t_2, x, y)+C'.\]
Replacing $(t_1,t_2, x, y)$ by $\left(\frac{t_1}{\e}, \frac{t_2}{\e},\frac{x}{\e}, \frac{y}{\e}\right)$ and multiplying by $\ep$ gives
\[\e \bar{h} \left(\frac{t_1}{\e},\frac{t_2}{\e}, \frac{x}{\e}, \frac{y}{\e}\right) \leq \e h \left(\frac{t_1}{\e},\frac{t_2}{\e}, \frac{x}{\e}, \frac{y}{\e}\right)+\e C'.\]
Note that the left-hand side equals $\bar{h}(t_1,t_2, x, y)$ and hence,
\[\bar{h}(t_1,t_2, x, y)-\e h\left(\frac{t_1}{\e},\frac{t_2}{\e}, \frac{x}{\e}, \frac{y}{\e}\right)\leq C\left(1+
\frac{|y-x|^2}{(t_2-t_1)^2}\right)\ep.\]
The lower bound
\[\bar{h}(t_1,t_2, x, y)-\e h\left(\frac{t_1}{\e},\frac{t_2}{\e}, \frac{x}{\e}, \frac{y}{\e}\right)\geq -C\left(1+
\frac{|y-x|^2}{(t_2-t_1)^2}\right)\ep\]
follows analogously from superadditivity.
\end{proof}

\section{Proof of Corollary \ref{aeequality}}\label{Corapp}

\begin{proof}[Proof of Corollary \ref{aeequality}]
   We first prove (1). We observe that for any $\alpha\in \widetilde{\mathcal{A}}$ and any $s\in [t,T]$, $X_s=X_t+\int_t^s \alpha_r dr$ is $\sigma(X_t)$-measurable, hence $\mathcal{L}(X_T)$ is some pushforward if $m=\mathcal{L}(X_t)$. The techniques of Proposition \ref{prop:UeqU2} yield
    \[
\widetilde{U}^\e(t,m)=\inf_{\nu\in \mathcal{P}_2}\left\{ \inf_{f: f_\#m=\nu}\left\{\int h^\e(t,T,x,f(x))m(dx) \right\}  +G(\nu)  \right\},
\]
which is associated with the Monge formulation of the optimal transport problem with cost function $h^\e(t,T,\cdot,\cdot)\colon \R^d\times\R^d\to \R$. Recall $U_1^\e$ introduced at the beginning of section \ref{altcontrset}. By Lemma \ref{lem:Uep1costh}, we have
\begin{align}
U_1^\ep(t,m)
&\ge \inf_{\nu\in \mathcal{P}_2}\left\{ \inf_{(X, Y) \in \mathcal{E}_{m,\nu}}\mathbb{E}[h^\e(t,T,X,Y)]+ G(\nu) \right\}\nonumber\\
&= \inf_{\nu\in \mathcal{P}_2}\left\{ \inf_{\pi\in \Pi(m,\nu)}\iint h^\e(t,T,x,y)\pi(dx,dy)+ G(\nu) \right\},\nonumber
\end{align}
which is associated with the Kantorovich formulation of the optimal transport problem. Since, by Proposition \ref{prop:U1eqUeqU2}, $U^\ep_1 =U^\e \leq \widetilde{U}^\ep$, we derive
\be\label{ineqs2}
\begin{split}
&\inf_{\nu\in \mathcal{P}_2}\left\{ \inf_{\pi\in \Pi(m,\nu)}\iint h^\e(t,T,x,y)\pi(dx,dy)+ G(\nu) \right\}\le U^\e(t,m)\\
&\hspace{5cm}\le \inf_{\nu\in \mathcal{P}_2}\left\{ \inf_{f: f_\#m=\nu}\left\{\int h^\e(t,T,x,f(x))m(dx) \right\}  +G(\nu)  \right\}.
\end{split}
\ee
Since $h^\e(t,T,x,y)$ is continuous, lower bounded and has quadratic growth with respect to $x,y$ (see Propositions \ref{prop:mmeas} and \ref{prop:subsuperm}), we can apply the result of Lacker-Beiglbock \cite[Proposition 2.11]{beiglbock2018denseness} or Prateli \cite[Theorem B]{pratelli2007equality} to derive that the Monge-Kantorovich optimal transport problems are equal and hence \eqref{ineqs2} holds as chains of equalities if $m$ is non-atomic.

Now we first prove (2) using proof by contradiction. Assume $U^\varepsilon(t,\delta_{x_0})<\widetilde{U}^\varepsilon(t,\delta_{x_0})$ for some $x_0\in \R^d$. Recall the definiton of $U^\varepsilon_2$ in section \ref{altcontrset}. By Proposition \ref{prop:U1eqUeqU2}, we have $U^\varepsilon_2(t,\delta_{x_0})=U^\varepsilon(t,\delta_{x_0})<\widetilde{U}^\varepsilon(t,\delta_{x_0})$. According to the definition of $\mathcal{A}_2$, there exists $\alpha\in \mathcal{A}_2$ such that
\begin{equation}\label{eqn:E_small}
\mathbb{E}\left[ \int_t^T L\left(\frac{X_s}{\varepsilon}, \alpha\left(s,\xi(\omega)\right)\right) ds + \int_{\mathbb{R}^d}g(x)\mathrm{d}m_T(\mathrm{d}x)\right] < \widetilde{U}^\varepsilon(t,\delta_{x_0}),
\end{equation}
We define $X_s(\omega)$ as the solution to the ODE \[dX_s(\omega)=\alpha\left(s,\xi(\omega)\right)ds,\quad s\in [t,T],\quad X_t(\omega)=x_0.\]
\eqref{eqn:E_small} implies that
\[
\mathbb{E}_{\omega}\left[ \int_t^T L\left(\frac{X_s(\omega)}{\varepsilon}, \alpha\left(s,\xi(\omega)\right)\right) ds + g(X_T(\omega))\right] < \widetilde{U}^\varepsilon(t,\delta_{x_0})
\]
Thus, there exists $\omega_0\in \Omega$ such that
\[X_t(\omega_0)=x_0,\quad \int_t^T L\left(\frac{X_s(\omega_0)}{\varepsilon}, \alpha\left(s,\xi(\omega_0)\right)\right) ds + g(X_T(\omega_0)) < \widetilde{U}^\varepsilon(t,\delta_{x_0}).\]

Now, we define $\tilde{\alpha}(s):=\alpha\left(s,\xi(\omega_0)\right)$ for all $s\in [t,T]$. We observe that $\tilde{\alpha}\in \widetilde{\mathcal{A}}$. Let $\tilde{X}_s$ be the solution to the ODE with control $\tilde{\alpha}$, that is, $d\tilde{X}_s=\tilde{\alpha}(s)ds$ for $s\in [t,T]$ and $\tilde{X}_t=x_0$. Then,
\[
\int_t^T L\left(\frac{X_s(\omega_0)}{\varepsilon}, \alpha\left(s,\xi(\omega_0)\right)\right) ds + g(X_T(\omega_0)) = \int_t^T L\left(\frac{\tilde X_s}{\varepsilon}, \tilde{\alpha}(s)\right) ds + g(X_T) < \widetilde{U}^\varepsilon(t,\delta_{x_0})
\]
This is a contradiction to the definition of $\widetilde{U}^\varepsilon(t,\delta_{x_0})$. Hence, $U^\varepsilon(t,\delta_{x_0})=\widetilde{U}^\varepsilon(t,\delta_{x_0})$. We conclude the proof.
\end{proof}

\nd
The equality $U^\e(t,m)=\widetilde{U}^\e(t,m)$ does not necessarily hold if $m$ has atoms (see \cite[Corollary 6.10]{CecchinDaudinJacksonMartini2025}). This hints at the fact that $\widetilde{U}^\e(t,\cdot)$ might be discontinuous at some non-atomic measures. We now present such an example.
    \vspace{2mm}

    \nd\textbf{Example of discontinuous $\widetilde{U}^\e$.} Consider $d=1$, $t=0,\; T=1$, $G(\nu)={\bf d}_1(\nu,\text{Leb}_{|[0,1]})$, $L(x,v)=|v|^2/2$ and $m=\delta_0$.
    \vspace{1mm}

    \nd
    Since the dynamics from $\widetilde{\mathcal{A}}$ are deterministic given $X_0$, we have that $(X_s)_{s\in [0,1]}$ is a deterministic curve in $\R$ starting from $0$ and we compute
    $\widetilde{U}^\e(0,\delta_0)=\inf_{y\in \R}\left\{\frac{|y|^2}{2} +\int_0^1|y-x|dx  \right\}=\frac{1}{3}.$
    We consider the sequence of measures
    $$m_n=\frac{1}{n}\sum_{k=1}^n\delta_{k/n^2}\xrightarrow{{\bf d_2}}\delta_0$$
    and the control $\alpha^n\in\widetilde{\mathcal{A}}$ with $\alpha^n(s,\frac{k}{n^2})=\frac{k}{n}-\frac{k}{n^2}$. We observe that $X_1=\frac{k}{n}$ with probability $\frac{1}{n}$, so $\mathcal{L}(X_1^{\alpha^n})=\frac{1}{n}\sum_{k=1}^n\delta_{k/n}$. We have
    \begin{align*}
    \widetilde{U}^\e(0,m_n)\le \mathbb{E}\left[ \int_0^1\frac{|\alpha^n_s|^2}{2}ds +G(\mathcal{L}(X_1^{\alpha^n}))\right]=\frac{1}{2n}\sum_{k=1}^n\left(\frac{k}{n}-\frac{k}{n^2}\right)^2+{\bf d}_1(\frac{1}{n}\sum_{k=1}^n\delta_{k/n},\text{Leb}_{|[0,1]}).
    \end{align*}
    The last expression converges as $n\to +\infty$ to $\frac{1}{2}\int_0^1x^2dx=\frac{1}{6}<\frac{1}{3}$. Hence, $\limsup_n \widetilde{U}^\e(0,m_n)<\widetilde{U}^\e(0,\delta_0)$.

\section{Technical proofs}\label{tech}

\begin{proof}[Proof of Proposition \ref{Nregularity}]
The second estimate \eqref{notuniform} follows as in \cite[Theorem 3.3]{gangbo2021finite} (note that we may assume that $\|\bm x-\bm y\|_{2-r}/N^{\frac{1}{2-r}}=(\sum_{i=1}^N(x_i-y_i)^{2-r})^{\frac{1}{2-r}}/N^{\frac{1}{2-r}}={\bf d}_{2-r}(m^N_{\bm x},m^N_{\bm y})$ because $V^N_\e(t,\bm x)$ is symmetric in $x_i,x_j$ since $G(m^N_{\bm x})$ is symmetric). To prove \eqref{epsilonunifinorm} we argue as follows.
\vspace{1mm}

\nd
Since $G$ is ${\bf d}_2$-Lipschitz, by  \cite[Theorem 4.2]{cosso2023smooth}, we may find a sequence of functions $G_{k}\colon\mathcal{P}_2\to \R$, $k\in\mathbb{N}$, which are smooth and ${\bf d}_2$-Lispchitz continuous with the same Lipschitz constant as $G$ such that
$$\lim_{k\to \infty}\sup_{m\in B}|G_k(m)-G(m)|=0,$$
for every compact set $B\subset\mathcal{P}_2$. We will show the desired estimate when the terminal condition in \eqref{HJB2} is $V^N_{\e,k}(T,\bm x)=G_k(m^N_{\bm x})$ for any $k\in\mathbb{N}$ and then the result will follow by sending $k\to \infty$ due to the stability of viscosity solutions.
Let $m\in \mathcal{P}_2$, $\phi:\R^d\to \R^d$ measurable with at most linear growth and set $m_t:=(\text{Id}+t\phi)_\#m$, for $t>0$. We have
\begin{align*}
    \left|\frac{G_k(m_t)-G_k(m)}{t}\right|\le \frac{C_G}{t}{\bf d}_2(m_t,m)\le \frac{C_G}{t}\|t\phi\|_{L^2_m}=C_G\|\phi\|_{L^2_m}.
\end{align*}
Sending $t\to 0^+$ yields
$$\left|\int_{\R^d} D_mG_k(m,x)\cdot \phi(x) m(dx)\right|\le C_G\|\phi\|_{L^2_m}\stackrel{\phi\text{ arbitrary}}{\implies} \int_{\R^d}|D_mG_k(m,x)|^2m(dx)\le C_G.$$
Therefore, if $m=m^N_{\bm x}$ for some $\bm x\in (\R^d)^N$, we derive
$$\frac{1}{N}\sum_{i=1}^N|D_mG(m^N_{\bm x},x^i)|^2\le C_G.$$

\nd
For convenience, we reverse the time and we show the result for the unique solution $V_{\e,k}^N\colon [0,+\infty)\times (\R^d)^N\to \R$ of the PDE
\be\label{HJBNrevtime}
\begin{cases}
    \partial_tV^N_{\e,k}+\frac{1}{N}\sum_{i=1}^NH(\frac{x^i}{\e},x^i,ND_{x^i}V^N_{\e,k},m_{\bm x}^N)=0, &(t,\bm x)\in (0,+\infty)\times (\R^d)^N,\\
    V^N_{\e,k}(0,\bm x)= G_k(m_{\bm x}^N), & \bm x\in (\R^d)^N.
\end{cases}
\ee
Note that, in view of \eqref{quadgrowth}, for all $\e>0$ and $\bm x\in (\R^d)^N$ we have
\begin{align}
    \left|\frac{1}{N}\sum_{i=1}^NH\left(\frac{x^i}{\e},D_mG(m^N_{\bm x},x^i),m_{\bm x}^N\right)\right|\le C_H\frac{N+\sum_{i=1}^N\left|D_mG_k(m^N_{\bm x},x^i)\right|^2}{N}\le 1+C_G.
\end{align}
This implies that $V^{N,+}_{\e,k}(t,\bm x)=Ct+G_k(m^N_{\bm x})$ and $V^{N,-}_{\e,k}(t,\bm x)=-Ct+G_k(m^N_{\bm x})$ are a classical supersolution and a classical subsolution of \eqref{HJBNrevtime}, for some constant $C>0$ large enough and independent of $N,\e,k$, respectively. By comparison principle, we get
$$-Ct+G_k(m^N_{\bm x})\le V^N_{\e,k}(t,\bm x)\le Ct+G_k(m^N_{\bm x}),\;\text{for all }(t,\bm x)\in (0,+\infty)\times (\R^d)^N,$$
which gives the Lipschitz regularity in time at $t=0$. To get the Lipschitz regularity in time for any $s>0$, we notice that $(t,\bm x)\mapsto V^N_{\e,k}(s+t,\bm x)$ is also a viscosity solution of \eqref{HJBNrevtime} with initial condition $V^{N}_{\e,k}(s,\bm x)$. Then, in light of the inequality
$$V^N_{\e,k}(s,\bm x)-\|G_k(m^N_{\cdot})-V^N_{\e,k}(s,\cdot)\|_{\infty}\le G_k(m^N_{\bm x})\le V^N_{\e,k}(s,\bm x)+\|G_k(m^N_{\cdot})-V^N_{\e,k}(s,\cdot)\|_{\infty}$$
and the comparison principle, we get
$$V^N_{\e,k}(t+s,\bm x)-\|G_k(m^N_{\cdot})-V^N_{\e,k}(s,\cdot)\|_{\infty}\le V^N_{\e,k}(t,\bm x)\le V^N_{\e,k}(t+s,\bm x)+\|G_k(m^N_{\cdot})-V^N_{\e,k}(s,\cdot)\|_{\infty}$$
or
$$\left|V^N_{\e,k}(t+s,\bm x)- V^N_{\e,k}(t,\bm x)   \right|\le \|G_k(m^N_{\cdot})-V^N_{\e,k}(s,\cdot)\|_{\infty}\le Cs,$$
for any $(t,\bm x)\in [0,+\infty)\times (\R^d)^N$.
\vspace{2mm}

\nd
We further know that $V^N_{\e,k}$ is Lipschitz continuous in both variables (see \cite[Theorem 1.34]{tran_hamilton-jacobi_2021}), therefore $V^N_{\e,k}$ satisfies \eqref{HJBNrevtime} a.e. Thus, we have from the Lipschitz continuity with respect to $t$
\begin{align*}
    C\ge |\partial_tV^N_{\e,k}(t,\bm x)|\ge \frac{1}{N}\sum_{i=1}^NH\left(\frac{x^i}{\e},x^i,ND_{x^i}V^N_{\e,k},m_{\bm x}^N\right)&\ge \frac{c_H}{N}\left(-N+\sum_{i=1}^N|ND_{x^i}V^N_{\e,k}(t,\bm x)|^{2}\right)\\
    &\ge -c_H +c_HN\sum_{i=1}^N|D_{x^i}V^N_{\e,k}(t,\bm x)|^{2},\;\;\text{ a.e.}
\end{align*}
This gives a bound of the form $C_0\ge N\sum_{i=1}^N|D_{x^i}V^N_{\e,k}|^{2}$ almost everywhere. Now let $\bm x,\bm y\in (\R^d)^N$ and $t\ge 0$. Since $V^N_{\e,k}$ is symmetrical in the $\bm x$ variables, we may assume that ${\bf d}_{2}^2(m^N_{\bm x}, m^N_{\bm y})=N^{-1}\sum_{i=1}^N|x^i-y^i|^2=N^{-1}\|\bm x-\bm y\|_2^2$. We have
\begin{align*}
    |V^N_{\e,k}(t,\bm x)-V^N_{\e,k}(t,\bm y)|= \int_0^1\frac{\text{d}}{\text{d}s} V^N_{\e,k}(t, s\bm x+(1-s)&\bm y)ds\le \esssup_{\bm z\in(\R^d)^N}\;\sum_{i=1}^N|D_{x^i}V^N(t,\bm z)|\;|x^i-y^i|\\
    &\le \esssup_{\bm z\in(\R^d)^N}\;\left(\sum_{i=1}^N |D_{x^i}V^N(t,\bm z)|^{2}  \right)^{\frac{1}{2}}\|\bm x-\bm y\|_2\\
    &\le C_0^{\frac{1}{2}}N^{-\frac{1}{2}}\|\bm x-\bm y\|_2=C_0^{\frac{1}{2}}{\bf d}_2(m^N_{\bm x},m^N_{\bm y}).
\end{align*}
The proof is complete.
\end{proof}

\begin{proof}[Proof of Corollary \ref{cor:lipschitzestimate}]
    The uniform convergence and the viscosity property were explained in the discussion preceding the statement of this Corollary. We now prove \eqref{uniforminepsilon}.
    \vspace{1mm}

    \nd
    Fix $t,s\in [0,T]$ and $m_1,m_2\in \mathcal{P}_2$. Let $(\widetilde{\Omega},\widetilde{\mathcal{F}},\widetilde{\mathbb{P}})$ be a probability space supporting two sequence of independent and identically distributed random variables $(X_k)_{k\in\mathbb{N}}\in L^2$ and $(Y_k)_{k\in\mathbb{N}}\in L^2$ such that $\mathcal{L}(X_k)=m_1$ and $\mathcal{L}(Y_k)=m_2$, for all $k\in\mathbb{N}$. We consider the empirical random measures
    $$m^N_1=\frac{1}{N}\sum_{k=1}^N\delta_{X_k}\quad\text{and}\quad m^N_2=\frac{1}{N}\sum_{k=1}^N\delta_{Y_k}.$$
    By virtue of the Glivenko-Cantelli law of large numbers (see the footnote below for a proof\footnote{The proof follows that of \cite[Theorem 1.1]{achdou2019mean}, which shows the convergence ${\bf d}_1(m_1^N,m_1)\to 0$ assuming that $X_k\in L^1$ for all $k\in \mathbb{N}$. By \cite[Remark 7.1.11]{ambrosio2005gradient}, it suffices to show $\int \varphi(x) m_1^N(dx)\xrightarrow{N\to +\infty}\int\varphi(x)m_1(dx)$ for all bounded and continuous $\varphi\colon \R^d\to \R$ and $\int|x|^2m_1^N(dx)\xrightarrow{N\to +\infty} \int |x|^2m_1(dx)$, almost surely. The first limit holds almost surely for each $\varphi$ by the law of large numbers. By the discussion in the beginning of \cite[Section 5.1]{ambrosio2005gradient} and a separability argument, we may select the set of measure zero can be chosen independently of $\varphi$. The second limit holds almost surely by the law of large numbers since $X_k\in L^2$, hence by throwing away a set of measure zero, both limits hold almost surely.}), we know that
    \be\label{empiricalconvergence}
    {\bf d}_{2}(m_1^N,m_1)+{\bf d}_{2}(m_2^N,m_2)\xrightarrow{N\to +\infty}0,\;\;\text{almost surely}
    \ee
   in the probability space $(\widetilde{\Omega},\widetilde{\mathcal{F}},\widetilde{\mathbb{P}})$. We write, due to the triangle inequality,
    \be\label{triangle}
    \begin{split}
        \left|U^\e(t,m_1)-U^\e(s,m_2)\right|\le& \left| U^\e(t,m_1)-U_\e^N(t,m_1^N)\right|\\
        &\qquad+\left|U_\e^N(t,m_1^N)-U^N_\e(s,m_2^N)  \right| +\left|U^N_\e(s,m_2^N)-U^\e(s,m_2) \right|.
    \end{split}
    \ee
    We observe by the ${\bf d}_{2-r}$-Lipschitz continuity of $U^N_\e$
    \begin{align}\label{firstterm}
 \left| U^\e(t,m_1)-U_\e^N(t,m_1^N)\right| &\le \left|U^\e(t,m_1)-U^N_\e(t,m_1) \right|+ \left|U_\e^N(t,m_1)-U^N_\e(t,m_1^N) \right|\nonumber\\
&\le \left|U^\e(t,m_1)-U^N_\e(t,m_1) \right|+ C{\bf d}_{2-r}(m_1,m_1^N).
    \end{align}
    Similarly,
        \begin{align}\label{secondterm}
 \left| U^\e(s,m_2^N)-U_\e^N(s,m_2^N)\right| \le \left|U^\e(s,m_2)-U^N_\e(s,m_2) \right|+ C{\bf d}_{2-r}(m_2,m_2^N).
    \end{align}
    By the pointwise convergence $U^N_\e\xrightarrow{N\to +\infty}U^\e$, the bounded convergence theorem and \eqref{empiricalconvergence}, inequalities \eqref{firstterm} and \eqref{secondterm} imply
    \be\label{firstthrird}
 \left| U^\e(t,m_1)-U_\e^N(t,m_1^N)\right| + \left| U^\e(s,m_2^N)-U_\e^N(s,m_2^N)\right| \xrightarrow{N\to +\infty}0,
    \ee
    since convergence in ${\bf d}_2$ gives convergence in ${\bf d}_{2-r}$. Finally, by the definition of $U^N_\e$ in \eqref{extension} and \eqref{epsilonunifinorm} we have
\begin{align}\label{secondintriangle}
\left|U_\e^N(t,m_1^N)-U^N_\e(s,m_2^N)\right|&=\left|V_\e^N(t,X_1,\ldots,X_N)-V^N_\e(s,Y_1,\ldots,Y_N)\right|\nonumber\\
&\le C_0\left( |t-s|+{\bf d}_2(m_1^N,m_2^N)  \right)\nonumber\\
&\qquad\qquad\qquad\xrightarrow{N\to \infty}C_0\left( |t-s|+{\bf d}_2(m_1,m_2) \right),
\end{align}
where the convergence is almost sure. Sending $N\to +\infty$ in \eqref{triangle} , the desired \eqref{uniforminepsilon} follows from \eqref{firstthrird} and \eqref{secondintriangle}.
\end{proof}

\section{Measurable selection}\label{app:measselect}

\begin{lemma}\label{lem:measapp}
Assume (A1) and (A2). Suppose that \(0\leq t_1<t_2\leq T\) and fix \(\varepsilon>0\). For any $w,y,z \in \R^d$, $\mu \in \mathcal{P}_2$, and sufficiently large $n\in \mathbb{N}$, there exists a measurable path $\gamma_n:[t_1,t_2]\times\mathbb{R}^d \times\mathbb{R}^d \times\mathbb{R}^d\to\mathbb{R}^d$ such that
\[
\forall \, w,y,z \in \R^d,\quad \gamma(\cdot, w,y,z) \in \mathrm{AC}([t_1,t_2];\R^d),\quad \gamma_n(t_1,w,y,z)=y, \quad \gamma_n(t_2,w,y,z)=z,
\]
and
\[
\int_{t_1}^{t_2}
L \left(\frac{\gamma_n(t,w,y,z)}{\varepsilon}, w, \dot{\gamma}_n(t,w,y,z),\mu\right)\,dt
<
h^\e\left(w,\mu;t_1, t_2, y, z\right)+\frac{C}{\sqrt{n}}\left(1+|y-z|^2\right),
\]
where $C=C(\ep,C_{1,H}, C_{2, H}, C_H, K_0, t_2-t_1)>0$.
\end{lemma}

\begin{proof}
Let $n \in \mathbb{N}$ and consider the lattice $\frac{1}{n}\mathbb{Z}^d$. For each $i,j,k \in \frac{1}{n}\mathbb{Z}^d$, let $\gamma_{\left(i,j,k\right)} \in {\rm AC} \left([t_1,t_2];\mathbb{R}^d\right)$ satisfy
\[
\gamma_{\left(i,j,k\right)} \in \operatorname*{argmin}_{\substack{\gamma\in \text{AC}\left([t_1,t_2];\mathbb{R}^d\right)\\
\gamma(t_1)=j,\gamma(t_2)=k}}\int^{t_2}_{t_1}L\left(\frac{\gamma(t)}{\ep}, i, \dot{\gamma}(t),\mu\right)dt\,,
\]
that is,
\[
h^\ep\left(i,\mu;t_1, t_2, j, k\right)=\int_{t_1}^{t_2}
L \left(\frac{\gamma_{\left(i,j,k\right)}(t)}{\varepsilon},i,\dot{\gamma}_{\left(i,j,k\right)}(t),\mu\right)\,dt,\quad\
\gamma_{\left(i,j,k\right)}(t_1)=j,\quad \gamma_{\left(i,j,k\right)}(t_2)=k.
\]
Define $\phi_n:\R^d \to \R^d$ by
\[
\phi_n(x):=\frac{1}{n} \left(\lfloor nx_1\rfloor, \lfloor nx_2\rfloor, \cdots, \lfloor nx_d\rfloor\right).
\]
Then, $\left|\phi_n(y)-y\right|\leq \frac{1}{n}$ and
\begin{equation}\label{eqn:phiest}
\left|\phi_n(y)-\phi_n\left(z\right)\right|^2 \leq \left(\left|\phi_n(y)-y\right|+\left|y-z\right|+\left|z-\phi_n\left(z\right)\right|\right)^2 \leq \frac{8}{n^2} +2\left|y-z\right|^2.
\end{equation}
Define $\tilde{\gamma}:[t_1,t_2]\times\R^d \times \R^d \times \R^d \to \R^d$ by
\[\tilde{\gamma}(t,w,y,z):=\gamma_{\left(\phi_n(w), \phi_n(y),\phi_n(z)\right)}(t).\]
For each $(w,y,z)$, the map $t \mapsto \tilde{\gamma}(t,w,y,z)$ is absolutely continuous. For each $t \in [t_1,t_2]$, $(w,y,z) \mapsto \tilde{\gamma}(t,w,y,z)$ is measurable since it is piecewise continuous. Hence $\tilde{\gamma}$ is a Carath\'eodory function, thus jointly measurable in $(t,w,y,z)$ by \cite[Lemma 4.51]{aliprantis2006infinite}. Since the straight line connecting $\phi_n(y)$ and $\phi_n(z)$
\[
\eta(t):=\phi_n(y)+\frac{t-t_1}{t_2-t_1} \left(\phi_n(z)-\phi_n(y)\right)
\]
is an admissible path for $m^\ep(\phi_n(w),\mu;t_1,t_2,\phi_n(y),\phi_n(z))$. By \eqref{eqn:Lmquad},
\[
\begin{aligned}
-K_0(t_2-t_1)+\frac{1}{C_{2,H}}\int_{t_1}^{t_2} \left|\dot{\tilde{\gamma}}(t,w,y,z)\right|^2\, dt &\leq h^\ep(\phi_n(w),\mu;t_1,t_2,\phi_n(y),\phi_n(z))\\
&\leq \frac{1}{C_{1,H}}\cdot \frac{\left|\phi_n(y)-\phi_n(z)\right|^2}{(t_2-t_1)}+K_0(t_2-t_1)
\end{aligned}
\]
Hence,
\begin{equation} \label{eqn:gammasbd}
      \int_{t_1}^{t_2} \left|\dot{\tilde{\gamma}}(t,w,y,z)\right|^2\, dt\leq C(1+|y-z|^2),
\end{equation}
for some constant $ C:=C(C_{1,H}, C_{2, H}, K_0,t_2-t_1)>0$. We claim that for $n$ large enough, there exists \[
d\in \left\{t_1, t_1+\frac{1}{n}, t_1+\frac{2}{n}, \cdots, t_1+\frac{1}{n} \left(\left\lfloor n\left(t_2-t_1\right) \right\rfloor-1\right)\right\}
\]
such that
\begin{equation}\label{eqn:gammadotnbd}
   \int_{d}^{d+\frac{1}{n}} \left|\dot{\tilde{\gamma}}(t,w,y,z)\right|^2\,dt\leq  \frac{C}{\sqrt{n}} \left(1+|y-z|^2\right).
\end{equation}
If not,
\[
    \int_{t_1}^{t_1+\frac{1}{n}\left\lfloor n\left(t_2-t_1\right) \right\rfloor} \left|\dot{\tilde{\gamma}}(t,w,y,z)\right|^2\,dt\geq \frac{\left\lfloor n\left(t_2-t_1\right) \right\rfloor C}{\sqrt{n}} \left(1+|y-z|^2\right)\geq 4C\left(1+|y-z|^2\right),
\]
for $n$ large enough, which contradicts \eqref{eqn:gammasbd}.

Define $\gamma_n:[t_1, t_2] \times \R^d  \times \R^d \times \R^d\to \R^d$ by
\[
   \gamma_n(t,w,y,z) = \left\{\begin{aligned}
   & y+4n(t-t_1)(\phi_n(y)-y), \quad  \text{ if } t\in \left[t_1, t_1+\frac{1}{4n}\right], \\
   &\tilde{\gamma}\left(t-\frac{1}{4n},w,y,z\right), \qquad  \text{ if } t \in \left[t_1+\frac{1}{4n},d+\frac{1}{4n}\right],\\
   &\tilde{\gamma}\left(
2\left(t-d -\frac{1}{4n}\right) +d,w,y,z\right), \qquad  \text{ if } t \in \left[d+\frac{1}{4n},d+\frac{3}{4n}\right],\\
   &\tilde{\gamma} \left(t +\frac{1}{4n},w,y,z\right), \qquad  \text{ if } t \in \left[d+\frac{3}{4n},t_2-\frac{1}{4n}\right],\\
   &\phi_n(z)+4n\left(t-t_2+\frac{1}{4n}\right)(z-\phi_n(z)), \quad  \text{ if } t\in \left[t_2-\frac{1}{4n}, t_2\right].
   \end{aligned}
   \right.
\]
Here, we accelerate the trajectory $\tilde{\gamma}$ over the time period $\left[d, d+\frac{1}{n}\right]$ to save a total time of $\frac{1}{2n}$, and we use a time interval of length $\frac{1}{4n}$ to connect $y$ and $\phi_n(y)$, and similarly to connect $z$ and $\phi_n(z)$. Note that $\gamma_n(t_1, w,y,z)=y, \gamma_n(t_2,w,y,z)=z$. Moreover, for each $t$, $(w,y,z)\mapsto \gamma_n(t,w,y,z)$ is measurable. For each $(w,y,z)$, $t \mapsto \gamma_n(t,w,y,z)$ is absolutely continuous. Hence, $ \gamma_n$ is a Carath\'eodory function and is measuarable. Then,
\[
\begin{aligned}
\int_{t_1}^{t_2}
L \left(\frac{\gamma_n(t,w,y,z)}{\varepsilon},w, \dot{\gamma}_n(t, w,y,z), \mu\right)\,dt=\mathrm{I}+\mathrm{II}+\mathrm{III},
\end{aligned}
\]
where
\[
\mathrm{I}:=\int_{t_1}^{t_1+\frac{1}{4n}}
L \left(\frac{\gamma_n}{\varepsilon},w,4n(\phi_n(y)-y),\mu\right)\,dt+
\int_{t_2-\frac{1}{4n}}^{t_2}
L \left(\frac{\gamma_n}{\varepsilon},w,4n(z-\phi_n(z)),\mu\right)\,dt,
\]

\[
\mathrm{II}:=\int_{[t_1, t_2]\setminus[d, d+\frac{1}{n}]}
L \left(\frac{\tilde{\gamma}}{\e},w, \dot{\tilde{\gamma}},\mu\right)\,dt, \qquad
\mathrm{III}: =\frac{1}{2}\int_{d}^{d+\frac{1}{n}}
L \left(\frac{\tilde{\gamma}}{\e},w,2\dot{\tilde{\gamma}},\mu\right)\,dt.
\]
Using \eqref{eqn:Lmquad} and \eqref{eqn:gammadotnbd}, we obtain
\begin{equation}
\begin{aligned}\label{eqn:Iest}
|\mathrm{I}| & \leq \frac{1}{4C_{1, H}}\int_{t_1}^{t_1+\frac{1}{4n}}\left|4n(\phi_n(y)-y)\right|^2\,dt+ \frac{1}{4C_{1, H}}\int_{t_2-\frac{1}{4n}}^{t_2}\left|4n\left(z-\phi_n(z)\right)\right|^2\,dt +\frac{K_0}{2n}\\
&\leq \frac{1}{n}\left(\frac{2}{C_{1,H}}+\frac{K_0}{2}\right),
\end{aligned}
\end{equation}
and
\begin{equation}\label{eqn:IIIest}
|\mathrm{III}|\leq \frac{1}{2C_{1,H}} \int_{d}^{d+\frac{1}{n}}\left|\dot{\tilde{\gamma}}(t,w,y,z)\right|^2\,dt+\frac{K_0}{2n} \leq \frac{C}{\sqrt{n}} \left(1+|y-z|^2\right) + \frac{K_0}{2n},
\end{equation}
for some constant $C:=C(C_{1,H}, C_{2, H}, K_0,t_2-t_1)>0$. Write
\[
m^\ep\left(\phi_n(w), \mu;t_1, t_2, \phi_n(y), \phi_n(z)\right)= \tilde{{\rm II}}+\tilde{{\rm III}},
\]
where
\[\tilde{{\rm II}}=\int_{[t_1, t_2]\setminus[d, d+\frac{1}{n}]} L \left(\frac{\tilde{\gamma}}{\varepsilon},\phi_n(w),\dot{\tilde{\gamma}},\mu\right)\,dt, \qquad \tilde{{\rm III}}=\int_d^{d+\frac{1}{n}}L \left(\frac{\tilde{\gamma}}{\varepsilon},\phi_n(w),\dot{\tilde{\gamma}},\mu\right)\,dt.
\]
Then by \eqref{eqn:LmLip} and H\"oder's inequality,
\[
\begin{aligned}
\left|{\rm II}-{\rm \tilde{II}}\right|  &\leq  C_{L} \int_{[t_1, t_2]\setminus[d, d+\frac{1}{n}]}\left(1+2\left|\dot{\tilde{\gamma}}(t,w,y,z)\right|\right)\left|\phi_n(w)-w\right| \, dt \\
&\leq \frac{C_L(t_2-t_1)}{n}+\frac{2C_L}{n} \int_{t_1}^{t_2} \left|\dot{\tilde{\gamma}}(t,w,y,z)\right|\,dt \leq\frac{C}{n^\frac{5}{4}}\left(1+|z-y|^2\right)^\frac{1}{2}+\frac{C}{n},
\end{aligned}
\]
and by \eqref{eqn:LmLip} and \eqref{eqn:gammadotnbd},
\[
\left|\tilde{{\rm III}}\right| \leq \frac{1}{C_{1,H}}\int_{d}^{d+\frac{1}{n}}\left|\dot{\tilde{\gamma}}(t,w,y,z)\right|^2\,dt+\frac{K_0}{n} \leq \frac{C}{\sqrt{n}} \left(1+|y-z|^2\right)+\frac{K_0}{n}.
\]
Combining
\eqref{eqn:gammadotnbd}, \eqref{eqn:Iest}, and \eqref{eqn:IIIest}, for $n$ large enough, we have
\begin{equation}
\begin{aligned}\label{eqn:gammanphi}
&\left|\int_{t_1}^{t_2}
L \left(\frac{\gamma_n}{\e},w, \dot{\gamma}_n, \mu\right)\,dt- h^\ep\left(\phi_n(w), \mu;t_1, t_2, \phi_n(y), \phi_n(z)\right)\right|\\
=&\left|
\int_{t_1}^{t_2}
L \left(\frac{\gamma_n}{\varepsilon},w,\dot{\gamma}_n, \mu\right)\,dt-
\int_{t_1}^{t_2} L \left(\frac{\tilde{\gamma}}{\varepsilon},\phi_n(w),\dot{\tilde{\gamma}},\mu\right)\,dt,\right|\\
\leq & |\mathrm{I}|+ |\mathrm{III}|+\left|{\rm II}-\tilde{{\rm II}}\right|+\left|\tilde{{\rm III}}\right|\\
\leq & \frac{C}{\sqrt{n}} \left(1+|y-z|^2\right) .
\end{aligned}
\end{equation}
By \eqref{eqn:mlip} and \eqref{eqn:phiest},
\begin{equation}
\begin{aligned}\label{eqn:phifex}
&\left|h^\ep\left(\phi_n(w), \mu;t_1, t_2, \phi_n(y), \phi_n(z)\right)-h^\ep\left(w, \mu;t_1, t_2, y, z\right)\right| \\
\leq\;&
C\Bigl[\Bigl(\left|w-\phi_n(w)\right|+\left| y-\phi_n(y)\right|+\lvert z-\phi_n(z)\rvert \Bigr)\Bigr. \\
& \Bigl. \qquad  \qquad \qquad \qquad \cdot
\Bigl(
1
+\lvert y-\phi_n(y)\rvert
+\lvert z-\phi_n(z)\rvert
+\lvert y-z\rvert^2
+\lvert \phi_n(y)-\phi_n(z)\rvert^2
\Bigr)\Bigr]\\
\leq & \frac{3C}{n}\left(1+\frac{2}{n}+\frac{8}{n^2}+3\left| y-z\right|^2\right),
\end{aligned}
\end{equation}
where $C=C(\ep,C_{1,H}, C_{2, H}, C_H, K_0, t_2-t_1)>0$. Combining \eqref{eqn:gammanphi} and \eqref{eqn:phifex}, we obtain, for $n$ sufficiently large,
\[
\int_{t_1}^{t_2}
L \left(\frac{\gamma_n}{\e},w, \dot{\gamma}_n, \mu\right)\,dt \leq h^\ep\left(w, \mu;t_1, t_2, y, z\right) +   \frac{C}{\sqrt{n}}\left(1+|y-z|^2\right),
\]
where $\gamma_n:[t_1, t_2]\times \R^d \times \R^d\times \R^d\to \R^d$ is jointly measurable.
\end{proof}

\begin{remark}
   The same conclusion remains valid under assumptions (A1') and (A2), namely when $L$ and $h^\e$ are independent of $w, \mu$.
\end{remark}

\bibliographystyle{alpha}
\bibliography{ref}

\end{document}